\documentclass[11pt]{amsart}

\usepackage{amsmath, amsthm, amssymb, a4wide}

\title{The Moduli Space of Polynomial Maps and Their Fixed-Point Multipliers}
\author{Toshi Sugiyama}
\thanks{This work was done when the author was a graduate student of Kyoto University}
\thanks{Published in Adv. Math. 322C (2017) pp.132--185}
\date{\today}

\address{Nada High School, Uozaki-kita-machi 8-5-1, Higashinada, Kobe 658-0082, Japan}
\email{sugiyama.toshi@gmail.com}


\newcommand{\z}{\zeta}
\newcommand{\la}{\lambda}
\newcommand{\La}{\Lambda}
\newcommand{\mb}[1]{\mathbb{#1}}
\newcommand{\Pd}{\mb{P}^{d-2}}
\newcommand{\key}{Key Lemma}
\newcommand{\mult}{\mathrm{mult}}

\theoremstyle{plain}
\newtheorem{maintheorem}{Main Theorem}

\newtheorem*{keylemma}{Key Lemma}
\newtheorem{vartheorem}{Theorem}

\newtheorem{varproposition}[vartheorem]{Proposition}
\newtheorem{theorem}{Theorem}[section]

\newtheorem{lemma}[theorem]{Lemma}
\newtheorem{proposition}[theorem]{Proposition}
\newtheorem{conjecture}{Conjecture}
\newtheorem*{problem}{Problem}

\theoremstyle{definition}
\newtheorem{remark}[theorem]{Remark}
\newtheorem{definition}[theorem]{Definition}
\newtheorem{example}{Example}

\numberwithin{equation}{section}

\begin{document}

\begin{abstract}
 We consider the family $\mathrm{MP}_d$ of affine conjugacy classes of
 polynomial maps of one complex variable with degree $d \geq 2$, and 
 study the map $\Phi_d:\mathrm{MP}_d\to \widetilde{\Lambda}_d \subset \mathbb{C}^d / \mathfrak{S}_d$
 which maps each $f \in \mathrm{MP}_d$ to the set of fixed-point
 multipliers of $f$.
 We show that the local fiber structure of the map $\Phi_d$
 around $\bar{\lambda} \in \widetilde{\Lambda}_d$
 is completely determined by certain two sets 
 $\mathcal{I}(\lambda)$ and $\mathcal{K}(\lambda)$
 which are subsets of the power set of $\{1,2,\ldots,d \}$.
 Moreover for any $\bar{\lambda} \in \widetilde{\Lambda}_d$,
 we give an algorithm for counting the number of elements
 of each fiber $\Phi_d^{-1}\left(\bar{\lambda}\right)$ 
 only by using $\mathcal{I}(\lambda)$ and $\mathcal{K}(\lambda)$.
 It can be carried out in finitely many steps, and often by hand.
\end{abstract}

\maketitle

\section{Introduction}


Let $\mathrm{MP}_d$ be the family of affine conjugacy classes of polynomial maps of one complex variable with degree $d \geq 2$,
and $\mathbb{C}^d / \mathfrak{S}_d$ the set of unordered collections of $d$ complex numbers.
Then the aim of this paper is to give a {\it complete description} of
the fiber structure of the map
\[
    \Phi_d : \mathrm{MP}_d \to \widetilde{\Lambda}_d \subset \mathbb{C}^d / \mathfrak{S}_d
\]
which maps each $f \in \mathrm{MP}_d$
to the set of fixed-point multipliers of $f$,
except where $f \in \mathrm{MP}_d$ has multiple fixed points.
%
%
%

Since multipliers of fixed points have played a central role in the study of the complex dynamics,
%
it is natural 
to ask to what extent fixed-point multipliers of $f$
determine the original map $f$.
For polynomial maps,
since the set of fixed-point multipliers is invariant under the action of affine transformations,
the question is to count the number of
affine conjugacy classes of polynomial maps 
when the set of its fixed-point multipliers are given.
It is formulated in the following form:
how many elements there are on each fiber of the above map $\Phi_d : \mathrm{MP}_d \to \mathbb{C}^d / \mathfrak{S}_d$.
Here, since the set of fixed-point multipliers always satisfies a certain relation 
by the fixed point theorem (see Proposition~\ref{pr.1.2}),
the image of $\Phi_d$ is contained in a certain hyperplane $\widetilde{\Lambda}_d$ 
in $\mathbb{C}^d / \mathfrak{S}_d$.
Hence the main object of our study is the map
$\Phi_d : \mathrm{MP}_d \to \widetilde{\Lambda}_d.$

For $d=2$, it is easily verified that $\Phi_2$ is bijective.
In the case $d=3$, Milnor~\cite{mi_cub} showed that $\Phi_3$ is also bijective,
which was the starting point of his study of the complex dynamics of cubic polynomials.
For $d \geq 4$, Fujimura and Nishizawa have long studied the map $\Phi_d$ 
in their series of papers such as~\cite{NishizawaFujimura},~\cite{Fujimura2} and~\cite{fu}.
Especially their achievement is summarized in Fujimura's paper~\cite{fu}, which includes the following:
\begin{itemize}
 \item $\Phi_d$ is not surjective for $d\ge 4$. Moreover for $d=4$ or $5$, she found all $\bar{\la} \in \widetilde{\Lambda}_d$
	whose inverse image of $\Phi_d$ is empty.
 \item Generic fiber of $\Phi_d$ consists of $(d-2)!$ points.  Moreover if $\Phi_d^{-1}(\bar{\la})$ is finite, then 
	$\#\left( \Phi_d^{-1}(\bar{\la}) \right) \leq (d-2)!$ always holds.
 \item For $d=4$, she found $\#\left( \Phi_4^{-1}(\bar{\la}) \right)$ for all $\bar{\la} \in \widetilde{\Lambda}_4$.
\end{itemize}
Here, we denote the cardinality of a set $X$ by $\#\left( X \right)$.
Similar results for rational maps are given by Milnor in~\cite[p.152, Problem 12-d]{mi_book} and~\cite{mi_qua}.

Based on the results above,
this paper provides an algorithm for counting the number of elements of each fiber $\Phi_d^{-1}(\bar{\la})$ 
for {\it all} $\bar{\la}=\{\la_1,\dots,\la_d\} \in \widetilde{\Lambda}_d$ and for {\it all} $d\geq 4$
except when $\la_i = 1$ for some $i$.
In practice, for each $\lambda=(\lambda_1,\ldots ,\lambda_d)\in \Lambda_d \subset \mathbb{C}^d$ with $\la_i \ne 1$,
certain two subsets $\mathcal{I}(\lambda), \mathcal{K}(\lambda)$ 
of the power set of $\{1,2,\ldots,d \}$ are defined, 
and the number of elements of a fiber $\Phi_d^{-1}(\bar{\lambda})$
is completely determined by $\mathcal{I}(\lambda)$ and $\mathcal{K}(\lambda)$.
Moreover we give an algorithm for counting the number $\#\left(\Phi_d^{-1}\left(\bar{\lambda}\right)\right)$
only by using $\mathcal{I}(\lambda)$ and $\mathcal{K}(\lambda)$
(see Main Theorems~\ref{mthm.1},~\ref{mthm.3}, Definition~\ref{df.1.6} 
and Section~\ref{sec.1.1.2}).
The algorithm can be carried out in finitely many steps, and only by hand.
%
%
Moreover in Main Theorem~\ref{mthm.2}
we show that the local fiber structure of $\Phi_d$ around $\bar{\lambda}$ 
is also determined by $\mathcal{I}(\lambda)$ and $\mathcal{K}(\lambda)$.

\vspace{12pt}


We shall provide some more concerning results.

Several kinds of compactifications of $\mathrm{MP}_d$ 
have been constructed independently by 
Silverman~\cite{sil}, by DeMarco and McMullen~\cite{demarco} and
by Fujimura and Taniguchi~\cite{tani}.
Silverman's is based on the GIT compactification of the moduli space of
rational maps,
while the compactifications of DeMarco and MuMullen and of Fujimura and Taniguchi 
are both based on the consideration of the multipliers of polynomial
maps.
Especially, Fujimura and Taniguchi's compactification is strongly
related to the definition of the set $\mathcal{I}(\lambda)$ in this paper (see Remarks~\ref{rem.1.3} and~\ref{rm.2.1}).

Regarding the moduli space of rational maps, let us recall
an important result of McMullen~\cite{Mc}.
He 
investigated the map $\overline{\Psi}_d$
which maps each M\"{o}bius conjugacy class
of rational maps of $\widehat{\mathbb{C}}$ of degree $d$
to the set of the multipliers of its periodic points of all periods,
and showed that the map $\overline{\Psi}_d$ is finite-to-one
with few exceptions.
To state the result explicitly, we denote by $\mathrm{MR}_d$ the family of
M\"{o}bius conjugacy classes of rational maps of degree $d$
on the Riemann sphere $\widehat{\mathbb{C}}$,
and define the map 
$\Psi^{(n)}_d:\mathrm{MR}_d \to \mathbb{C}^{d^n+1}/\mathfrak{S}_{d^n+1}$ 
which maps each $f \in \mathrm{MR}_d$ to the set of multipliers of $n$-periodic points of $f$.
Under the above notation,
he considered the map
\[
\overline{\Psi}^N_d := \prod_{n=1}^N\Psi_d^{(n)}:
 \mathrm{MR}_d \to \prod_{n=1}^N 
 \left(\mathbb{C}^{d^n+1}/\mathfrak{S}_{d^n+1}\right).
\]
It is not hard to see that $\overline{\Psi}^1_2$ is an embedding,
and in fact maps $\mathrm{MR}_2$ isomorphically onto a hyperplane in 
$\mathbb{C}^3/\mathfrak{S}_3$ (see~\cite{mi_qua}).
However by looking at (flexible or rigid) Latt\`{e}s examples,
we can no longer expect $\overline{\Psi}^N_d$ to be an embedding for many $d$ even when $N$ is sufficiently large 
(see~\cite{mi_lattes} for Latt\`{e}s examples).
He showed that
 for sufficiently large $N$, the map $\overline{\Psi}^N_d$ is
 finite-to-one except when $d$ is a square,
 in which case it is also finite-to-one away from
 the Latt\`{e}s locus.
Here, the Latt\`{e}s locus consists of one or two points whose inverse images
are one parameter families.
Furthermore, rigid Latt\`{e}s examples imply
that for any positive integer $h$
there exist infinitely many degrees $d$
such that for any $N$, the map $\overline{\Psi}^N_d$ is at least $h$-to-one 
 (see~\cite{Mc} for more detail).
However, it does not appear to be known if $\overline{\Psi}^N_3$ is injective.
Hutz and Tepper~\cite{HutzTepper} showed that $\overline{\Psi}_3^2$ is $12$-to-one map.

There are several other papers such as~\cite{Gorbovickis1} and~\cite{Gorbovickis2},
that discuss the use of multipliers of periodic points 
to parameterize the moduli space of polynomial or rational maps.

In another direction,
Bousch~\cite{bou}, Morton~\cite{mor} and Silverman~\cite{sil} have studied
the algebraic properties of the hypersurfaces
consisting of periodic points of polynomial or rational maps
in the product space of $\widehat{\mathbb{C}}$ and some parameter space.

\vspace{12pt}

We have three main theorems in this paper.
The rest of Introduction is devoted to state Main Theorems~\ref{mthm.1}, \ref{mthm.2} and~\ref{mthm.3}.
%
%
To state them explicitly, we fix our notation first.

For $d \ge 2$, we put
\begin{equation}\label{eq.1.1}
 \mathrm{Poly}_d := \left\{f \in \mb{C}[z] \bigm| \deg f = d  \right\}
 \quad \textrm{and} \quad
 \mathrm{Aut}(\mathbb{C})
 := \left\{\gamma (z) = az+b \bigm| a,b \in \mb{C},\ a \ne 0 \right\}.
\end{equation}
Since $\gamma \in \mathrm{Aut}(\mathbb{C})$
naturally acts on $f \in \mathrm{Poly}_d$
by $\gamma \cdot f := \gamma \circ f \circ \gamma^{-1}$,
we can define its quotient $\mathrm{MP}_d := \mathrm{Poly}_d / \mathrm{Aut}(\mathbb{C})$,
which we usually call the moduli space of polynomial maps
of degree $d$.
We put $\mathrm{Fix}(f) := \{ z \in \mathbb{C} \bigm| f(z)=z\}$ for $f \in \mathrm{Poly}_d$, where $\mathrm{Fix}(f)$ is considered counted with multiplicity.
Hence we always have $\# \left(\mathrm{Fix}(f)\right) = d$.
Since the set of fixed-point multipliers $\left(f'(\z)\right)_{\z \in \mathrm{Fix}(f)}$
is invariant under the action of $\mathrm{Aut}(\mathbb{C})$,
we can naturally define the map
$\Phi_d:\mathrm{MP}_d \to \mathbb{C}^{d}/\mathfrak{S}_{d}$
by $\Phi_d(f):=\left(f'(\z)\right)_{\z \in \mathrm{Fix}(f)}$.
Here, $\mathfrak{S}_{d}$ denotes the $d$-th symmetric group 
which acts on $\mathbb{C}^{d}$ by the permutation of coordinates.

Note that a fixed point $\z \in \mathrm{Fix}(f)$
is multiple if and only if $f'(\z)=1$.

\begin{proposition}[Fixed point theorem]\label{pr.1.2}
 Let $d$ be a natural number with $d \ge 2$ and
 suppose that a polynomial map $f \in \mathrm{Poly}_d$ has no multiple fixed point.
 Then we have $\sum_{\z \in \mathrm{Fix}(f)} \frac{1}{1-f'(\z)} = 0$.
\end{proposition}

Proposition~\ref{pr.1.2} is shown by the integration
$\frac{1}{2\pi\sqrt{-1}}\oint_{|z| = R}\frac{dz}{z - f(z)}$
for sufficiently large $R$.
We put
$\Lambda_d :=\left\{(\lambda_1,\ldots ,\lambda_d)\in\mathbb{C}^d \ \left| \
  \sum_{i=1}^d \prod_{j\ne i} \left( 1- \lambda_j \right) = 0
 \right. \right\}$,
$\widetilde{\Lambda}_d := \Lambda_d / \mathfrak{S}_d$
and $\textit{pr}: \Lambda_d \to \widetilde{\Lambda}_d$.
Then the image of the map $\Phi_d$ is contained in
$\widetilde{\Lambda}_d$ by Proposition~\ref{pr.1.2} and by the fact that
$(\lambda_1,\ldots ,\lambda_d)\in\mathbb{C}^d$
always belongs to $\Lambda_d$
if at least two of $\lambda_i$ are equal to $1$.
In the following, we consider the map
\[
 \Phi_d : \mathrm{MP}_d \to \widetilde{\Lambda}_d
\]
defined by $f \mapsto \left(f'(\zeta) \right)_{\zeta \in \mathrm{Fix}(f)}$.
In the main theorems of this paper,
we restrict our attention to the map $\Phi_d$ on the domain
where polynomial maps have no multiple fixed points, i.e.,
on the domains
$V_d := \left\{ (\lambda_1,\ldots ,\lambda_d) \in \Lambda_d \bigm|
 \la_i \ne 1 \ \textrm{for any} \ 1\le i\le d \right\}$
and $\widetilde{V}_d := V_d/ \mathfrak{S}_d$,
which are Zariski open subsets
of $\La_d$ and $\widetilde{\La}_d$ respectively.
Throughout this paper, we always denote by $\bar{\la}$
the equivalence class of $\la \in \Lambda_d$ in $\widetilde{\Lambda}_d$, i.e., 
$\bar{\la}=\textit{pr}(\la)$, and never denote the complex conjugate of $\la$.

It is not hard to see that
 in the case $d=2$ or $3$, the map $\Phi_d$ is bijective.
However we can no longer expect $\Phi_d$ to be bijective if $d\ge 4$;
yet we can expect $\Phi_d$ to be generically finite by the remark below:
\begin{remark}\label{rem.1.4}
 We have
 $\mathrm{MP}_d\cong\mb{C}^{d-1}/\left(\mb{Z}/(d-1)\mb{Z}\right)$ and
 $\widetilde{\Lambda}_d\cong\mb{C}^{d-1}$.
 Especially we have
 $\dim_{\mb{C}}\mathrm{MP}_d=\dim_{\mb{C}}\widetilde{\Lambda}_d=d-1$.
\end{remark}


We now state the first main theorem in this paper.

\begin{maintheorem}\label{mthm.1}
 Let $d$ be a natural number with $d \ge 4$ and suppose
 $\la = (\la_1, \ldots, \la_d) \in V_d$.
 Then the following statements hold:
 \begin{enumerate}
  \item We always have
	$0 \le \#\left( \Phi_d^{-1}\left(\bar{\la}\right) \right) \le (d-2)!$.
	\label{en.mt1}

  \item The cardinality $\#\left( \Phi_d^{-1}\left(\bar{\la}\right) \right)$ is a function of
	 the two sets
	\begin{align*}
	 \mathcal{I}(\lambda) :=& \left\{ I \subsetneq \{1,2,\ldots,d\} \
	   \left| \ I \ne \emptyset, \quad
	     \sum_{i \in I}\frac{1}{1-\lambda_i}=0
	   \right. \right\} \ {\text and}  \\
	 \mathcal{K}(\lambda) :=& \left\{ K \subseteq \{1,2,\ldots,d\}
	   \bigm| K\ne\emptyset.\
	    \text{ If $i,j \in K$, then $\lambda_i = \lambda_j$} \right\}.
	\end{align*}
	Moreover $\#\left( \Phi_d^{-1}\left(\bar{\la}\right) \right)$ is computed 
	in finitely many steps only by using $\mathcal{I}(\lambda)$ and $\mathcal{K}(\lambda)$. 
	\label{en.mt2}
  \item If
	$\mathcal{I}(\la) \subseteq \mathcal{I}(\la')$ and
	$\mathcal{K}(\la) \subseteq \mathcal{K}(\la')$ hold
	for $\la, \la' \in V_d$, then
	$\#\left( \Phi_d^{-1}\left(\bar{\la}\right) \right) \ge 
	 \#\left( \Phi_d^{-1}\left(\bar{\la'}\right) \right)$ holds.
	\label{en.mt3}

  \item The equality
	$\#\left( \Phi_d^{-1}\left(\bar{\la}\right) \right) = (d-2)!$
	holds if and only if the set $\mathcal{I}(\lambda)$ is empty and
	the complex numbers $\lambda_1,\ldots,\lambda_d$ are mutually distinct.
	\label{en.mt4}

  \item If there exist non-zero integers $c_1,\ldots,c_d$
	which satisfy the conditions
	$c_1(1-\lambda_1) = \cdots = c_d(1-\lambda_d)$
	and $\sum_{i=1}^{d}|c_i| \le 2(d-2)$,
	then the set $\Phi_d^{-1}\left(\bar{\la}\right)$ is empty.
	\label{en.mt5}

  \item In the case $d \le 7$,
	the converse of the assertion~(\ref{en.mt5}) holds.
	\label{en.mt6}

  \item In every degree $d$, the Chebyshev polynomial provides an example of an element of $\Phi_d^{-1}\left(\bar{\la}\right)$
	if $\la \in V_d$ satisfies the condition $c_1(1-\lambda_1) = \cdots = c_d(1-\lambda_d)$ for some non-zero integers $c_i$
	with $\sum_{i=1}^{d}c_i = 0$,  $\sum_{i=1}^{d}|c_i| = 2(d-1)$ and $|c_i| \le 2$ for $1 \le i \le d$.
	\label{en.mt7}
 \end{enumerate}
\end{maintheorem}

The algorithm of the computation in Main Theorem~\ref{mthm.1}(\ref{en.mt2}) is given later
in Definition~\ref{df.1.6} and Main Theorem~\ref{mthm.3}.

\begin{remark}\label{rem.1.3}
 There is some overlap between Main Theorem~\ref{mthm.1} above and the results by Fujimura.
	\begin{itemize}
	 \item She showed in~\cite{fu} that if $\Phi_d^{-1}\left(\bar{\la}\right)$ is finite for $\bar{\la} \in \widetilde{\Lambda}_d$,
		then $0 \le \#\left( \Phi_d^{-1}\left(\bar{\la}\right) \right) \le (d-2)!$ holds.
		We removed the assumption that $\Phi_d^{-1}\left(\bar{\la}\right)$ is finite in the case $\bar{\la} \in \widetilde{V}_d$.
	 \item She showed in~\cite[Theorem 6]{fu} that if $\mathcal{I}(\lambda)$ is empty, 
		then $\#\left( \Phi_d^{-1}\left(\bar{\la}\right) \right) = (d-2)!$ holds counted with multiplicity.
		Main Theorem~\ref{mthm.1}(\ref{en.mt4}) is a strengthening of this result.
	 \item She also gave a sufficient condition for $\Phi_d^{-1}(\bar{\la})$ to be empty in~\cite[Theorem 12]{fu}.
		For $d \leq 5$, her condition is equivalent to that in Main Theorem~\ref{mthm.1}(\ref{en.mt5}).
		However for $d \geq 6$, her condition is stricter than ours.
		In the case $d = 6$, 
		Example~\ref{ex.0.1} in Section~\ref{sec.1.1.2} in this paper is the unique example 
		which satisfies our condition~(\ref{en.mt5}) but not Fujimura's condition in her Theorem 12.
	 \item In the case $d \leq 5$, she also showed Main Theorem~\ref{mthm.1}(\ref{en.mt6}) in~\cite[Theorem 5]{fu}.
	 \item Fujimura and Taniguchi's compactification~\cite{tani} gives us a geometric insight of the fiber structure of $\Phi_d$.
		Especially it provides an intuitional explanation of the reasons
		why $\mathcal{I}(\la)$ naturally arises in the computation of $\#\left( \Phi_d^{-1}\left(\bar{\la}\right) \right)$.
		See also Remark~\ref{rm.2.1}.
	\end{itemize}
\end{remark}

\begin{remark}
 The importance of this paper is that
 we can {\it completely} count the number of elements of each fiber $\Phi_d^{-1}(\bar{\lambda})$ 
 for {\it all} $\lambda \in V_d$ {\it without exception}
 as we will see in Main Theorem~\ref{mthm.3} and Section~\ref{sec.1.1.2}.
 The main technical tools that we use for the proof of main theorems
 are a certain extension of Bezout's theorem on projective space $\mathbb{P}^n$
 (see Proposition~\ref{pr.4.1}) and the relation between intersection multiplicity and 
 the degree of finite branched covering 
 (see Propositions~\ref{pr.5.9}, \ref{pr.5.11}, \ref{pr.6.4}, \ref{pr.6.5} and \ref{pr.6.6}),
 which are common in the area of complex algebraic geometry.
\end{remark}

\begin{remark}
 The assertion~(\ref{en.mt7}) shows that the estimate $\sum_{i=1}^{d}|c_i| \le 2(d-2)$ in the assertion~(\ref{en.mt5}) is sharp,
 because $\sum_{i=1}^{d}|c_i|$ must be even.
 However this does not assure the converse of~(\ref{en.mt5}).
\end{remark}

\begin{conjecture}\label{cj.1}\ 
 \begin{enumerate}
  \item The converse of the assertion~(\ref{en.mt5}) also holds
	in the case $d\ge 8$.
	\label{en.cj1.1}
  \item If
	$\mathcal{I}(\la) \subsetneq \mathcal{I}(\la')$ and
	$\mathcal{K}(\la) \subseteq \mathcal{K}(\la')$ hold
	for $\la, \la' \in V_d$, then
	$\#\left( \Phi_d^{-1}\left(\bar{\la}\right) \right) >
	 \#\left( \Phi_d^{-1}\left(\bar{\la'}\right) \right)$ holds.
	\label{en.cj1.2}
 \end{enumerate}
\end{conjecture}
The above conjecture is completely reduced to the problems on combinatorics
by Main Theorem~\ref{mthm.3}.

The local fiber structure of the map $\Phi_d$ is also determined 
by $\mathcal{I}(\la)$ and $\mathcal{K}(\la)$ as in the following:

\begin{maintheorem}\label{mthm.2}\
 \begin{enumerate}
  \item For any $\la,\la'\in V_d$ with
	 $\mathcal{I}(\la)=\mathcal{I}(\la')$ and
	 $\mathcal{K}(\la)=\mathcal{K}(\la')$,
	 there exist open neighborhoods $\widetilde{U}\ni \bar{\la}$,
	 $\widetilde{U}'\ni \bar{\la}'$
	 in $\widetilde{V}_d$ and biholomorphic maps
	 $\mathfrak{L}:\Phi_d^{-1}\bigl(\widetilde{U}\bigr)\to
	  \Phi_d^{-1}\bigl(\widetilde{U}'\bigr)$,
	 $\widetilde{L}:\widetilde{U}\to \widetilde{U}'$
	 and $L:U\to U'$ with $L(\la)=\la'$
	 such that the following conditions~(\ref{en.mthm2.1.1})
	 and~(\ref{en.mthm2.1.2}) are satisfied,
	 where $U,U'$ are the connected components of
	 $\textit{pr}^{-1}\bigl(\widetilde{U}\bigr)$, 
	$\textit{pr}^{-1}\bigl(\widetilde{U}'\bigr)$ in $V_d$
	containing $\la,\la'$ respectively. \label{en.mthm2.1}
	 \begin{enumerate}
	  \item The equalities
		$\Phi_d\circ\mathfrak{L}=\widetilde{L}\circ\Phi_d
		 |_{\Phi_d^{-1}(\widetilde{U})}$ and
		$\textit{pr}\circ L=\widetilde{L}\circ\textit{pr}|_U$ hold.
		\label{en.mthm2.1.1}
	  \item For any $\la''\in U$,
		the equalities $\mathcal{I}(\la'')=\mathcal{I}(L(\la''))$ and
		$\mathcal{K}(\la'')=\mathcal{K}(L(\la''))$ hold.
		\label{en.mthm2.1.2}
	 \end{enumerate}
  \item  For any
	 $(\mathcal{I},\mathcal{K})\in
	 \left\{(\mathcal{I}(\la),\mathcal{K}(\la))\bigm| \la\in V_d \right\}$,
	 the following properties~(\ref{en.mthm2.2.1}), (\ref{en.mthm2.2.2}) 
	 and~(\ref{en.mthm2.2.3}) 
	 hold for the sets
	 \label{en.mthm2.2}
	 \begin{align*}
	  \widetilde{V}\left(\mathcal{I},\mathcal{K}\right)
	   &:= \left\{\bar{\la}\in \widetilde{V}_d \bigm| \la\in V_d,\ 
	  \mathcal{I}(\la)=\mathcal{I} \text{ and } 
	  \mathcal{K}(\la)=\mathcal{K}\right\}, \\
	  \widetilde{V}\left(\mathcal{I},*\right)
	   &:= \left\{\bar{\la}\in \widetilde{V}_d \bigm| \la\in V_d,\ 
	  \mathcal{I}(\la)=\mathcal{I}\right\}, \\
	  \widetilde{V}\left(*,\mathcal{K}\right)
	   &:= \left\{\bar{\la}\in \widetilde{V}_d \bigm| \la\in V_d,\ 
	  \mathcal{K}(\la)=\mathcal{K}\right\}:
	 \end{align*}
	 \begin{enumerate}
	  \item the map
		$
	  \Phi_d|_{\Phi_d^{-1}(\widetilde{V}\left(\mathcal{I},*\right))} :
	  \Phi_d^{-1}\bigl(\widetilde{V}\left(\mathcal{I},*\right)\bigr) \to
	   \widetilde{V}\left(\mathcal{I},*\right)
		$
		is proper. \label{en.mthm2.2.1}
	  \item The map
		$
	  \Phi_d|_{\Phi_d^{-1}(\widetilde{V}\left(*,\mathcal{K}\right))} :
	  \Phi_d^{-1}\bigl(\widetilde{V}\left(*,\mathcal{K}\right)\bigr) \to
	   \widetilde{V}\left(*,\mathcal{K}\right)
		$
		is locally homeomorphic. \label{en.mthm2.2.2}
	  \item For each connected component $X$ of $\Phi_d^{-1}\bigl(
		\widetilde{V}\left(\mathcal{I},\mathcal{K}\right)\bigr)$,
		the map
		$
		 \Phi_d|_X
		 : X \to \widetilde{V}\left(\mathcal{I},\mathcal{K}\right)
		$
		is an unbranched covering. \label{en.mthm2.2.3}
	 \end{enumerate}
 \end{enumerate}
\end{maintheorem}

\begin{remark}\label{rm.2.1}
 The above assertion~(\ref{en.mthm2.2.1}) implies that $\mathcal{I}(\la)$ dominates
 the information on the number of `holes' on each fiber of the map $\Phi_d$.
 Fujimura and Taniguchi~\cite{tani} showed that
 the map $\Phi_d : \mathrm{MP}_d \to \widetilde{\Lambda}_d$
 is extended to
 the map $\widehat{\Psi}_d : \widehat{M}_d \to \mathbb{P}^{d-1}$,
 where $\widehat{M}_d$ is their compactification of $\mathrm{MP}_d$.
 In our context, 
 the condition $\mathcal{I}(\la) \ne \emptyset$ holds for $\la \in V_d$ if and only if
 $\widehat{\Psi}_d^{-1}(\bar{\la}) \cap 
 \left(\widehat{M}_d \setminus \mathrm{MP}_d \right) \ne \emptyset$.

 On the other hand, the above assertion~(\ref{en.mthm2.2.2}) implies that 
 the condition $\mathcal{K}(\la) \supsetneq \{ \{1\}, \ldots, \{d\} \}$ holds for $\la \in V_d$
 if $\bar{\la}$ lies on the branch locus of the map $\Phi_d$.
\end{remark}

To state Main Theorem~\ref{mthm.3} explicitly, we need some more notations,
which are defined in Definition~\ref{df.1.6} and are often used later
in the proof of the main theorems.
After reading Sections~\ref{sec.2},~\ref{sec.3},~\ref{sec.4} and~\ref{sec.7},
the readers will find that the process in Main Theorem~\ref{mthm.3} is natural.

\begin{definition}\label{df.1.6}
 Let $\la=(\la_1,\ldots,\la_d)$ be an element of $V_d$.  Then
 \begin{itemize}
  \item we put 
	\[
	 \mathfrak{I}(\la) := \left\{ \left\{I_1,\ldots,I_l\right\}\ \left| \ 
	 \begin{matrix}
	  I_1 \amalg \cdots \amalg I_l = \{1,\ldots,d \},\ \ l \ge 2, \\
	  I_u\in\mathcal{I}(\la)
	    \textrm{ for each } 1\le u\le l
	 \end{matrix}
	 \right.\right\},
	\]
	where $I_1 \amalg \cdots \amalg I_l$ denotes the disjoint union 
	of $I_1,\ldots,I_l$.
	The partial order $\prec$ in $\mathfrak{I}(\la)$ is defined
	by the refinement of partitions, namely, 
	for $\mb{I}, \mb{I}'\in\mathfrak{I}(\la)$,
	the relation $\mb{I} \prec \mb{I}'$ holds if and only if
	$\mb{I}'$ is a refinement of $\mb{I}$ as partitions of $\{1,\ldots,d \}$.
	Note that $\mathfrak{I}(\la)$ gives the equivalent information
	as $\mathcal{I}(\la)$.
	(For more detail, see Remark~\ref{rm.4.4} and Section~\ref{sec.1.1.2}.)
  \item We denote by $K_1,\ldots,K_q$ the collection of maximal elements 
	of $\mathcal{K}(\la)$ with respect to the inclusion relations, i.e., 
	\[
	 \left\{K_1,\ldots,K_q\right\}= \left\{K\in\mathcal{K}(\la) \bigm|
	   i\in K,\ j\in\{1,\ldots,d\}\setminus K\Longrightarrow\la_i\ne\la_j
	 \right\}.
	\]
	Note that the equality $K_1 \amalg \cdots \amalg K_q = \{1,\ldots,d \}$
	always holds by definition.
	We put $\kappa_w:=\#(K_w)$ for $1\le w\le q$
	and denote by $g_w$ the greatest common divisor of 
	$\kappa_1,\ldots,\kappa_{(w-1)},(\kappa_w)-1,\kappa_{(w+1)},\ldots,\kappa_q$
	for each $1\le w\le q$.
  \item We define the function $m$ by $m(z):=\frac{1}{1-z}$
	for $z\in \mathbb{C}\setminus \{1\}$.
  \item We may assume $\la \in V_d$ to be in the form
	\[
	 \la=(\underbrace{\la_1,\ldots,\la_1}_{\kappa_1},\ldots,
  	  \underbrace{\la_q,\ldots,\la_q}_{\kappa_q}),
	\]
	where $\la_1,\ldots,\la_q$ are mutually distinct.
	For each $1\le w\le q$ and for each divisor $t$ of $g_w$ with $t\ge 2$,
	we put $d[t]:= \frac{d-1}{t}+1$ and
	denote by $\la[t]$ the element of $V_{d[t]}$ such that
	\begin{multline*}
	 \la[t]:=(
	  \underbrace{m^{-1}(tm(\la_1)),\ldots,
	   m^{-1}(tm(\la_1))}_{\frac{\kappa_1}{t}},
	  \ldots, \\
	  \underbrace{m^{-1}(tm(\la_w)),\ldots,
	   m^{-1}(tm(\la_w))}_{\frac{(\kappa_w)-1}{t}},
	  \ldots,
	  \underbrace{m^{-1}(tm(\la_q)),\ldots,
	   m^{-1}(tm(\la_q))}_{\frac{\kappa_q}{t}},
	 \la_w).
	\end{multline*}
	Note that $w$ is determined by $t$ and 
	that $\mathcal{I}(\la[t])$ is determined 
	by $\mathcal{I}(\la),\mathcal{K}(\la)$ and $t$.
 \end{itemize}
\end{definition}
\begin{maintheorem}\label{mthm.3}
 For $\la=(\la_1,\ldots,\la_d) \in V_d$,
 the cardinality $\#\left( \Phi_d^{-1}(\bar{\lambda}) \right)$
 is computed in the following steps.
 \begin{itemize}
  \item For each $\mathbb{I}=\{I_1,\ldots,I_l\}\in \mathfrak{I}(\la)$,
	we define the number $e_{\mathbb{I}}(\la)$ inductively by the equality
	\begin{equation}\label{eq.7.D1}
	 e_{\mathbb{I}}(\la) :=
	 \left(\prod_{u=1}^l \bigl( \#\left(I_u\right) -1 \bigr)! \right)
	  - \sum_\textrm{\scriptsize $\begin{matrix}
				       \mathbb{I}' \in \mathfrak{I}(\la) \\
				       \mathbb{I}' \succ \mathbb{I}, \;
				       \mathbb{I}' \ne \mathbb{I}
				     \end{matrix}$}
	 \left(
	  e_{\mathbb{I}'}(\la) \cdot \prod_{u=1}^l 
	  \left( \prod_{k=\#(I_u)-\chi_u(\mathbb{I}')+1 }^{\#(I_u)-1}k \right)
	 \right),
	\end{equation}
	where we put
	$\chi_u(\mathbb{I}'):=\#\left(\left\{ I'\in\mathbb{I}' \bigm|
	I' \subseteq I_u \right\}\right)$
	for $\mathbb{I}' \succ \mathbb{I}$.
	Note that in the case $\chi_u(\mathbb{I}') = 1$, we assume that
	$\prod_{k=\#(I_u)-\chi_u(\mathbb{I}')+1 }^{\#(I_u)-1}k = \prod_{k=\#(I_u)}^{\#(I_u)-1}k = 1$.
  \item We put
	\begin{equation}\label{eq.7.C1}
	 s_d(\la) := (d-2)!
	  - \sum_{\mathbb{I} \in \mathfrak{I}(\la)} \left(
          e_{\mathbb{I}}(\la)\cdot
          \prod_{k=d-\#(\mathbb{I})+1}^{d-2}k \right).
	\end{equation}
	Note that in the case $\#(\mathbb{I}) = 2$, we assume that 
	$\prod_{k=d-\#(\mathbb{I})+1}^{d-2}k = \prod_{k=d-1}^{d-2}k = 1$.
  \item Moreover we define the numbers $c_t(\la)$ for
	$t\in \bigcup_{1\le w \le q} \left\{t \bigm| t|g_w \right\}$
	by the equalities
	\begin{equation}\label{eq.7.E1}
	 \sum_{t|b,\ b|g_w}
	  \frac{t}{b}\,c_b(\la) = \frac{s_{d[t]}(\la[t])}{
	  \left(\frac{\kappa_1}{t}\right)!\cdots
	  \left(\frac{\kappa_{(w-1)}}{t}\right)!
	  \left(\frac{(\kappa_w)-1}{t}\right)!
	  \left(\frac{\kappa_{(w+1)}}{t}\right)!
	  \cdots\left(\frac{\kappa_q}{t}\right)! }
	\end{equation}
	for $(w,t) \in \left\{(w,t) \bigm|
	1 \le w \le q,\ t|g_w,\ t \ge 2 \right\}$,
	and
	\begin{equation}\label{eq.7.E2}
	 c_1(\la)
	 + \sum_{w=1}^q\left(\sum_{t|g_w,\ t\ge 2} \frac{1}{t}\,c_t(\la)\right)
	 = \frac{s_d(\la)}{\kappa_1!\cdots\kappa_q!},
	\end{equation}
	where $t|b$ denotes that $t$ divides $b$
	for positive integers $t$ and $b$.
  \item Then the numbers $e_{\mathbb{I}}(\la)$, $s_d(\la)$ and $c_t(\la)$ are
	non-negative integers.
	Moreover we have
	\begin{equation}\label{eq.7.E3}
	 \#\left(\Phi_d^{-1}\left(\bar{\la}\right)\right) = \sum_t c_t(\la)
	  = c_1(\la)+\sum_{w=1}^q\left(\sum_{t|g_w,\ t\ge 2} c_t(\la)\right).
	\end{equation}
 \end{itemize}
\end{maintheorem}

\begin{remark}
 Note that all the numbers defined in Main Theorem~\ref{mthm.3}
 are determined
 by $\mathcal{I}(\la)$ and $\mathcal{K}(\la)$.
 Especially the number $s_d(\la)$ is determined only by $\mathcal{I}(\la)$.
 If we count the number $\#\left(\Phi_d^{-1}\left(\bar{\la}\right)\right)$ with multiplicity,
 then we always have $\#\left(\Phi_d^{-1}\left(\bar{\la}\right)\right) = s_d(\la)$.
 However in our context, we do not consider $\#\left(\Phi_d^{-1}\left(\bar{\la}\right)\right)$
 counted with multiplicity, and therefore need some more computation.
 The number $s_d(\la)$ is the cardinality of the set $S_d(\la)$
 which will be defined in Definition~\ref{df.2.2}.
\end{remark}

\begin{remark}
 Under the isomorphism $\mathrm{MP}_d\cong\mb{C}^{d-1}/\left(\mb{Z}/(d-1)\mb{Z}\right)$
 in Remark~\ref{rem.1.4}, the action of $\mb{Z}/(d-1)\mb{Z}$ on $\mb{C}^{d-1}$ is not free,
 and $\mathrm{MP}_d$ has the set of singular points ${\rm Sing}(\mathrm{MP}_d)$ 
 for $d \geq 4$.
 If $\bar{\la} \in \widetilde{V}_d$ lies away from the locus $\Phi_d({\rm Sing}(\mathrm{MP}_d))$, 
 then the set $\left\{(w,t) \bigm| 1 \le w \le q,\ t|g_w,\ t \ge 2 \right\}$
 in the third step in Main Theorem~\ref{mthm.3} 
 is empty,
 and therefore we have
 $\#\left(\Phi_d^{-1}\left(\bar{\la}\right)\right) = c_1(\la)
  = s_d(\la)/(\kappa_1!\cdots\kappa_q!)$.
\end{remark}

\begin{problem}
 Give a combinatorial proof of the fact that for any $\la\in V_d$ and for any $t$,
 the number $c_t(\la)$ defined above is a non-negative integer.
 Note that the proof given in this paper is not combinatorial.
\end{problem}

For parameters $\la \in \La_d \setminus V_d$, we have the following:

\begin{remark}
 For $\la=(\la_1,\ldots,\la_d)\in \La_d\setminus V_d$
 with $\#\left\{i\bigm| \la_i=1\right\}\ge4$,
 some connected components of 
 the inverse image $\Phi_d^{-1}\left(\bar{\la}\right)$ may have
 dimension greater than or equal to $1$. However, if we put
 \[
  \mathrm{MP}''_d:= \bigl\{f\in \mathrm{MP}_d\bigm|
   \text{$f$ has at most one multiple fixed point}\bigr\},
 \]
 then the map
 $\Phi_d|_{\mathrm{MP}''_d} : \mathrm{MP}''_d \to \widetilde{\La}_d$
 is finite. Moreover similar results to the main theorems hold for
 $\Phi_d|_{\mathrm{MP}''_d}$ and for any $\la \in \La_d\setminus V_d$,
 whose proofs are also similar to those of the main theorems.
\end{remark}

We shall also comment about $f\in \mathrm{MP}_d$ having more than two multiple fixed points.
 For any $\z\in \mathrm{Fix}(f)$,
 the holomorphic index of $f$ at $\z$ is defined to be the complex number
 $\iota(f,\z):=
  \frac{1}{2\pi\sqrt{-1}}\oint_{|z-\z|=\epsilon}\frac{dz}{z - f(z)}$,
 where $\epsilon$ is a sufficiently small positive real number.
 The index $\iota(f,\z)$ is invariant under biholomorphic transformations,
 and is equal to $\frac{1}{1-f'(\z)}$ if $\z$ is not multiple.
 We denote by $m(f,\z)$ the fixed-point multiplicity of $f$
 at $\z\in\mathrm{Fix}(f)$.
 Then we always have $\sum_{\z\in\mathrm{Fix}(f)}m(f,\z)=\deg f$ and
 $\sum_{\z\in\mathrm{Fix}(f)}\iota(f,\z)=0$.
 Moreover we have $\iota(f,\z)\ne 0$ whenever $m(f,\z)=1$.
 Note that $\mathrm{Fix}(f)$ is not considered counted with multiplicity
 only here and in the following conjecture.

\begin{conjecture}
 We consider the map
 $\widetilde{\Phi}_d$, instead of $\Phi_d$,
 which assigns\\
 $\widetilde{\Phi}_d(f) =\bigl(\left[\iota(f,\z),m(f,\z)\right]\bigr)_{\z\in\mathrm{Fix}(f)}$
 to each $f \in \mathrm{MP}_d$,
 so that the target space of $\widetilde{\Phi}_d$ is defined to be the family of unordered collections of
 pairs $[m_i,d_i]$ with $m_i \in \mathbb{C},\ d_i \in \mathbb{Z},\ d_i \geq 1,\ \sum_i d_i = d$
 and $\sum_i m_i = 0$.
 Then it is conjectured that the map $\widetilde{\Phi}_d$ is finite 
 and that similar results to the main theorems hold
 for $\widetilde{\Phi}_d$ and for any parameter value without exception.
\end{conjecture}

We have ten sections in this paper.
In Section~\ref{sec.1.1.2}, we give some examples which illustrate 
the calculation 
of $\#\left(\Phi_d^{-1}\left(\bar{\la}\right)\right)$ in Main Theorem~\ref{mthm.3}.   
In Section~\ref{sec.1.2}, we give
 the detailed program of the 
remaining 
sections.
Sections from~\ref{sec.2} to~\ref{sec.8}
are devoted to the proofs of
Main Theorems~\ref{mthm.1},~\ref{mthm.2} and ~\ref{mthm.3}.

\proof[\bf Acknowledgements]
The results of this paper were obtained when the author was a graduate student of 
Department of Mathematics, Kyoto University.
He would like to express his thanks to his supervisor Professor Mitsuhiro Shishikura
for various kinds of advices and informations.
He also thanks 
to Professor Hiroshi Kokubu and Professor Hiroki Sumi 
for valuable advices on this paper, 
and to Professor Masashi Kisaka for useful informations
concerning the results.
Finally he expresses his thanks to the referee for valuable comments and helpful suggestions.

\section{Some Examples}\label{sec.1.1.2}

In this section, we give three examples which illustrate the calculation 
of $\#\left(\Phi_d^{-1}\left(\bar{\la}\right)\right)$ in Main Theorem~\ref{mthm.3}. 

\begin{example}\label{ex.0.1}
 We consider an element $\la = (\la_1,\dots,\la_6) \in V_6$ satisfying the equality
 \[
    \frac{1}{1-\lambda_1} : \cdots : \frac{1}{1-\lambda_6} = 1:1:2:-1:-1:-2.
 \]
 In this case we have $\#\left(\Phi_6^{-1}\left(\bar{\la}\right)\right) = 0$
 by the assertion~(\ref{en.mt5}) in Main Theorem~\ref{mthm.1}; 
 however in this example we shall find it again 
 by following the steps in Main Theorem~\ref{mthm.3}.

 By definition, we have
 $\mathfrak{I}(\la)= \left\{\mb{I}_{\omega}\bigm| 1\le \omega \le 8\right\}$,
 where
 \begin{align*}
  \mb{I}_1 &= \bigl\{ \{1,4\},\{2,5\},\{3,6\} \bigr\}, \
  \mb{I}_2  = \bigl\{ \{1,5\},\{2,4\},\{3,6\} \bigr\}, \\
  \mb{I}_3  &= \bigl\{ \{1,2,4,5\},\{3,6\} \bigr\},\
  \mb{I}_4  = \bigl\{ \{1,4\},\{2,3,5,6\} \bigr\},\
  \mb{I}_5 = \bigl\{ \{2,5\},\{1,3,4,6\} \bigr\},\\
  \mb{I}_6  &= \bigl\{ \{1,5\},\{2,3,4,6\} \bigr\},\
  \mb{I}_7 = \bigl\{ \{2,4\},\{1,3,5,6\} \bigr\} \text{ and }
  \mb{I}_8  = \bigl\{ \{1,2,6\},\{3,4,5\} \bigr\}.
 \end{align*}
 We have $\mb{I}_3\prec\mb{I}_1,\mb{I}_4\prec\mb{I}_1,\mb{I}_5\prec\mb{I}_1,
  \mb{I}_3\prec\mb{I}_2,\mb{I}_6\prec\mb{I}_2$ and $\mb{I}_7\prec\mb{I}_2$;
 hence the maximal elements of $\mathfrak{I}(\la)$ are
 $\mb{I}_1$, $\mb{I}_2$ and $\mb{I}_8$.

 By the equality~(\ref{eq.7.D1}), we have $e_{\mathbb{I}_1}(\la) = e_{\mathbb{I}_2}(\la) = (2-1)! \cdot (2-1)!\cdot (2-1)! = 1$
 and $e_{\mathbb{I}_8}(\la) = (3-1)! \cdot (3-1)! = 4$. Moreover we have 
 $e_{\mathbb{I}_3}(\la) = (4-1)! \cdot (2-1)! - \left( e_{\mathbb{I}_1}(\la) \cdot 3 + e_{\mathbb{I}_2}(\la) \cdot 3 \right)
 = 6-(3+3)=0$,
 $e_{\mathbb{I}_4}(\la) = e_{\mathbb{I}_5}(\la) = (2-1)! \cdot (4-1)! - e_{\mathbb{I}_1}(\la) \cdot 3 = 6-3=3$ and
 $e_{\mathbb{I}_6}(\la) = e_{\mathbb{I}_7}(\la) = (2-1)! \cdot (4-1)! - e_{\mathbb{I}_2}(\la) \cdot 3 = 6-3=3$.
 Hence by the equality~(\ref{eq.7.C1}), we have $s_6(\la) = (6-2)! - \left(\sum_{\omega=1}^2 e_{\mathbb{I}_{\omega}}(\la) \cdot 4 
 + \sum_{\omega=3}^8 e_{\mathbb{I}_{\omega}}(\la) \right) = 24 - (4+4+0+3+3+3+3+4)=0$,
 which implies $\#\left(\Phi_6^{-1}\left(\bar{\la}\right)\right) = c_1(\la) =0.$
\end{example}

\begin{example}\label{ex.0.2}
 In this example we consider $\la = (\la_1,\dots,\la_{31})\in V_{31}$ with
 \[
    \frac{1}{1-\lambda_1} : \cdots : \frac{1}{1-\lambda_{31}} = \underbrace{6 : \dots : 6}_{25} : \underbrace{-25 : \dots : -25}_{6}.
 \]
 In this case we have $\mathfrak{I}(\la)=\emptyset$ and $s_{31}(\la) = 29!$ by the equality~(\ref{eq.7.C1}).

 On the other hand, by Definition~\ref{df.1.6}, we have $q=2$, $K_1 = \{1, \dots, 25\}$, $K_2 = \{26, \dots, 31\}$, $\kappa_1=25$,
 $\kappa_2=6$, $g_1=\gcd(\kappa_1-1, \kappa_2)=6$, $g_2 = 5$,
 $\bigcup_{1\le w \le 2} \left\{t \bigm| t|g_w \right\} = \left\{ 1, 2, 3, 6, 5 \right\}$,
 $d[2] = \frac{31-1}{2}+1 = 16$, $d[3] = 11$, $d[6] = 6$ and $d[5] = 7$.
 Moreover we have\\
 $\la[2] = \left( \la[2]_1, \dots, \la[2]_{16} \right) \in V_{16}$ with
 $\frac{1}{1-\la[2]_1} : \dots : \frac{1}{1-\la[2]_{16}}
  = \underbrace{12 : \dots : 12}_{\frac{\kappa_1-1}{2}=12} : \underbrace{-50 : -50 : -50}_{\frac{\kappa_2}{2}=3} : 6$.
 Similarly we have\\
 $\la[3] = \left( \la[3]_1, \dots, \la[3]_{11} \right) \in V_{11}$ with
 $\frac{1}{1-\la[3]_1} : \dots : \frac{1}{1-\la[3]_{11}} = \underbrace{18 : \dots : 18}_{8} : -75 : -75 : 6$,\\
 $\la[6] = \left( \la[6]_1, \dots, \la[6]_{6} \right) \in V_{6}$
 with $\frac{1}{1-\la[6]_1} : \dots : \frac{1}{1-\la[6]_{6}} = \underbrace{36 : \dots : 36}_{4} : -150 : 6$ and\\
 $\la[5] = \left( \la[5]_1, \dots, \la[5]_{7} \right) \in V_{7}$ with
 $\frac{1}{1-\la[5]_1} : \dots : \frac{1}{1-\la[5]_{7}}
 = \underbrace{30 : \dots : 30}_{\frac{\kappa_1}{5}=5} : \underbrace{-125}_{\frac{\kappa_2-1}{5}=1} : -25$.

 Since $\mathfrak{I}(\la)=\emptyset$, we have $\mathfrak{I}(\la[t])=\emptyset$ for $t=2, 3, 6, 5$,
 which implies $s_{16}(\la[2]) = 14!$, $s_{11}(\la[3]) = 9!$, $s_{6}(\la[6]) = 4!$ and $s_{7}(\la[5]) = 5!$ by the equality~(\ref{eq.7.C1}).
 By the equality~(\ref{eq.7.E1}) for $(w,t) = (1,6)$, $(1,3)$, $(1,2)$ and $(2,5)$, we have 
 \begin{align*}
   \frac{6}{6}c_6(\la) = \frac{s_{6}(\la[6])}{\left( \frac{\kappa_1 - 1}{6} \right)!\cdot\left( \frac{\kappa_2}{6} \right)!}
   = \frac{4!}{4!\cdot1!} &= 1, \quad
   \frac{3}{3}c_3(\la) + \frac{3}{6}c_6(\la)
   = \frac{s_{11}(\la[3])}{\left( \frac{\kappa_1 - 1}{3} \right)!\cdot\left( \frac{\kappa_2}{3} \right)!}
   = \frac{9!}{8!\cdot 2!} = \frac{9}{2},\\
   \frac{2}{2}c_2(\la) + \frac{2}{6}c_6(\la)
   = \frac{s_{16}(\la[2])}{\left( \frac{\kappa_1 - 1}{2} \right)!\cdot\left( \frac{\kappa_2}{2} \right)!}
   &= \frac{14!}{12!\cdot 3!} = \frac{91}{3}, \quad
   \frac{5}{5}c_5(\la) = \frac{s_{7}(\la[5])}{\left( \frac{\kappa_1}{5} \right)!\cdot\left( \frac{\kappa_2 - 1}{5} \right)!}
   = \frac{5!}{5!\cdot1!} = 1
 \end{align*}
 respectively, which implies $c_6(\la) = 1$, $c_3(\la) = 4$, $c_2(\la) = 30$ and $c_5(\la) = 1$.
 Moreover by the equality~(\ref{eq.7.E2}), we have 
 \[
   c_1(\la) + \frac{1}{2}c_2(\la) + \frac{1}{3}c_3(\la) + \frac{1}{6}c_6(\la) + \frac{1}{5}c_5(\la)
   = \frac{s_{31}(\la)}{\kappa_1! \cdot \kappa_2!} = \frac{29!}{25!\cdot 6!} 
   = \frac{7917}{10},
 \]
 which implies $c_1(\la) = 775$.
 Hence by~(\ref{eq.7.E3}), we have 
 \[
   \#\left(\Phi_{31}^{-1}\left(\bar{\la}\right)\right) = c_1(\la) + c_2(\la) + c_3(\la) + c_6(\la) + c_5(\la) = 775 + 30 + 4 + 1 + 1 = 811.
 \]
\end{example}

\begin{example}\label{ex.0.3}
 Here we consider a little complicated example, which is $\la = (\la_1,\dots,\la_9) \in V_9$ with
 $\frac{1}{1-\lambda_1} : \cdots : \frac{1}{1-\lambda_{9}} = 2:2:2:2:-1:-1:-2:-2:-2$.
 In this case, by Definition~\ref{df.1.6}, we have $q=3$, $\kappa_1=4$, $\kappa_2=2$, $\kappa_3=3$, $g_1=g_2=1$ and $g_3=2$.
 Hence we must find $s_9(\la)$ and $s_5(\la[2])$, and after that
 by the equalities~(\ref{eq.7.E1}) and~(\ref{eq.7.E2}) we have
 \begin{equation}\label{eq.ex3}
   \frac{2}{2}c_2(\la) = \frac{s_5(\la[2])}{\left(4/2\right)! \cdot \left(2/2\right)! \cdot \left((3-1)/2\right)!} = \frac{s_5(\la[2])}{2}
   \quad \text{and} \quad
   c_1(\la) + \frac{1}{2}c_2(\la) = \frac{s_9(\la)}{4!\cdot 2! \cdot 3!}.
 \end{equation}

 We shall find $s_5(\la[2])$ first.
 Since $\la[2] = (\la[2]_1,\dots,\la[2]_5) \in V_5$ with
 $\frac{1}{1-\la[2]_1} : \cdots : \frac{1}{1-\la[2]_{5}} = 4:4:-2:-4:-2$,
 we have $\mathfrak{I}(\la[2]) = \left\{\mb{I}'_1, \mb{I}'_2 \right\}$, where
 $\mb{I}'_1 = \bigl\{ \{1,4\},\{2,3,5\} \bigr\}$ and $\mb{I}'_2 = \bigl\{ \{2,4\},\{1,3,5\} \bigr\}$.
 Hence we have $e_{\mb{I}'_1}(\la[2]) = e_{\mb{I}'_2}(\la[2]) = (2-1)!\cdot (3-1)!=2$ and
 $s_5(\la[2]) = (5-2)! - \left( e_{\mb{I}'_1}(\la[2]) + e_{\mb{I}'_2}(\la[2]) \right) = 6-(2+2)=2$, which implies $c_2(\la) = \frac{2}{2}=1$
 by the equality~(\ref{eq.ex3}).

 On the other hand, the computation of $s_9(\la)$ is much more complicated than that of $s_5(\la[2])$.
 First of all, $\mathfrak{I}(\la)$ consists of $130$ elements, and we shall express them by 
 	\begin{multline*}
	 \mathfrak{I}(\la) =  \left\{\mb{I}_{(1,\omega)}\bigm| 1\le \omega \le 24\right\}
	 \cup \left\{\mb{I}_{(2,\omega)}\bigm| 1\le \omega \le 36\right\}
	 \cup \left\{\mb{I}_{(3,\omega)}\bigm| 1\le \omega \le 36\right\}\\
	 \cup \left\{\mb{I}_{(4,\omega)}\bigm| 1\le \omega \le 12\right\}
	 \cup \left\{\mb{I}_{(5,\omega)}\bigm| 1\le \omega \le 18\right\}
	 \cup \left\{\mb{I}_{(6,\omega)}\bigm| 1\le \omega \le 4\right\}.
	\end{multline*}
 Here $\mb{I}_{(1,\omega)}$ for $1 \le \omega \le 24$ are of the form
 $\bigl\{ \{\sigma(1),5,6\},\{\sigma(2),7\},\{\sigma(3),8\},\{\sigma(4),9\} \bigr\}$
 for $\sigma \in \mathfrak{S}_4 = \mathrm{Aut}(\{1,2,3,4\})$.
 Similarly $\mb{I}_{(2,\omega)}$, $\mb{I}_{(3,\omega)}$, $\mb{I}_{(4,\omega)}$, $\mb{I}_{(5,\omega)}$ and $\mb{I}_{(6,\omega)}$ are
 of the form
	\begin{gather*}
	 \bigl\{ \{\sigma(1),\sigma(2),5,6,\tau(7)\},\{\sigma(3),\tau(8)\},\{\sigma(4),\tau(9)\} \bigr\},\\
	 \bigl\{ \{\sigma(1),5,6\},\{\sigma(2),\tau(7)\},\{\sigma(3),\sigma(4),\tau(8),\tau(9)\} \bigr\},\\
	 \bigl\{ \{\sigma(1),\sigma(2),\sigma(3),5,6,\tau(7),\tau(8)\},\{\sigma(4),\tau(9)\} \bigr\},\\
	 \bigl\{ \{\sigma(1),\sigma(2),5,6,\tau(7)\},\{\sigma(3),\sigma(4),\tau(8),\tau(9)\} \bigr\}\\
	 \text{and} \ \bigl\{ \{\sigma(1),5,6\},\{\sigma(2),\sigma(3),\sigma(4),7,8,9\} \bigr\}
	\end{gather*}
 respectively for $\sigma \in \mathfrak{S}_4 = \mathrm{Aut}(\{1,2,3,4\})$ and $\tau \in \mathfrak{S}_3  = \mathrm{Aut}(\{7,8,9\})$.
 By~(\ref{eq.7.D1}) we have $e_{\mb{I}_{(1,\omega)}}(\la) = 2! \cdot 1! \cdot 1! \cdot 1!=2$.
 For each $1 \le \omega \le 36$,
 we have $\#\left( \left\{ \omega' \mid \mb{I}_{(2,\omega)} \prec \mb{I}_{(1,\omega')} \right\} \right) = 2$;
 hence by~(\ref{eq.7.D1}) we have $e_{\mb{I}_{(2,\omega)}}(\la) = 4! \cdot 1! \cdot 1! - 2 \cdot 4 \times 2 = 8$.
 Similarly for each $1 \le \omega \le 36$, we have
 $\#\left( \left\{ \omega' \mid \mb{I}_{(3,\omega)} \prec \mb{I}_{(1,\omega')} \right\} \right) = 2$,
 which implies
 $e_{\mb{I}_{(3,\omega)}}(\la) = 2! \cdot 1! \cdot 3! - 2 \cdot 3 \times 2 = 0$.
 Since
 $\#\left( \left\{ \omega' \mid \mb{I}_{(4,\omega)} \prec \mb{I}_{(1,\omega')} \right\} \right) = 6$,
 $\#\left( \left\{ \omega' \mid \mb{I}_{(4,\omega)} \prec \mb{I}_{(2,\omega')} \right\} \right) = 6$ and
 $\#\left( \left\{ \omega' \mid \mb{I}_{(4,\omega)} \prec \mb{I}_{(3,\omega')} \right\} \right) = 3$,
 we have 
 $e_{\mb{I}_{(4,\omega)}}(\la) = 6! \cdot 1! - (2 \cdot (5 \cdot 6) \times 6 + 8 \cdot 6 \times 6 + 0 \cdot 6 \times 3) = 72$.
 Similarly we have 
 $e_{\mb{I}_{(5,\omega)}}(\la) = 4! \cdot 3! - (2 \cdot 4 \cdot 3 \times 4 + 8 \cdot 3 \times 2 + 0 \cdot 4 \times 2)=0$ and
 $e_{\mb{I}_{(6,\omega)}}(\la) = 2! \cdot 5! - (2 \cdot (4 \cdot 5) \times 6 + 0 \cdot 5 \times 9)=0$.
 Therefore by~(\ref{eq.7.C1}) we have
 $s_9(\la) = 7! - \left( 2 \cdot (6\cdot 7) \times 24 + 8 \cdot 7 \times 36 + 72 \times 12 \right) = 144$.

 To summarize, we have $c_2(\la)=1$ and $c_1(\la)+\frac{1}{2}c_2(\la) = \frac{144}{4!\cdot2!\cdot3!}$ by~(\ref{eq.ex3}),
 which implies $c_1(\la)=0$ and $\#\left(\Phi_{9}^{-1}\left(\bar{\la}\right)\right) = c_1(\la) + c_2(\la) = 0+1=1$.
 Here, the unique element of $\Phi_{9}^{-1}\left(\bar{\la}\right)$ is represented by $f_9(x)$ which is the one defined in the proof of 
 Proposition~\ref{pr.mt7}.
\end{example}

\section{Detailed program of the proof}\label{sec.1.2}

In this section, we describe the detailed program of the proof of the main theorems.

Sections from~\ref{sec.2} to~\ref{sec.8} are devoted to the proofs
of Main Theorems~\ref{mthm.1},~\ref{mthm.2} and~\ref{mthm.3}.
The proofs are self-contained except for the basic knowledge
of the intersection theory on the projective space $\mb{P}^n$ 
(see Section 4 of Chapter 0 and Section 3 of Chapter 1 in~\cite{GriffithsHarris}) 
and the theory on finite branched coverings.
The most important tool for the proof,
which is stated in Proposition~\ref{pr.4.1},
is an extension of Bezout's theorem on $\mb{P}^n$
especially in the case that
some components of the common zeros of $n$ homogeneous polynomials
are not points or are components
which are proper subsets of other components. 
The most difficult and most crucial part in the proof of the main theorems is the proof of Theorem~\ref{thm.B}. 
Theorem~\ref{thm.B} is stated in Section~\ref{sec.4}, and its proof is described in Section~\ref{sec.6}.
Main Theorem~\ref{mthm.2} is naturally proved in the process of 
proving Main Theorems~\ref{mthm.1} and~\ref{mthm.3}.
The assertions~(\ref{en.mt5}) and~(\ref{en.mt7}) in Main Theorem~\ref{mthm.1}
are proved in Section~\ref{sec.2}, and
the assertions~(\ref{en.mt1}) and~(\ref{en.mt4})
in Main Theorem~\ref{mthm.1} are proved in Section~\ref{sec.4}.
On the other hand, the proofs of the rest are completed in Section~\ref{sec.8}.

In Section~\ref{sec.2}
we rewrite the set $\Phi_d^{-1}\left(\bar{\la}\right)$
as follows:
for each $\la\in V_d$,
we define the subsets $T_d(\la)$, $S_d(\la)$ and $B_d(\la)$ of $\Pd$,
where $T_d(\la)$ is the set of the common zeros
of some $(d-2)$ homogeneous polynomials $\varphi_1,\ldots,\varphi_{d-2}$
on $\Pd$, and 
$T_d(\la)=S_d(\la)\amalg B_d(\la)$.
We define
the subgroup $\mathfrak{S}\left(\mathcal{K}(\la)\right)$
of $\mathfrak{S}_d$ acting on $S_d(\la)$,
and show the existence of the bijection
$\overline{\pi(\la)}:S_d(\la)/\mathfrak{S}\left(\mathcal{K}(\la)\right)
 \cong \Phi_d^{-1}\left(\bar{\la}\right)$ in Proposition~\ref{pr.2.3}.
By Proposition~\ref{pr.2.3},
we can divide the proof of Main Theorems~\ref{mthm.1} and~\ref{mthm.3}
into two steps:
the first one is to determine the cardinality $\#\left(S_d(\la)\right)$; 
the second one is to analyze the action
of $\mathfrak{S}\left(\mathcal{K}(\la)\right)$ on $S_d(\la)$.

In Section~\ref{sec.3} 
we review the intersection theory on $\mb{P}^n$ and
give an extension of Bezout's theorem on $\mb{P}^n$
in Proposition~\ref{pr.4.1},
which will be utilized crucially
for determining the cardinality $\#\left(S_d(\la)\right)$ afterward.
In Definitions~\ref{df.4.0} and~\ref{df.4.1},
we define the family $\mathcal{C}(\varphi_1,\ldots,\varphi_m)$
of irreducible varieties
for homogeneous polynomials $\varphi_1,\ldots,\varphi_m$ on $\mb{P}^n$
and the number\\ 
$\mult_C(\varphi_1,\ldots,\varphi_m)$
for each $C\in\mathcal{C}(\varphi_1,\ldots,\varphi_m)$
with $\mathrm{codim}\, C = m$.
Here, $\mathcal{C}(\varphi_1,\ldots,\varphi_m)$ stands for the family
of the ``components'' of the common zeros of $\varphi_1,\ldots,\varphi_m$
in $\mb{P}^n$.
In practice, it contains all the irreducible components
of the common zeros of $\varphi_1,\ldots,\varphi_m$,
and may also contain
some irreducible varieties which are 
proper subsets of
some irreducible components
of the common zeros of $\varphi_1,\ldots,\varphi_m$.
On the other hand, the number $\mult_C(\varphi_1,\ldots,\varphi_m)$ stands for 
the ``intersection multiplicity'' of $\varphi_1,\ldots,\varphi_m$ along $C$;
if $C$ is an irreducible component,
then it is 
the usual intersection multiplicity of $\varphi_1,\ldots,\varphi_m$ along $C$.
Proposition~\ref{pr.4.1} 
gives 
the relation
among these numbers, which is also reduced to the usual Bezout's theorem
if $\mathcal{C}(\varphi_1,\ldots,\varphi_n)$ consists only of points.

In Sections~\ref{sec.4},~\ref{sec.5} and~\ref{sec.6} 
we determine the cardinality $\#\left(S_d(\la)\right)$, 
based on Section~\ref{sec.3}.
More precisely, in Section~\ref{sec.4}, we
give the explicit expression of the set $B_d(\la)$ in Lemma~\ref{lm.4.6},
and determine the number $\mult_C(\varphi_1,\ldots,\varphi_m)$
for each $C\in\mathcal{C}(\varphi_1,\ldots,\varphi_{d-2})$
with $\mathrm{codim}\, C = m$ and $C \subseteq B_d(\la)$
in Theorems~\ref{thm.A} and~\ref{thm.B}.
Some of the elements of $\mathcal{C}(\varphi_1,\ldots,\varphi_{d-2})$ 
may be proper subsets
of other elements,
which makes their computation much complicated.
Proposition~\ref{pr.4.1}, Theorems~\ref{thm.A} and~\ref{thm.B}
give the exact expression of the cardinality $\#\left(S_d(\la)\right)$.
Sections~\ref{sec.5} and~\ref{sec.6} are devoted
to the proofs of Theorems~\ref{thm.A} and~\ref{thm.B} respectively.

In most cases, the action of $\mathfrak{S}\left(\mathcal{K}(\la)\right)$
on $S_d(\la)$ is free.
However in some cases, it is rather complicated.
In Section~\ref{sec.7}
we analyze the action of $\mathfrak{S}\left(\mathcal{K}(\la)\right)$
on $S_d(\la)$ in detail, and give the exact relation
between the cardinalities of
$S_d(\la)$ and $\Phi_d^{-1}\left(\bar{\la}\right)$ in Theorem~\ref{thm.E}.
To summarize, in Section~\ref{sec.8}
we complete the proof of the main theorems.

\section{Another expression of the set~$\Phi_d^{-1}\left(\bar{\la}\right)$}
\label{sec.2}

In this section we start proving the main theorems.
In the rest of this paper, we always assume that
$d$ is a natural number with $d \ge 4$.

An arbitrary polynomial map $f(z) \in \mb{C}[z]$ of degree $d$
can be expressed in the form
 $f(z) = z + \rho(z - \z_1)(z - \z_2) \cdots (z - \z_d)$,
where $\z_1, \z_2, \ldots , \z_d$ and $\rho$ are complex numbers 
with $\rho \ne 0$.
In this expression we have
$\mathrm{Fix}(f) = \left\{ \z_1,\z_2,\ldots,\z_d \right\}$ and
$f'(\z_i) = 1 + \rho \prod_{j \ne i}(\z_i - \z_j)$ for $1 \le i \le d$.
Hence to show Main Theorems~\ref{mthm.1} and~\ref{mthm.3},
we only need to count the number of the solutions of the equations
 $1 + \rho \prod_{j \ne i}(\z_i - \z_j) = \la_i$
for $1\le i\le d$ modulo affine conjugacy.
However we do not take this method.
The following is the key for the proof
of the main theorems.

\begin{keylemma}
 Let $f$ be a polynomial map of degree $d$ expressed in the form
 \begin{equation*}
  f(z) = z + \rho(z - \z_1)(z - \z_2) \cdots (z - \z_d),
 \end{equation*}
 where $\z_1,\ldots,\z_d$ and $\rho$ are complex numbers with $\rho \ne 0$.
 Then for $\la = (\la_1, \ldots, \la_d) \in V_d$,
 the equalities $f'(\z_i) = \la_i$ hold for $1 \le i \le d$
 if and only if the equalities
 \begin{equation}\label{eq.2.2}
  \sum_{i=1}^d \frac{1}{1 - \la_i} \, \z_i^k
   = \begin{cases}
      0            & (1 \le k \le d-2) \\
      -\frac{1}{\rho} & (k = d-1)
     \end{cases}
 \end{equation}
 hold and $\z_1, \ldots, \z_d$ are mutually distinct.
\end{keylemma}

\begin{remark}
 Similar result to \key{} is already given by Fujimura in~\cite[Lemma 9]{fu},
 while her proof is different from the following.
\end{remark}

\begin{proof}
 The integration
 $\frac{1}{2\pi\sqrt{-1}}\oint_{|z| = R}\frac{z^k}{z - f(z)}\,dz$
 for large real number $R$
 implies the equalities
 \begin{equation}\label{eq.2.4}
  \sum_{i=1}^d \frac{1}{1 - f'(\z_i)} \, \z_i^k
   = \begin{cases}
      0            & (0 \le k \le d-2) \\
      -\frac{1}{\rho} & (k = d-1)
     \end{cases}
 \end{equation}
 if $\z_1, \ldots, \z_d$ are mutually distinct.
 Since $\la_i\ne 1$ for $1\le i\le d$,
 the equalities $f'(\z_i) = \la_i$ for $1 \le i \le d$
 imply the mutual distinctness of $\z_1, \ldots, \z_d$
 and the equalities~(\ref{eq.2.4}),
 which verifies the necessary condition of the lemma.

 Suppose oppositely
 the equalities~(\ref{eq.2.2}) and the mutual distinctness of 
 $\z_1,\ldots,\z_d$.
 Note that the equalities~(\ref{eq.2.2}) are equivalent to
 \begin{equation}\label{eq.2.6}
  \begin{pmatrix}
   1          & 1          & \cdots & 1         \\
   \z_1       & \z_2       & \cdots & \z_d      \\
   \z_1^2     & \z_2^2     & \cdots & \z_d^2    \\
   \vdots     & \vdots     & \ddots & \vdots    \\
   \z_1^{d-1} & \z_2^{d-1} & \cdots & \z_d^{d-1}
  \end{pmatrix}
  \begin{pmatrix}
   \frac{1}{1 - \la_1} \\
   \frac{1}{1 - \la_2} \\
   \vdots                 \\
   \frac{1}{1 - \la_d}
  \end{pmatrix}
  = \begin{pmatrix}
     0 \\ \vdots \\ 0 \\ -\frac{1}{\rho}
    \end{pmatrix}.
 \end{equation}
 The mutual distinctness of $\z_1,\ldots,\z_d$
 implies~(\ref{eq.2.4}),
 which are equivalent to the equality
 obtained from~(\ref{eq.2.6})
 by replacing $\la_i$ by $f'(\z_i)$ for $1\le i \le d$.
 Therefore the invertibility of the square matrix
 in the left hand side of the equality~(\ref{eq.2.6})
 implies $\frac{1}{1-f'(\z_i)} = \frac{1}{1-\la_i}$ for $1 \le i \le d$,
 which completes the proof of \key.
\end{proof}

\setcounter{theorem}{1}

By \key, we associate the set
$\Phi_d^{-1}\left(\bar{\la}\right)$ 
with some other one whose cardinality is expected to be easier to count.
Recall that $\mb{P}^{d-2}$ denotes the complex 
projective space of dimension $d-2$.

\begin{definition}\label{df.2.2}
 For any $\la = (\la_1,\ldots, \la_d)\in V_d$, we put
 \begin{align*}
  T_d(\la) &:= \left\{(\z_1: \cdots : \z_{d-1}) \in \Pd \ \left| \
		\sum_{i=1}^{d-1}\frac{1}{1-\la_i}\,\z_i^k = 0
               \ \text{ for } 1\le k\le d-2 \right. \right\}, \\
  S_d(\la) &:= \left\{(\z_1: \cdots : \z_{d-1}) \in T_d(\la) \bigm|
               \textrm{$\z_1, \ldots, \z_{d-1}$ and $0$ are 
               mutually distinct} \right\}, \\
  B_d(\la) &:= T_d(\la) \setminus S_d(\la) \quad \text{and} \\
  \mathfrak{S}\left(\mathcal{K}(\la)\right) &:=
   \left\{ \sigma \in \mathfrak{S}_d \bigm|
    \la_{\sigma(i)}=\la_i \ \text{holds for any $i$.} \right\}.
 \end{align*}
\end{definition}

Note that $\mathfrak{S}\left(\mathcal{K}(\la)\right)$ is a subgroup of 
$\mathfrak{S}_d$
determined by $\mathcal{K}(\la)$ and is isomorphic to the group
$\mathfrak{S}_{\kappa_1} \times \cdots \times\mathfrak{S}_{\kappa_q}$,
where $\kappa_1,\ldots,\kappa_q$ and $K_1, \ldots, K_q$ are those
defined in Definition~\ref{df.1.6}.

\begin{proposition}\label{pr.2.3}
 For $\la = (\la_1,\ldots, \la_d) \in V_d$,
 we can define the surjection
  $\pi(\la) : S_d(\la) \to \Phi_d^{-1}\left(\bar{\la}\right)$
 by
 \begin{equation*}
  (\z_1: \cdots : \z_{d-1}) \mapsto f(z) = z+ \rho z(z-\z_1)\cdots(z-\z_{d-1}),
 \end{equation*}
 where 
  $-\frac{1}{\rho} = \sum_{i=1}^{d-1}\frac{1}{1-\la_i}\,\z_i^{d-1}$.
 The group $\mathfrak{S}\left(\mathcal{K}(\la)\right)$
 acts on $S_d(\la)$ by the permutation of the coordinates 
 $\z_1, \ldots,\z_{d-1}$ and $0$, 
 namely, it is defined by
 \[
  \sigma \cdot (\z_1 : \cdots : \z_{d-1}) :=
  (\z_{\sigma^{-1}(1)}-\z_{\sigma^{-1}(d)} : \cdots :
   \z_{\sigma^{-1}(d-1)}-\z_{\sigma^{-1}(d)}),
 \]
 where we are assuming $\z_d = 0$.
 Finally
 the map $\pi(\la) : S_d(\la) \to \Phi_d^{-1}\left(\bar{\la}\right)$
 induces the bijection
 \begin{equation*}
  \overline{\pi(\la)} : S_d(\la) / \mathfrak{S}\left(\mathcal{K}(\la)\right)
   \stackrel{\cong}{\to} \Phi_d^{-1}\left(\bar{\la}\right).
 \end{equation*}
\end{proposition}

To prove Proposition~\ref{pr.2.3},
we consider the auxiliary definitions, lemma and proposition.

\begin{definition}
 We put
 \[
  Q_d(\la) 
   := \left\{ (\z_1,\ldots,\z_d) \in \mb{C}^d \
        \left| \ \begin{matrix}
		  \sum_{i=1}^d \frac{1}{1-\la_i}\,\z_i^k = 0 \ \
		   \textrm{for} \ 1\le k\le d-2 \\
		  \z_1,\ldots,\z_d \textrm{ are mutually distinct}
		 \end{matrix}
        \right. \right\},
 \]
 and denote by $G$
 the projection map 
 $G : \mathrm{Poly}_d \to \mathrm{MP}_d = \mathrm{Poly}_d/\mathrm{Aut}(\mathbb{C})$,
 where $\mathrm{Aut}(\mathbb{C})$ and its action on $\mathrm{Poly}_d$ are those 
 defined in~(\ref{eq.1.1}).
\end{definition}

The groups $\mathrm{Aut}(\mathbb{C}), \mathfrak{S}_d$
and its subgroup $\mathfrak{S}\left(\mathcal{K}(\la)\right)$
naturally act on $\mb{C}^d$, and
the actions of $\mathrm{Aut}(\mathbb{C})$ and $\mathfrak{S}_d$ on $\mb{C}^d$ commute.

\begin{lemma}\label{lm.2.5}
 Let $\la = (\la_1,\ldots, \la_d)$ be an element of $V_d$.  Then
 \begin{enumerate}
  \item  we can define the map
	 $\varpi(\la):Q_d(\la)\to G^{-1}\circ\Phi_d^{-1}\left(\bar{\la}\right)$
	 by
	 \[
	 (\z_1,\ldots,\z_d) \mapsto f(z):= z + \rho (z-\z_1)\cdots(z-\z_d),
	 \]
	 where 
	  $-\frac{1}{\rho} = \sum_{i=1}^{d}\frac{1}{1-\la_i}\,\z_i^{d-1}$.
	 \label{en.2.5.1}
  \item The map $\varpi(\la)$ is surjective.\label{en.2.5.2}
  \item The set $Q_d(\la)$ is invariant
	 under the action of $\mathrm{Aut}(\mathbb{C})$ on $\mb{C}^d$.
	 \label{en.2.5.3}
  \item The actions of $\mathrm{Aut}(\mathbb{C})$ on $Q_d(\la)$
	 and on $G^{-1}\circ\Phi_d^{-1}\left(\bar{\la}\right)$
	 commute with the map $\varpi(\la)$, i.e., the equality
	$\varpi(\la)(\gamma\cdot\z)=\gamma\circ\varpi(\la)(\z)\circ\gamma^{-1}$
	 holds for any $\z \in Q_d(\la)$ and $\gamma \in \mathrm{Aut}(\mathbb{C})$.
	 \label{en.2.5.4}
  \item The set $Q_d(\la)$ is invariant under the action of 
	 $\mathfrak{S}\left(\mathcal{K}(\la)\right)$ on $\mb{C}^d$.
	 \label{en.2.5.5}
  \item For $\z, \z' \in Q_d(\la)$,
	 the equality $\varpi(\la)(\z) = \varpi(\la)(\z')$ holds
	 if and only if the equality $\z' = \sigma \cdot \z$ holds
	 for some $\sigma \in \mathfrak{S}\left(\mathcal{K}(\la)\right)$.
	 \label{en.2.5.6}
 \end{enumerate}
\end{lemma}
\begin{proof}
 Most of the assertions are obvious by Key Lemma.
 We only check the existence of the complex number $\rho$ 
 in the assertion~(\ref{en.2.5.1})
 and the necessary condition of the assertion~(\ref{en.2.5.6}).

 If we cannot determine $\rho \in \mb{C}^*$,
 then we have $\sum_{i=1}^{d}\frac{1}{1-\la_i}\,\z_i^{d-1}=0$,
 which implies $\frac{1}{1-\la_i} = 0$ for $1\le i\le d$
 by the equality~(\ref{eq.2.6}).
 Hence the contradiction assures the existence of $\rho$.

 Let $\z=(\z_1,\ldots,\z_d), \z'=(\z'_1,\ldots,\z'_d)$
 be elements of $Q_d(\la)$ with $\varpi(\la)(\z) =\varpi(\la)(\z')=:f$.
 Then by the definition of $\varpi(\la)$,
 there exists a permutation $\sigma \in \mathfrak{S}_d$
 with $\z'=\sigma\cdot\z$.
 On the other hand, by \key,
 we have $f'(\z_i)=f'(\z'_i)=\la_i$ for $1\le i\le d$.
 Since $\z'_i=\z_{\sigma^{-1}(i)}$ for $1\le i\le d$,
 we have $\la_i=\la_{\sigma(i)}$ for $1\le i\le d$,
 which implies $\sigma\in \mathfrak{S}\left(\mathcal{K}(\la)\right)$.
 Thus the necessary condition of~(\ref{en.2.5.6}) is verified.
\end{proof}

\begin{definition}
 We put $\widetilde{Q}_d(\la) := Q_d(\la) / \mathrm{Aut}(\mathbb{C})$.
\end{definition}

\begin{proposition}\label{pr.2.8}
 For $\la = (\la_1,\ldots, \la_d) \in V_d$,
 the map $\varpi(\la)$ in Lemma~\ref{lm.2.5} induces the surjection
  $\widetilde{\varpi}(\la) : \widetilde{Q}_d(\la) \to
   \Phi_d^{-1}\left(\bar{\la}\right)$.
 The group $\mathfrak{S}\left(\mathcal{K}(\la)\right)$ acts
 on $\widetilde{Q}_d(\la)$,
 which induces the bijection
 \[
  \overline{\varpi(\la)} : \widetilde{Q}_d(\la) /
   \mathfrak{S}\left(\mathcal{K}(\la)\right)
   \to \Phi_d^{-1}\left(\bar{\la}\right).
 \]
 Moreover $\widetilde{Q}_d(\la)$ is canonically identified with $S_d(\la)$
 by the bijection $\iota(\la):S_d(\la) \to \widetilde{Q}_d(\la)$
 which maps $(\z_1:\cdots:\z_{d-1}) \in S_d(\la)$ to 
 the equivalence class of $(\z_1,\ldots,\z_{d-1},0)$ in $\widetilde{Q}_d(\la)$.
 Under this identification,
 $\widetilde{\varpi}(\la)\circ\iota(\la) = \pi(\la)$ holds, and
 the actions of $\mathfrak{S}\left(\mathcal{K}(\la)\right)$
 on $S_d(\la)$ and on $\widetilde{Q}_d(\la)$ commute with the map $\iota(\la)$.
\end{proposition}

\begin{proof}[Proof of Propositions~\ref{pr.2.8} and~\ref{pr.2.3}]
 Proposition~\ref{pr.2.8} is a direct consequence of Lemma~\ref{lm.2.5},
 whereas
 Proposition~\ref{pr.2.3} is just a corollary of Proposition~\ref{pr.2.8}. 
\end{proof}

We make use of
the bijection $\iota(\la) : S_d(\la) \cong \widetilde{Q}_d(\la)$
in the proof;
in the process of determining the cardinality $\#\left(S_d(\la)\right)$
we consider only $S_d(\la)$,
while
we utilize $\widetilde{Q}_d(\la)$
in the process of analyzing the action
of $\mathfrak{S}\left(\mathcal{K}(\la)\right)$ on $S_d(\la)$.

\begin{proposition}\label{pr.mt5}
  The assertion~(\ref{en.mt5}) in Main Theorem~\ref{mthm.1} holds.
\end{proposition}
\begin{remark}
 As already mentioned in Remark~\ref{rem.1.3}, Fujimura~\cite[Theorem 12]{fu} proved a weaker 
 statement 
 of 
 Proposition~\ref{pr.mt5}, while 
 her proof is very similar to the following.
\end{remark}
\begin{proof}
 Since the
 map $G\circ\varpi(\la):Q_d(\la)\to\Phi_d^{-1}\left(\bar{\la}\right)$
 is surjective,
 it suffices to show that the set $Q_d(\la)$ is empty.
 Note first that the conditions $\sum_{i=1}^d(1-\la_i)^{-1}=0$ and $c_1(1-\lambda_1) = \cdots = c_d(1-\lambda_d)$
 imply $\sum_{i=1}^d c_i=0$.
 We may assume that the integers $c_1,\ldots,c_j$ are positive
 and that the 
 rests
 are negative.
 Then the condition $\sum_{i=1}^d \left|c_i \right| \le 2(d-2)$ is  equivalent 
 to $\sum_{i=1}^j c_i = \sum_{i=j+1}^d -c_i \le d-2$, and 
 the defining equations
 $\sum_{i=1}^d \frac{1}{1-\la_i} \z_i^k = 0$ for $1\le k \le d-2$
 are equivalent to the equations
 \[
  \underbrace{\z_1^k + \cdots + \z_1^k}_{c_1} + \cdots +
  \underbrace{\z_j^k + \cdots + \z_j^k}_{c_j} = 
  \underbrace{\z_{j+1}^k + \cdots + \z_{j+1}^k}_{-c_{j+1}} + \cdots +
  \underbrace{\z_d^k + \cdots + \z_d^k}_{-c_d}
 \]
 for $1\le k \le d-2$.
 Hence the $k$-th fundamental symmetric expressions of
 \begin{equation}\label{eq.mt5.1}
  \underbrace{\z_1, \ldots, \z_1}_{c_1}, \ldots,
  \underbrace{\z_j, \ldots, \z_j}_{c_j}
  \quad \text{and} \quad
  \underbrace{\z_{j+1}, \ldots, \z_{j+1}}_{-c_{j+1}}, \ldots,
  \underbrace{\z_d, \ldots, \z_d}_{-c_d}
 \end{equation}
 coincide for $1\le k \le d-2$.
 Therefore
 the condition $\sum_{i=1}^j c_i = \sum_{i=j+1}^d -c_i \le d-2$
 assures that the left half of (\ref{eq.mt5.1})
 is some permutation of the right half of (\ref{eq.mt5.1}),
 which contradicts the mutual distinctness of $\z_1,\ldots,\z_d$.
 Thus the set $Q_d(\la)$ is empty. 
\end{proof}

\begin{proposition}\label{pr.mt7}
  The assertion~(\ref{en.mt7}) in Main Theorem~\ref{mthm.1} holds.
\end{proposition}
\begin{proof}
 To prove the proposition, we may assume that
 $c_1,\dots, c_d$ is a permutation of\\
 $1, -1, \underbrace{2,\dots,2}_{\frac{d}{2}-1}, \underbrace{-2,\dots,-2}_{\frac{d}{2}-1}$
 or $1, 1, \underbrace{2,\dots,2}_{\frac{d-3}{2}}, \underbrace{-2,\dots,-2}_{\frac{d-1}{2}}$ according as $d$ is even or odd.
 
 Let $U_{d-2}(z)$ be Chebyshev polynomial of the second kind of degree $d-2$.
 By definition, 
 $U_{d-2}(z)$ is a polynomial of degree $d-2$ satisfying the equality $U_{d-2}(\cos\theta) = \sin (d-1)\theta/\sin\theta$.
 Put $f_d(z) = z + \rho (z^2-1)U_{d-2}(z)$ for $\rho \in \mathbb{C} \setminus \{0\}$.
 Then we have $\mathrm{Fix}(f_d) = \left\{ \cos(k\pi/(d-1)) \mid k=0, 1,\dots, d-1 \right\}$.
 Moreover by a direct calculation we have $f'_d(1) = 1+\rho\cdot 2(d-1),\ f'_d(-1) = 1+\rho\cdot 2(-1)^{d-1}(d-1)$ and
 $f'_d\left(\cos\frac{k\pi}{d-1}\right) = 1 + \rho\cdot (-1)^k(d-1)$ for $1 \le k \le d-2$.
 Hence for any $\lambda \in V_d$ with $c_1(1-\lambda_1) = \dots = c_d(1-\lambda_d)$,
 we have $\Phi_d(f_d) = \bar{\lambda}$ for suitable $\rho$.
\end{proof}

\begin{remark}
 In practice, for any $d$, a similar computation to Example~\ref{ex.0.3} in Section~\ref{sec.1.1.2} assures
 the equality $\#\left( \Phi_d^{-1}(\bar{\lambda}) \right) = 1$
 for $\la \in V_d$ given in the proof of Proposition~\ref{pr.mt7}. Hence the unique element of $\Phi_d^{-1}(\bar{\lambda})$
 is represented by the above $f_d(z)$ for any $d$.
\end{remark}

\section{Review of the intersection theory on $\mb{P}^n$}\label{sec.3}

This section
summarizes the facts about the intersection theory on $\mb{P}^n$,
and states extended Bezout's theorem in 
Proposition~\ref{pr.4.1}.
For detailed explanation of the basic knowledge of this section, 
see Section 4 of Chapter 0 and Section 3 of Chapter 1 in~\cite{GriffithsHarris}.

Let $C$ be an algebraic variety of dimension $k$ in $\mathbb{P}^n$.
Then 
generic $(n-k)$-plane $\mb{P}^{n-k} \subset \mb{P}^n$ intersects $C$ 
transversely; we may thus define the degree of $C$ to be the number of
intersection points of $C$ with a generic linear subspace $\mb{P}^{n-k}$,
which does not depend on the choice of $\mb{P}^{n-k}$.
For example,
for any homogeneous polynomial $\varphi(\z)$ of degree $d$ on $\mb{P}^n$,
the degree of the zeros of $\varphi$ is always $d$.

Secondly we remember the definition of the intersection multiplicity
$\mult_{C_{\mu}}(C,C')$ of varieties $C$ and $C'$ in $\mb{P}^n$
along an irreducible component $C_{\mu}$ of $C\cap C'$
with $\dim C_{\mu} = \dim C + \dim C' -n$.
If $C_{\mu}$ is a point,
then the intersection multiplicity is defined as follows:
in a local coordinate having the origin as $C_{\mu}$,
$C$ meets $C'+\epsilon$ transversely around the origin
for generic small $\epsilon\in\mb{C}^n$,
where $C'+\epsilon$ denotes the translation of $C'$ by $\epsilon$
with respect to the given local coordinate;
we may thus define the intersection multiplicity $\mult_{C_{\mu}}(C,C')$
to be the number of intersection points of $C$ and $C'+\epsilon$
around the origin for sufficiently small generic $\epsilon$,
which does not depend on the choice of $\epsilon$ nor a local coordinate.
In the general case with $\dim C_{\mu} = \dim C + \dim C' -n$,
the intersection multiplicity $\mult_{C_{\mu}}(C,C')$ is defined
to be the number $\mult_p(C\cap H, C'\cap H)$
on $H$, where $p$ is a generic smooth point of $C_{\mu}$ and
$H$ is a submanifold in a neighborhood of $p$ intersecting $C_{\mu}$
transversely at $p$ and with complementary dimension of $C_{\mu}$.

Next we state the relation among the intersection multiplicities defined above.
Let $C,C'$ be algebraic varieties in $\mb{P}^n$ with $\dim C =k$ and
$\dim C' = k'$, and $C_1,\ldots, C_r$ the irreducible components of $C\cap C'$.
Suppose that the equality
$\dim C_{\mu} = \dim C + \dim C' -n$ holds for any $\mu$.
Then the topological intersection of $C$ and $C'$ is given by
 $\left(C\cdot C'\right) = \sum_{\mu=1}^r \mult_{C_{\mu}}(C,C')\cdot C_{\mu}$,
which implies the equality
\begin{equation}\label{eq.4.0}
 \deg C\cdot\deg C' = \sum_{\mu=1}^r \mult_{C_{\mu}}(C,C') \cdot \deg C_{\mu}.
\end{equation}

On the basis of those mentioned above,
we state Definitions~\ref{df.4.0},~\ref{df.4.1}
and Proposition~\ref{pr.4.1}.

\begin{definition}~\label{df.4.0}
 We define the family $\mathcal{C}(\varphi_1, \ldots, \varphi_m)$
 for homogeneous polynomials $\varphi_1, \ldots, \varphi_m$ on $\mb{P}^n$
 inductively as follows:
 if $m=1$, then $\mathcal{C}(\varphi_1)$ is the family
 of the irreducible components of the zeros of $\varphi_1$ in $\mb{P}^n$;
 in the case $m\ge 2$, putting
 \[
 \mathcal{C}'
  := \left\{C'\in\mathcal{C}(\varphi_1, \ldots, \varphi_{m-1}) \bigm|
        C' \subseteq \{\varphi_m=0\} \right\} \quad \text{and} \quad
 \mathcal{C}'':=\mathcal{C}(\varphi_1,\ldots,\varphi_{m-1})
  \setminus\mathcal{C}',
 \]
 we define the family $\mathcal{C}(\varphi_1, \ldots, \varphi_m)$
 by
 \[
  \mathcal{C}(\varphi_1, \ldots, \varphi_m) := \mathcal{C}' \cup 
   \bigcup_{C''\in\mathcal{C}''}
    \left\{C \bigm| C \text{ is an irreducible component of }
     C''\cap\{\varphi_m=0\}\right\}.
 \]
\end{definition}

By definition, a variety $C$ in $\mb{P}^n$ is
an irreducible component of the common zeros of $\varphi_1,\ldots,\varphi_m$
if and only if 
$C$ is a maximal element of $\mathcal{C}(\varphi_1,\ldots,\varphi_m)$
with respect to the inclusion relations.
Making use of the family $\mathcal{C}(\varphi_1,\ldots,\varphi_m)$,
we are able to consider ``components'' of the common zeros which are
proper subsets of some irreducible components of the common zeros.

\begin{definition}\label{df.4.1} 
 We shall define the number $\mult_C(\varphi_1, \ldots,\varphi_m)$
 for homogeneous polynomials $\varphi_1, \ldots, \varphi_m$ on $\mb{P}^n$
 and an irreducible variety $C$ in $\mb{P}^n$ with $\mathrm{codim}\, C=m$.
 If $C\notin \mathcal{C}(\varphi_1, \ldots,\varphi_m)$,
 then we put $\mult_C(\varphi_1,\ldots,\varphi_m)=0$;
 if $C \in \mathcal{C}(\varphi_1, \ldots,\varphi_m)$,
 we define $\mult_C(\varphi_1, \ldots,\varphi_m)$
 by induction of $m$ in the following manner:
 if $m=1$, then the number $\mult_C(\varphi_1)$ is
 the usual order of zeros of $\varphi_1$ along $C$;
 in the case $m\ge 2$, the number $\mult_C(\varphi_1, \ldots,\varphi_m)$ is
 defined 
by the equality
 \begin{equation}\label{eq.4.1}
   \mult_C(\varphi_1,\ldots,\varphi_m) =
    \sum_{C' \in \mathcal{C}_C} \mult_{C'}(\varphi_1,\ldots,\varphi_{m-1})
     \cdot \mult_C(C',\varphi_m),
 \end{equation}
 where $\mathcal{C}_C
  = \left\{C'\in\mathcal{C}(\varphi_1,\ldots,\varphi_{m-1})\bigm|
            \mathrm{codim}\,C'=m-1,\ 
            C\subseteq C',\ C'\nsubseteq \{\varphi_m=0\}\right\}$.
 Here, for a homogeneous polynomial $\varphi$, an irreducible variety $C'$ with $C' \nsubseteq \{\varphi=0\}$ 
 and an irreducible component $C$ of $C' \cap \{\varphi=0\}$,
 the number $\mult_C(C',\varphi)$ is defined by
 \[
    \mult_C(C',\varphi) := \sum_{C''\in \mathcal{C}(\varphi),\ C \subseteq C''} \mult_C(C',C'') \cdot \mult_{C''}(\varphi).
 \]
 Note that the notation $\mult_C(C',\varphi)$ is also used in the following sections.
\end{definition}

At any rate, Definition~\ref{df.4.1} assigns a positive integer
$\mult_C(\varphi_1, \ldots,\varphi_m)$
to each $C\in \mathcal{C}(\varphi_1, \ldots,\varphi_m)$
with $\mathrm{codim}\, C = m$.
By definition, if $C$ is an irreducible component of the common zeros 
of $\varphi_1, \ldots, \varphi_m$ with $\mathrm{codim}\, C = m$,
then the number $\mult_C(\varphi_1, \ldots,\varphi_m)$ defined above
is the usual intersection multiplicity
of $\varphi_1, \ldots, \varphi_m$ along $C$.
We state the relation among the numbers defined above
in Proposition~\ref{pr.4.1}.

\begin{proposition}\label{pr.4.1}
 Let $\varphi_1,\ldots,\varphi_n$ be homogeneous polynomials on $\mb{P}^n$,
 put $\mathrm{codim}\, C =: l_C$
 for each $C\in\mathcal{C}(\varphi_1,\ldots,\varphi_n)$,
 and suppose that the inclusion relation
 \begin{equation}\label{eq.4.3.0}
 \left\{C\in\mathcal{C}(\varphi_1,\ldots,\varphi_k)\bigm|
  \mathrm{codim}\, C <k\right\}
  \subseteq \mathcal{C}(\varphi_1,\ldots,\varphi_n)
 \end{equation}
 holds for every $1 \le k \le n$.
 Then we have 
 $\left\{C\in\mathcal{C}(\varphi_1,\ldots,\varphi_n)\bigm| l_C=k\right\}
  \subseteq \mathcal{C}(\varphi_1,\ldots,\varphi_k)$ for every $1 \le k \le n$.
 Moreover we have the equality
 \begin{equation}\label{eq.4.2}
   \prod_{k=1}^n \deg \varphi_k =
    \sum_{C\in\mathcal{C}(\varphi_1,\ldots,\varphi_n)}
    \left( \deg C \cdot
      \mult_C(\varphi_1,\ldots,\varphi_{l_C}) \cdot
      \prod_{k=l_C+1}^n \deg \varphi_k
     \right).
 \end{equation}
 Here, 
 in the case $l_C=n$,
 we assume that $\prod_{k=l_C+1}^n\deg \varphi_k = \prod_{k=n+1}^n\deg \varphi_k=1$.
\end{proposition}
\begin{proof}
 We put $\mathcal{C}_k :=
  \left\{C\in\mathcal{C}(\varphi_1,\ldots,\varphi_k)\bigm|
  \mathrm{codim}\, C = k,\ C\subseteq \{\varphi_{k+1}=0\} \right\}$
 for each $1 \le k \le n-1$.
 Then by Definition~\ref{df.4.0} and the assumption~(\ref{eq.4.3.0}),
 we have $\mathcal{C}_1 \amalg \cdots \amalg \mathcal{C}_k \subseteq 
  \mathcal{C}(\varphi_1,\ldots,\varphi_k)$
 and $\left\{C\in\mathcal{C}(\varphi_1,\ldots,\varphi_k)\bigm|
 \mathrm{codim}\, C = k \right\} = \mathcal{C}(\varphi_1,\ldots,\varphi_k)
 \setminus \left(\mathcal{C}_1 \amalg \cdots \amalg \mathcal{C}_{k-1}\right)$
 for every $1 \le k \le n-1$,
 which implies the former assertion of the proposition.

 To prove the latter, it suffices to show the equality
 \stepcounter{equation}
 \begin{equation}\label{eq.4.2pr}\tag*{$(\theequation)_k$}
  \prod_{l=1}^k \deg\varphi_l
   = 
  \sum_{C\in\mathcal{C}(\varphi_1,\ldots,\varphi_k)} \left(
   \deg C \cdot \mult_C(\varphi_1,\ldots,\varphi_{l_C}) \cdot
   \prod_{l=l_C+1}^k \deg\varphi_l
   \right)
 \end{equation}
 by induction of $k$, because (\theequation)$_n$ is the same as~(\ref{eq.4.2}).
 The equality~(\theequation)$_1$ is in the form
 $\deg \varphi_1 = \sum_{C \in \mathcal{C}(\varphi_1)} \deg C \cdot \mult_C(\varphi_1)$,
 which obviously holds.
 Multiplying both sides of the equality~\ref{eq.4.2pr}
 by $\deg\varphi_{k+1}$,
 we have
 \begin{align*}
	   \prod_{l=1}^{k+1} \deg\varphi_l
   	   &= \deg \varphi_{k+1} \cdot
	  \sum_{C\in\mathcal{C}(\varphi_1,\ldots,\varphi_k)} \left(
	   \deg C \cdot \mult_C(\varphi_1,\ldots,\varphi_{l_C}) \cdot
	   \prod_{l=l_C+1}^k \deg\varphi_l
	   \right)\\
	   &= \sum_{C\in\mathcal{C}_1 \amalg \dots \amalg \mathcal{C}_k} \left(
	   \deg C \cdot \mult_C(\varphi_1,\ldots,\varphi_{l_C}) \cdot
	   \prod_{l=l_C+1}^{k+1} \deg\varphi_l
	   \right)\\
	   &{}\hspace{10pt} + 
	  \sum_{C\in\mathcal{C}(\varphi_1,\ldots,\varphi_k) \setminus \left(\mathcal{C}_1 \amalg \dots \amalg \mathcal{C}_k\right)}
	   \Bigl(
	   \deg \varphi_{k+1} \cdot \deg C \cdot \mult_C(\varphi_1,\ldots,\varphi_k)
	   \Bigr).
 \end{align*}
 We put 
 $\mathcal{C}'_k:= \mathcal{C}(\varphi_1,\ldots,\varphi_k) \setminus \left(\mathcal{C}_1 \amalg \dots \amalg \mathcal{C}_k\right)$
 for the brevity of notation.
 Then for every $C\in\mathcal{C}'_k$,
	we have $\deg \varphi_{k+1} \cdot \deg C = \sum_{\mu =1}^r \mult_{C_\mu}(C,\varphi_{k+1})\cdot \deg C_{\mu}$
	by (\ref{eq.4.0}) and by the definition of $\mult_{C_{\mu}}(C,\varphi_{k+1})$, 
	where $C_1,\dots,C_r$ are the irreducible components 
	of $C \cap \{\varphi_{k+1}=0\}$.
	Therefore, putting $\mult_{C'}(C,\varphi_{k+1})=0$ for $C'$ different from $C_1,\dots,C_r$, we have 
	\begin{equation}\label{eq.4.4}
	 \begin{split}
	   &\sum_{C\in\mathcal{C}'_k} \Bigl(
	   \deg \varphi_{k+1} \cdot \deg C \cdot \mult_C(\varphi_1,\ldots,\varphi_k)
	   \Bigr)\\
	  &= \sum_{C\in\mathcal{C}'_k} \left(
		\left(
		\sum_{C'\in\mathcal{C}(\varphi_1,\ldots,\varphi_{k+1}) \setminus 
			\left(\mathcal{C}_1 \amalg \dots \amalg \mathcal{C}_k\right)}
		  \mult_{C'}(C,\varphi_{k+1})\cdot \deg C'
		\right)
	      \cdot \mult_C(\varphi_1,\ldots,\varphi_k)
	   \right)\\
	   &= \sum_{C'\in\mathcal{C}(\varphi_1,\ldots,\varphi_{k+1}) \setminus 
			\left(\mathcal{C}_1 \amalg \dots \amalg \mathcal{C}_k\right)}
		\left( \deg C' \cdot
		 \left(
		  \sum_{C\in\mathcal{C}'_k}
		   \mult_C(\varphi_1,\ldots,\varphi_k) \cdot \mult_{C'}(C,\varphi_{k+1})
		\right) \right)\\
	   &= \sum_{C'\in\mathcal{C}(\varphi_1,\ldots,\varphi_{k+1}) \setminus 
			\left(\mathcal{C}_1 \amalg \dots \amalg \mathcal{C}_k\right)}
		\Bigl(
		   \deg C' \cdot \mult_{C'}(\varphi_1,\ldots,\varphi_{k+1}) 
		\Bigr)
	 \end{split}
	\end{equation}
	by Definition~\ref{df.4.1}.
\addtocounter{equation}{-1}\hspace{-5pt}
	To summarize, we have $(\theequation)_{k+1}$.
\addtocounter{equation}{1}
\end{proof}
Proposition~\ref{pr.4.1} is reduced to the usual Bezout's theorem
if $\mathcal{C}(\varphi_1,\ldots,\varphi_n)$ consists only of points.
Proposition~\ref{pr.4.1} is utilized crucially
for determining the cardinality $\#\left(S_d(\la)\right)$ 
in Section~\ref{sec.4}.
\begin{remark}
 The family $\mathcal{C}(\varphi_1,\ldots,\varphi_m)$ and
 the number $\mult_C(\varphi_1,\ldots,\varphi_m)$ may vary
 when the order of $\varphi_1,\ldots,\varphi_m$ changes.
 Hence Definitions~\ref{df.4.0} and~\ref{df.4.1}
 may appear to be a little strange in some sense;
 however this works very well
 for the computation of the cardinality $\#\left(S_d(\la)\right)$.
 In the following, we give an example
 in which the family $\mathcal{C}(\varphi_1,\varphi_2)$ and
 the number $\mult_{\mathrm{P}_2}(\varphi_1,\varphi_2)$
 differ from $\mathcal{C}(\varphi_2,\varphi_1)$ and
 $\mult_{\mathrm{P}_2}(\varphi_2,\varphi_1)$ respectively.
 Consider $\varphi_1=y(y-x)$ and $\varphi_2=y\left(yz^2+x^3-2x^2z\right)$
 on $\mb{P}^2=\{(x:y:z)\}$.
 We put $\mathrm{P}_1=\{(1:1:1)\}$, $\mathrm{P}_2=\{(0:0:1)\}$,
 $\mathrm{P}_3=\{(2:0:1)\}$,
 $C_0=\{y=0\}$, $C_1=\{x=y\}$ and $C_2=\{yz^2+x^3-2x^2z=0\}$.
 Then we have
 $\mathcal{C}(\varphi_1,\varphi_2)=\{C_0, \mathrm{P}_1, \mathrm{P}_2\}$
 and
 $\mathcal{C}(\varphi_2,\varphi_1)
  =\{C_0, \mathrm{P}_1, \mathrm{P}_2, \mathrm{P}_3\}$.
 Moreover we have
 $\mult_{\mathrm{P}_2}(\varphi_1,\varphi_2)=
 \mult_{C_1}(\varphi_1)\cdot\mult_{\mathrm{P}_2}(C_1,\varphi_2)=1\cdot 2=2$
 and
 $\mult_{\mathrm{P}_2}(\varphi_2,\varphi_1)
 =\mult_{C_2}(\varphi_2)\cdot\mult_{\mathrm{P}_2}(C_2,\varphi_1)=1\cdot 3=3$.
 However Proposition~\ref{pr.4.1} holds as 
 we will see
 \begin{align*}
  \deg C_0 \cdot \mult_{C_0}(\varphi_1)\cdot \deg\varphi_2
   + \deg \mathrm{P}_1 \cdot \mult_{\mathrm{P}_1}(\varphi_1,\varphi_2)
   + \deg \mathrm{P}_2 \cdot \mult_{\mathrm{P}_2}&(\varphi_1,\varphi_2) \\
  = 1\cdot 1 \cdot 4 + 1\cdot 2+ 1\cdot 2 = 8
   &= \deg\varphi_1\cdot\deg\varphi_2, \\
  \deg C_0 \cdot \mult_{C_0}(\varphi_2) \cdot \deg\varphi_1
   + \deg \mathrm{P}_1 \cdot \mult_{\mathrm{P}_1}(\varphi_2,\varphi_1)
   + \deg \mathrm{P}_2 \cdot \mult_{\mathrm{P}_2}&(\varphi_2,\varphi_1) \\
   + \deg \mathrm{P}_3 \cdot \mult_{\mathrm{P}_3}(\varphi_2,\varphi_1)
  = 1\cdot 1 \cdot 2 + 1\cdot 2 + 1\cdot 3 + 1 \cdot 1 = 8
   &= \deg\varphi_2\cdot\deg\varphi_1.
 \end{align*}
\end{remark}

\section{Outline of determining the
 cardinality $\#\left(S_d(\la)\right)$}\label{sec.4}

In this section we give an outline of determining 
the cardinality of the set $S_d(\la)$ defined in Definition~\ref{df.2.2}
for each $\la  \in V_d$.
The assertions~(\ref{en.mt1}) and~(\ref{en.mt4})
in Main Theorem~\ref{mthm.1} are also proved in this section.

For the brevity of notation we put
\[ 
   m_i := \frac{1}{1-\la_i} \quad \textrm{and} \quad
   \varphi_k(\z) := \sum_{i=1}^{d-1} m_i \z_i^k
\]
for each $i$ and $k$,
and we always assume that $\z_d = 0$.
Therefore $T_d(\la)$ is the set of the common zeros of
$\varphi_1,\ldots,\varphi_{d-2}$ in $\Pd$,
and $S_d(\la)$ consists of an
element $\z=(\z_1:\cdots:\z_{d-1})\in T_d(\la)$
with mutually distinct $\z_1, \ldots, \z_{d-1}$ and $\z_d$.
Moreover we may also consider
that
 $\varphi_k(\z) = \sum_{i=1}^d m_i \z_i^k$.

\begin{lemma}\label{lm.4.1}
 Let $\la$ be an element of $V_d$.
 Then $S_d(\la)$ is discrete in $\Pd$.
 Moreover
 we always have
  $\mult_{\z_0}\left(\varphi_1,\ldots,\varphi_{d-2}\right) = 1$
 for any $\z_0 \in S_d(\la)$.
\end{lemma}
\begin{proof}
 We consider the row vectors
  $\frac{\partial \varphi_k}{\partial \z} =
   \left(\frac{\partial \varphi_k}{\partial \z_1},\ldots,
    \frac{\partial \varphi_k}{\partial \z_{d-1}} \right) =
   \left(km_1 \z_1^{k-1},\ldots,km_{d-1} \z_{d-1}^{k-1}\right)$
 at $\z = \z_0 \in S_d(\la)$ for $1\le k \le d-1$.
 Since $\z_1,\ldots,\z_{d-1}$ are mutually distinct at $\z = \z_0$ and 
 since $m_i \ne 0$
 for any $i$, we have
 \[
  \det \sideset{^t}{}{\mathop{\left(
   \sideset{^t}{}{\mathop{\left(
     \frac{\partial\varphi_1}{\partial\z}\right)}},
    \ldots,
   \sideset{^t}{}{\mathop{\left(
     \frac{\partial\varphi_{d-1}}{\partial\z}\right)}}
        \right)}}
  = (d-1)! \cdot \prod_{i=1}^{d-1} m_i \cdot
  \det \begin{pmatrix}
	1          & \cdots & 1              \\
	\z_1       & \cdots & \z_{d-1}       \\
	\vdots     & \ddots & \vdots         \\
	\z_1^{d-2} & \cdots & \z_{d-1}^{d-2}
       \end{pmatrix}
  \ne 0.
 \]
 Therefore the row vectors
  $\frac{\partial \varphi_1}{\partial \z}, \ldots,
  \frac{\partial \varphi_{d-2}}{\partial \z}$
 are linearly independent at $\z = \z_0$,
 which proves the lemma.
\end{proof}

\begin{proposition}\label{pr.mt1}
 The assertion~(\ref{en.mt1}) in Main Theorem~\ref{mthm.1} holds
\end{proposition}
\begin{proof}
 Since the map $\pi(\la): S_d(\la) \to \Phi_d^{-1}\left(\bar{\la}\right)$ is surjective,
 it suffices to show the inequality $\#\left(S_d(\la)\right) \le (d-2)!$ for $\la \in V_d$.
 The following argument is similar to the proof of Proposition~\ref{pr.4.1}.

 Note first that if $C, C'$ are irreducible varieties in $\mathbb{P}^{d-2}$ 
 with $\mathrm{codim}\, C' = 1$ and $C \nsubseteq C'$,
 then all the irreducible components of $C \cap C'$ have codimension $\mathrm{codim}\, C + 1$.
 Hence a component $C \in \mathcal{C}(\varphi_1, \dots, \varphi_k)$ with $\mathrm{codim}\, C < k$ does not ``generate" 
 any elements of $S_d(\la)$ 
 since all the components of $S_d(\la)$ have ``$\mathrm{codim} = d-2$" by the discreteness of $S_d(\la)$.
 Therefore putting $\mathcal{C}_1 = \mathcal{C}(\varphi_1)$ and 
 \begin{align*}
	   \mathcal{C}'_k &:= \{C \in \mathcal{C}_{k} \mid C \nsubseteq \{\varphi_{k+1} = 0\} \}, \\
	   \mathcal{C}_{k+1} &:= \bigcup_{C' \in \mathcal{C}'_k} \left\{ C \mid 
	   C \textrm{ is an irreducible component of } C' \cap \{ \varphi_{k+1}=0\}\right\}
 \end{align*}
 inductively, we have $\mathrm{codim}\, C = k$ for every $C \in \mathcal{C}_k$ and also have
 $\left\{ \{\zeta\} \mid \zeta \in S_d(\la)  \right\} \subseteq \mathcal{C}_{d-2}$.
 Here, note that $\mathcal{C}_k$ and $\mathcal{C}'_k$ above are different from those 
 defined in the proof of Proposition~\ref{pr.4.1}.
 Applying the equalities~(\ref{eq.4.0}) and~(\ref{eq.4.1}) repeatedly, we have 
 \begin{align*}
  \deg \varphi_1 &= \sum_{C_1 \in \mathcal{C}_1} \mult_{C_1}(\varphi_1)\cdot \deg C_1\\
  \deg \varphi_2 \cdot \sum_{C_1 \in \mathcal{C}'_1} \mult_{C_1}(\varphi_1)\cdot \deg C_1
		&= \sum_{C_2 \in \mathcal{C}_2} \mult_{C_2}(\varphi_1,\varphi_2)\cdot \deg C_2\\
  \deg \varphi_3 \cdot \sum_{C_2 \in \mathcal{C}'_2} \mult_{C_2}(\varphi_1,\varphi_2)\cdot \deg C_2
		&= \sum_{C_3 \in \mathcal{C}_3} \mult_{C_3}(\varphi_1,\varphi_2,\varphi_3)\cdot \deg C_3\\
		\vdots\\
  \deg\varphi_{d-2}\cdot\sum_{C_{d-3}\in\mathcal{C}'_{d-3}} \mult_{C_{d-3}}(\varphi_1,\dots,&\varphi_{d-3}) \cdot\deg C_{d-3}\\
		= \sum_{C_{d-2} \in \mathcal{C}_{d-2}} &\mult_{C_{d-2}}(\varphi_1,\dots,\varphi_{d-2})\cdot \deg C_{d-2}
 \end{align*}
 by a similar calculation to~(\ref{eq.4.4}).
 Hence we have
 \[
   \prod_{k=1}^{d-2} \deg \varphi_k
    \geq \sum_{C_{d-2} \in \mathcal{C}_{d-2}} \mult_{C_{d-2}}(\varphi_1,\dots,\varphi_{d-2})\cdot \deg C_{d-2}
    \geq \#\left( \mathcal{C}_{d-2} \right)
    \geq \#\left( S_d(\la) \right).
 \]
 Since $\deg\varphi_k = k$,
 we have $\#\left( S_d(\la) \right) \leq \prod_{k=1}^{d-2} \deg\varphi_k = \prod_{k=1}^{d-2} k = (d-2)!$,
 which completes the proof of the proposition
\end{proof}

Proposition~\ref{pr.4.1} and Lemma~\ref{lm.4.1} imply that
in order to determine the cardinality $\#\left(S_d(\la)\right)$, 
we only need to find the degree $\deg C$ and 
the number $\mult_C(\varphi_1,\ldots,\varphi_{d-l})$
for each $2 \le l \le d-1$
and $C\in\mathcal{C}(\varphi_1,\ldots,\varphi_{d-2})$ 
with $\dim C = l-2$ included in $B_d(\la)$.
To state the explicit expression of the set $B_d(\la)$,
we shall make a definition of $E_d(\mb{I})$ for each $\mb{I} \in \mathfrak{I}(\la)$.
Recall the definition of $\mathfrak{I}(\la)$ for $\la \in V_d$
defined in Definition~\ref{df.1.6}.

\begin{definition}\label{df.4.2}
 Let $\la$ be an element of $V_d$.
 For each $\mb{I}=\{I_1,\ldots,I_l\} \in \mathfrak{I}(\la)$,
 we define the subset $E_d(\mb{I})$ of $\Pd$ by
 \[
  E_d(\mb{I}) := \left\{ (\z_1:\cdots:\z_{d-1})\in\Pd \ \left| \
   \begin{matrix}
    \textrm{If $i,j \in \{1,\ldots,d\}$ belong to
       the same $I_u$} \\
    \textrm{for some $u$, then $\z_i=\z_j$ holds}.
   \end{matrix}
   \right.\right\}.
 \]
\end{definition}
In the definition of $E_d(\mb{I})$, we are assuming $\z_d=0$.
By definition, the relation $\mb{I} \prec \mb{I}'$ holds for
$\mb{I},\mb{I}' \in \mathfrak{I}(\la)$
if and only if
$E_d(\mb{I})\subseteq E_d(\mb{I}')$ holds.
Moreover if $\#\left(\mb{I}\right) = l$,
then $E_d\left(\mb{I}\right)$ is an $(l-2)$-dimensional complex plane
in $\Pd$; hence the degree of $E_d\left(\mb{I}\right)$ is always $1$.
To help the reader to understand the definition of $E_d(\mb{I})$,
we give an example.

\begin{example}\label{ex.1}
 Let us consider again $\la \in V_6$ with $m_1:\cdots :m_6 = 1:1:2:-1:-1:-2$ introduced in Example~\ref{ex.0.1}.
 The notation follows that in Example~\ref{ex.0.1}.
 In this case, we have
 \begin{align*}
  E_6(\mb{I}_1) &= \left\{(\z_1:\z_2:0:\z_1:\z_2) \in \mb{P}^4 \ \left| \
   (\z_1:\z_2) \in \mb{P}^1  \right.\right\}, \\
  E_6(\mb{I}_2) &= \left\{(\z_1:\z_2:0:\z_2:\z_1) \in \mb{P}^4 \ \left| \
   (\z_1:\z_2) \in \mb{P}^1  \right.\right\}, \\
  E_6(\mb{I}_3) &= \{(1:1:0:1:1)\}, \
  E_6(\mb{I}_4)  = \{(1:0:0:1:0)\}, \
  E_6(\mb{I}_5)  = \{(0:1:0:0:1)\}, \\
  E_6(\mb{I}_6) &= \{(1:0:0:0:1)\}, \
  E_6(\mb{I}_7)  = \{(0:1:0:1:0)\} \text{ and }
  E_6(\mb{I}_8)  = \{(0:0:1:1:1)\}.
 \end{align*}
 $E_6 (\mb{I}_1)$ and $E_6 (\mb{I}_2)$ are complex lines in $\mb{P}^4$,
 whereas $E_6 (\mb{I}_{\omega})$ are points for $3 \le\omega \le 8$.
 We have $E_6 (\mb{I}_{\omega})\subset E_6 (\mb{I}_1)$
 for $\omega = 3,4$ and $5$,
 and $E_6 (\mb{I}_{\omega})\subset E_6 (\mb{I}_2)$
 for $\omega = 3,6$ and $7$.
\end{example}

\begin{remark}\label{rm.4.4}
 Since we always have the equality $\sum_{i=1}^d m_i =0$,
 we have
 \[ \textstyle
 \mathfrak{I}(\la) = \left\{ \mb{I} \subseteq \mathcal{I}(\la) \bigm|
                     \coprod_{I\in\mb{I}}I = \{1,\ldots,d\} \right\}
 \ \text{ and } \
  \mathcal{I}(\la)
   = \bigcup_{\mb{I} \in \mathfrak{I}(\la)} \mb{I}.  
 \]
 Hence $\mathfrak{I}(\la)$ gives the equivalent
 information as $\mathcal{I}(\la)$.
\end{remark}

Now we are in a position to give the explicit expression of the set $B_d(\la)$.

\begin{lemma}\label{lm.4.6}
 Let $\la$ be an element of $V_d$.
 Then we have the equality
 \begin{equation}\label{eq.4.6}
  B_d(\la) = \bigcup_{\mb{I} \in \mathfrak{I}(\la)} E_d(\mb{I}).
 \end{equation}
\end{lemma}
More strictly, $B_d(\la)$ is a union of $E_d(\mb{I})$ only for
maximal elements $\mb{I}$ of $\mathfrak{I}(\la)$ as set.
However as we will see later in Example~\ref{ex.3},
it is better to consider components $E_d(\mb{I})$ for $\mb{I}$
which are not necessarily maximal in $\mathfrak{I}(\la)$.
Note that the equality~(\ref{eq.4.6}) is only an equality as set.
\begin{proof}
 For any point $\z_0 = (\z_1: \cdots : \z_{d-1}) \in B_d(\la)$,
 we put
 \[
  \mb{I}(\z_0) := \left\{ I \subsetneq \{1,2,\ldots,d\} \ \left| \
   \begin{matrix}
    I \ne \emptyset. \qquad
    \textrm{If $i,j\in I$, then $\z_i = \z_j$.} \\
    \textrm{If $i \in I$ and $j \in \{1,2,\ldots,d\} \setminus I$,
      then $\z_i \ne \z_j$}.
   \end{matrix}
  \right.\right\},
 \]
 $\#\left(\mb{I}(\z_0)\right) =: l$,
 $\mb{I}(\z_0) =: \left\{I_1,\ldots, I_l \right\}$
 and $\alpha_u := \z_i$
 for $i\in I_u$ for each $1 \le u \le l$.
 Then by definition,
 $\{1,2,\ldots,d\}$ is a disjoint union of $I_1,\ldots, I_l$, and
 $\alpha_1,\ldots,\alpha_l$ are mutually distinct, one of which is 
 zero since $\z_d=0$ and $d \in I_u$ for some $1 \le u \le l$.
 Moreover
 since $\z_0\in B_d(\la)$,
 we have $2 \le l \le d-1$.

 Under the notation above,
 the defining equations
 $\varphi_k(\z_0) = \sum_{u=1}^l \left(\sum_{i\in I_u} m_i\right)\alpha_u^k =0$
 for $1 \le k \le d-2$
 are equivalent to the equality
 \begin{equation*}
  \begin{pmatrix}
   1              & \cdots & 1             \\
   \alpha_1       & \cdots & \alpha_l      \\
   \vdots         & \ddots & \vdots        \\
   \alpha_1^{d-2} & \cdots & \alpha_l^{d-2}
  \end{pmatrix}
  \begin{pmatrix}
   \sum_{i\in I_1} m_i \\
   \vdots              \\
   \sum_{i\in I_l} m_i
  \end{pmatrix}
  = \begin{pmatrix}
     0 \\ \vdots \\ 0
    \end{pmatrix},
 \end{equation*}
 which implies
  $\sum_{i\in I_u} m_i = 0$
 for $1 \le u \le l$ since $l \le d-1$.
 Therefore we have $\mb{I}(\z_0)\in\mathfrak{I}(\la)$
 and
 $\z_0 \in E_d\left(\mb{I}(\z_0)\right)$ for any $\z_0\in B_d(\la)$,
 which assures
  $B_d(\la) \subseteq 
   \bigcup_{\mb{I} \in \mathfrak{I}(\la)} E_d(\mb{I})$.
 The opposite inclusion relation is clear,
 which completes the proof of the lemma.
\end{proof}

\begin{proposition}\label{pr.mt4}
 The assertion~(\ref{en.mt4}) in Main Theorem~\ref{mthm.1} holds.
\end{proposition}
\begin{proof}
 By Proposition~\ref{pr.2.3}, the equality $\#\left(\Phi_d^{-1}\left(\bar{\la}\right)\right) =(d-2)!$
 holds if and only if $\#\left( S_d(\la) \right) = (d-2)!$ holds and that
 the action of $\mathfrak{S}\left(\mathcal{K}(\la)\right)$
 on $S_d(\la)$ is trivial.
 Here, if $\la_i=\la_j$ holds for some $i\ne j$, then
 the action of
 the permutation $(i,j)\in \mathfrak{S}\left(\mathcal{K}(\la)\right)$
 on $S_d(\la)$ is not trivial since $d \ge 4$.
 Hence the action of $\mathfrak{S}\left(\mathcal{K}(\la)\right)$
 on $S_d(\la)$ is trivial if and only if $\la_1,\dots,\la_d$ are mutually distinct.
 Moreover by Lemma~\ref{lm.4.6}, $\mathcal{I}(\la)$ is empty
 if and only if $B_d(\la)$ is empty.
 Hence to complete the proof of the proposition,
 we only need to show that the condition $\#\left( S_d(\la) \right) = (d-2)!$ is equivalent to 
 the condition that $B_d(\la)$ is empty.

 In the following, we use notations defined in the proof of Proposition~\ref{pr.mt1}.
 Looking at the proof of Proposition~\ref{pr.mt1} carefully, 
 we can find that the condition $\#\left( S_d(\la) \right) = (d-2)!$ is equivalent to the conditions
 \begin{equation}\label{eq.4.7}
  \mathcal{C}'_k = \mathcal{C}_k \text{ for } 1 \leq k \leq d-3 \quad \text{and}  \quad 
  \mathcal{C}_{d-2}=S_d(\la),
 \end{equation}
 since $\deg P = 1$ for a point $P$ and $\mult_{\zeta}(\varphi_1,\dots,\varphi_{d-2})=1$ for $\zeta \in S_d(\la)$.
 Here, we identify $\zeta \in S_d(\la)$ with $\{\zeta\}$ by abuse of notation.

 If the conditions $\mathcal{C}'_k = \mathcal{C}_k$ hold for every $1 \leq k \leq d-3$, 
 then the set of common zeros of $\varphi_1,\dots,\varphi_{d-2}$, 
 which we denote by $T_d(\la)$ in this paper, consists of discrete points;
 hence we have $T_d(\la) = B_d(\la) \amalg S_d(\la) = \mathcal{C}_{d-2}$. 
 Therefore in this case, $\mathcal{C}_{d-2}=S_d(\la)$
 holds if and only if $B_d(\la)$ is empty.

 On the other hand, if $\mathcal{C}'_k \subsetneq \mathcal{C}_k$ for some $1 \leq k \leq d-3$, then 
 for $C_k \in \mathcal{C}_k \setminus \mathcal{C}'_k$, all the irreducible components of 
 $C_k \cap \{\varphi_{k+2}=\dots = \varphi_{d-2}=0\}$ have codimension greater than or equal to $1$.
 Hence in this case $T_d(\la) = B_d(\la) \amalg S_d(\la)$ contains components greater than or equal to $1$. 
 Since $S_d(\la)$ consists of discrete points, $B_d(\la)$ is not empty.

 To summarize, we have shown that the condition~(\ref{eq.4.7}) is equivalent to the condition that $B_d(\la)$ is empty,
 which completes the proof of the proposition.
\end{proof}

\vspace{12pt}

In the rest of this section we give an example and some theorems
that exactly give the number $\mult_C(\varphi_1,\ldots,\varphi_{d-l})$
for each $C\in \mathcal{C}(\varphi_1,\ldots,\varphi_{d-2})$ with $\dim C = l-2$.
However their proofs,
which are the most crucial and difficult part 
in the proof of the main theorems,
will be given later in Sections~\ref{sec.5} and~\ref{sec.6}.

\begin{vartheorem}\label{thm.A}
 Let $\la$ be an element of $V_d$, and $\mb{I} = \{I_1,\ldots,I_l\}$
 a maximal element of $\mathfrak{I}(\la)$.
 Then $E_d(\mb{I})$
 is an irreducible component of the common zeros
 of $\varphi_1,\ldots, \varphi_{d-l}$ with
 its intersection multiplicity
 \[
  \mult_{E_d(\mb{I})}(\varphi_1,\ldots,\varphi_{d-l}) = 
   \prod_{u=1}^{l} \left(\#(I_u) -1\right)!.
 \]
\end{vartheorem}

\begin{example}\label{ex.3}
 We consider again $\la \in V_6$ introduced in Examples~\ref{ex.0.1} and~\ref{ex.1}.
 The notation follows that in Examples~\ref{ex.0.1} and~\ref{ex.1} again.
 In this case, we have $\Phi_6^{-1}\left(\bar{\la}\right)=\emptyset$
 by the assertion~(\ref{en.mt5}) in Main Theorem~\ref{mthm.1},
 which implies $S_6(\la) = \emptyset$.
 Hence in this example,
 we verify
 $S_6(\la) = \emptyset$ by the calculation of intersection multiplicities.

 By Example~\ref{ex.1} and Lemma~\ref{lm.4.6}, we have
  $B_6(\la) = E_6(\mb{I}_1) \cup E_6(\mb{I}_2) \cup E_6(\mb{I}_8)$
 as set. Moreover by Theorem~\ref{thm.A}, we have
  $\mult_{E_6(\mb{I}_{\omega})}(\varphi_1,\varphi_2,\varphi_3)
   = \left((2-1)!\right)^3 = 1$
 for $\omega = 1,2$, and 
 $\mult_{E_6(\mb{I}_8)}(\varphi_1,\varphi_2,\varphi_3,\varphi_4)
  = \left((3-1)!\right)^2 = 4$.
 Hence the common zeros of $\varphi_1,\varphi_2$ and $\varphi_3$
 are composed of 
 $E_6(\mb{I}_1),E_6(\mb{I}_2)$ and some curve $C$ whose degree is
  $\deg C = 3! - (1+1) = 4$.
 Moreover since
  $\deg C \cdot \deg \varphi_4 = 4 \cdot 4 = 16$,
 we have
 \[ 
  \#\left(S_6(\la)\right) = 16 - 
   \sum_{\z \in C \cap \{\varphi_4(\z) =0\} \cap B_6(\la)}
    \mult_{\{\z\}}(C,\varphi_4)
 \]
 by~(\ref{eq.4.0}).
 Here, we have
 $E_6(\mb{I}_8) \subseteq C \cap \{\varphi_4(\z) =0\} \cap B_6(\la)$
 with $\mult_{E_6(\mb{I}_8)}(C,\varphi_4) =4$.

 What occurs in the difference~``$16-4=12$''?
 It appears to be correct that $\#\left(S_6(\la)\right) = 12$;
 however this is not the case.
 In practice, the curve $C$ intersects the lines $E_6(\mb{I}_1)$ and 
 $E_6(\mb{I}_2)$.
 Precisely,
 the intersection points of the two curves $C$ and $E_6(\mb{I}_1)$
 are $E_6(\mb{I}_4)$ and $E_6(\mb{I}_5)$,
 while those of $C$ and $E_6(\mb{I}_2)$
 are $E_6(\mb{I}_6)$ and $E_6(\mb{I}_7)$;
 these four points do belong
 to the intersection $C \cap \{\varphi_4(\z) =0\} \cap B_6(\la)$.
 Moreover as we will see in Theorem~\ref{thm.B},
 we have
 $\mult_{E_6(\mb{I}_{\omega})}(C,\varphi_4)=
  \mult_{E_6(\mb{I}_{\omega})}(\varphi_1,\ldots,\varphi_4) = 3$
 for $4 \le \omega \le 7$.
 We thus have the equality
  $16 - (4 + 3 + 3 + 3 + 3) = 0$,
 which assures that $S_6(\la)$ is empty
 and that the intersection points of $C$ and $\{\varphi_4(\z)=0\}$ are
 $E_6(\mb{I}_{\omega})$ for $4 \le \omega \le 8$,
 which does not cause any contradiction.
 To summarize,
 the family $\mathcal{C}(\varphi_1,\varphi_2,\varphi_3,\varphi_4)$
 consists of $E_6(\mb{I}_{\omega})$ for $\omega=1,2,4,5,6,7$ and $8$,
 and 
 the equality
 \[
  4! - (1\cdot 4 + 1\cdot 4 + 3 + 3 + 3 +3 + 4) = 0
 \]
 implies that $S_6(\la)$ is an empty set.
 
 As a conclusion of Example~\ref{ex.3},
 we comment about the component $E_6(\mb{I}_3)$.
 The point $E_6(\mb{I}_3)$ may also appear as an element of 
 $\mathcal{C}(\varphi_1,\varphi_2,\varphi_3,\varphi_4)$.
 However by Theorem~\ref{thm.B} below, we have
  $\mult_{E_6(\mb{I}_3)}(\varphi_1,\varphi_2,\varphi_3,\varphi_4) = 0$,
 which means that in practice $E_6(\mb{I}_3)$ is not
 an element of $\mathcal{C}(\varphi_1,\varphi_2,\varphi_3,\varphi_4)$.
\end{example}

By Example~\ref{ex.3}, we found that
in order to count the number of the set $S_d(\la)$,
we must also consider the ``intersection multiplicities'' of ``components''
which are proper subsets of $E_d(\mb{I})$
for some maximal $\mb{I} \in \mathfrak{I}(\la)$.

To state Theorem~\ref{thm.B}, we need the following symbol:

\begin{definition}
 For $\la=(\la_1,\ldots,\la_d) \in V_d$ and $I \in \mathcal{I}(\la)$,
 we put $\la_I := (\la_i)_{i \in I}$.
\end{definition}
Note that $\la_I$ always belongs to $V_{\#(I)}$ by definition.

\begin{vartheorem}\label{thm.B}
 Let $\la$ be an element of $V_d$.  Then
 \begin{enumerate}
  \item we have
 $\left\{C\in\mathcal{C}(\varphi_1,\ldots,\varphi_{d-2}) \bigm|
  C \subseteq B_d(\la)\right\}
  \subseteq \left\{E_d(\mb{I})\bigm| \mb{I} \in \mathfrak{I}(\la)\right\}$.
  \item For any $2 \le l \le d-1$
 we have\\ \hspace*{35pt}
 $\left\{C\in\mathcal{C}(\varphi_1,\ldots,\varphi_{d-l})\bigm|
  \dim C > l-2 \right\}
  \subseteq\left\{E_d(\mb{I})\bigm| \mb{I} \in \mathfrak{I}(\la)\right\}$.
  \item For any $\mb{I} = \{I_1,\ldots,I_l\} \in \mathfrak{I}(\la)$,
 we have
 \begin{equation}\label{eq.4.B}
  \mult_{E_d(\mb{I})}(\varphi_1,\ldots,\varphi_{d-l})
   = \prod_{u=1}^l \biggl( 
   \Bigl( \#\left(I_u\right) -1 \Bigr) \cdot
   \#\Bigl(S_{\#(I_u)}\left(\la_{I_u}\right)\Bigr)  \biggr),
 \end{equation}
 where the cardinality $\#\bigl(S_{\#(I_u)}\left(\la_{I_u}\right)\bigr)$
 is defined to be $1$
 if $\#(I_u)$ is equal to or smaller than $3$.
 \end{enumerate}
\end{vartheorem}
By Proposition~\ref{pr.4.1} and Theorem~\ref{thm.B},
the variety $E_d(\mb{I})$ for $\mb{I}\in\mathfrak{I}(\la)$
is really an element of
$\mathcal{C}(\varphi_1,\ldots,\varphi_{d-2})$,
if and only if
the right hand side of the equality~(\ref{eq.4.B}) is strictly positive.

\begin{remark}\label{rm.4.10}
 If an element $\mb{I} = \{I_1,\ldots,I_l\} \in \mathfrak{I}(\la)$
 is maximal, then 
 $\mathfrak{I}(\la_{I_u})$ is empty for every $u$,
 which implies
  $\#\Bigl(S_{\#(I_u)}\left(\la_{I_u}\right)\Bigr)
   = \bigl( \#(I_u)-2 \bigr)!$
 by Definition~\ref{df.2.2}, Lemmas~\ref{lm.4.1} and~\ref{lm.4.6}.
 Thus Theorem~\ref{thm.A} is a special case of Theorem~\ref{thm.B}.
\end{remark}

By Proposition~\ref{pr.4.1} and Theorem~\ref{thm.B},
we have the following:

\begin{varproposition}\label{pr.C}
 Let $\la$ be an element of $V_d$.
 Then we have the equality
 \begin{equation}\label{eq.4.C}
  \#\left(S_d(\la)\right) = (d-2)!
   - \sum_{\mb{I} \in \mathfrak{I}(\la)} \left(
       \mult_{E_d(\mb{I})}(\varphi_1,\ldots,\varphi_{d-\#(\mb{I})}) \cdot
       \prod_{k=d-\#(\mb{I})+1}^{d-2}k
						      \right).
 \end{equation}
 Here, 
 for $\mb{I} \in \mathfrak{I}(\la)$ with $\#(\mb{I}) = 2$,
 we assume that $\prod_{k=d-\#(\mb{I})+1}^{d-2}k = \prod_{k=d-1}^{d-2}k = 1$.
\end{varproposition}

As we have seen in Theorem~\ref{thm.B} and Proposition~\ref{pr.C},
the cardinality $\#\left(S_d(\la)\right)$ is completely determined
by the data $\mathfrak{I}(\la)$.
Moreover it is practically computed only by hand,
though the process of its computation
may be rather long or complicated.
To relieve the long computation,
we give one more proposition.

\begin{varproposition}\label{pr.D}
 For $\la \in V_d$ and $\mb{I}=\{I_1,\ldots,I_l\} \in \mathfrak{I}(\la)$,
 the number
 $\mult_{E_d(\mb{I})}(\varphi_1,\ldots,\varphi_{d-l})$
 given in the equality~(\ref{eq.4.B}) is also equal to
 \begin{multline}\label{eq.4.D}
   \left( \prod_{u=1}^l \bigl( \#\left(I_u\right) -1 \bigr)! \right) \\
    - \sum_\textrm{\scriptsize $\begin{matrix}
				 \mb{I}' \in \mathfrak{I}(\la) \\
				 \mb{I}' \succ \mb{I}, \;
				 \mb{I}' \ne \mb{I}
			       \end{matrix}$}
    \left(
     \mult_{E_d(\mb{I}')}(\varphi_1,\ldots,\varphi_{d-\#\left(\mb{I}'\right)})
     \cdot \prod_{u=1}^l 
     \left( \prod_{k=\#(I_u)-\chi_u(\mb{I}')+1}^{\#(I_u)-1}k \right)
    \right),
 \end{multline}
 where $\chi_u\left(\mb{I}'\right)$ is the one
 defined in Main Theorem~\ref{mthm.3}.
 Here, if $\chi_u(\mb{I}')=1$, then 
 we assume that
 $\prod_{k=\#(I_u)-\chi_u(\mb{I}')+1}^{\#(I_u)-1}k = \prod_{k=\#(I_u)}^{\#(I_u)-1}k = 1$.
\end{varproposition}

Theorem~\ref{thm.A} is just a corollary of Theorem~\ref{thm.B}
by Remark~\ref{rm.4.10}.
However the proof of Theorem~\ref{thm.B} is much harder than that
of Theorem~\ref{thm.A}.
Therefore we prove Theorem~\ref{thm.A} first in Section~\ref{sec.5}, and 
based on its proof we prove Theorem~\ref{thm.B} in Section~\ref{sec.6}.
Proposition~\ref{pr.D} is also proved in Section~\ref{sec.6}.

\section{Proof of Theorem~\ref{thm.A}}\label{sec.5}

In this section we prove Theorem~\ref{thm.A} introduced in Section~\ref{sec.4},
together with preparing for the proof of Theorem~\ref{thm.B}.

We fix our notation first, which is valid
throughout Sections~\ref{sec.5} and~\ref{sec.6}.
For a given $\la \in V_d$
 and $\mb{I} = \{I_1,\ldots,I_l\} \in \mathfrak{I}(\la)$,
we put
 $\#(I_u) =: r_u +1$,
 $(\z_i)_{i\in I_u}=:(\z_{u,0},\z_{u,1},\ldots,\z_{u,r_u})$,
 $(\la_i)_{i\in I_u}=:(\la_{u,0},\la_{u,1},\ldots,\la_{u,r_u})$
and
 $m_{u,i}:=\frac{1}{1-\la_{u,i}}$.
Moreover we assume $\z_{l,0} = \z_d = 0$.
Then we have
 $\sum_{u=1}^l (r_u + 1) = d$,
 $\sum_{i=0}^{r_u} m_{u,i} = 0$,
 $\varphi_k(\z)
   = \sum_{u=1}^l \sum_{i=0}^{r_u} m_{u,i} \z_{u,i}^k$ and
\[
 E_d(\mb{I}) 
  = \left\{\z \in \Pd \bigm| \z_{u,0} = \z_{u,1} = \cdots =\z_{u,r_u}
     \textrm{ for } 1 \le u \le l \right\}
  \cong \mb{P}^{l-2}.
\]
Furthermore let $\alpha_1,\alpha_2,\ldots,\alpha_l$ be any mutually distinct 
complex numbers with $\alpha_l = 0$,
and we denote by $\alpha$ the point $\z \in E_d(\mb{I})$ which satisfies
$\z_{u,i} = \alpha_u$ for any $u$ and $i$.
In the following,
we find $\mult_{E_d(\mb{I})}(\varphi_1,\ldots,\varphi_{d-l})$
by cutting $E_d(\mb{I})$ at $\alpha$
by the plane
$\mathcal{H}(\alpha):=
 \left\{\z\in\Pd \bigm| \z_{u,0}=\alpha_u \text{ for } 1 \le u \le l \right\}$.
We put $\xi_{u,i} := \z_{u,i} - \alpha_u$,
$\xi_u := (\xi_{u,1},\ldots,\xi_{u,r_u}) \in \mb{C}^{r_u}$,
$\xi := (\xi_1,\ldots,\xi_l)\in \mb{C}^{d-l}$
and
\begin{equation}\label{eq.5.defpsi} 
 \psi_k(\xi) := \varphi_k(\alpha + \xi) =
   \sum_{u=1}^l\left( 
    m_{u,0}\alpha_u^k
    + \sum_{i=1}^{r_u}m_{u,i}\left(\alpha_u + \xi_{u,i}\right)^k
   \right).
\end{equation}
Then $\xi$ is a local coordinate system of $\mathcal{H}(\alpha)$
centered at $\alpha$.
\begin{proposition}\label{pr.5.0}
 For any $\mb{I} = \{I_1,\ldots,I_l\} \in \mathfrak{I}(\la)$
 and for generic $\alpha \in E_d(\mb{I})$,
 we have
 \begin{equation}\label{eq.5.1}
  \mult_{E_d(\mb{I})}(\varphi_1,\ldots,\varphi_{d-l})
   = \mult_0(\psi_1,\ldots,\psi_{d-l}).
 \end{equation}
\end{proposition}
\begin{proof}
 Obvious by definition.
\end{proof}
In practice, the equality~(\ref{eq.5.1}) always holds for any $\alpha$
if $\alpha_1,\ldots,\alpha_l$ are mutually distinct,
which will be verified in Proposition~\ref{pr.6.6}.

We shall rewrite the equations $\psi_k(\xi)=0$.
Putting
\[ 
 p_{u,k}(\xi_u) = \sum_{i=1}^{r_u}m_{u,i}\xi_{u,i}^k
\]
for each $u$ and $k$,
we have
\begin{equation}\label{eq.5.psip}
 \begin{split}
 \psi_k(\xi)
 &= 
    \sum_{u=1}^l\left(
      \left( \sum_{i=0}^{r_u}m_{u,i} \right) \alpha_u^k
       + \sum_{i=1}^{r_u} \sum_{h=1}^k
           m_{u,i} \binom{k}{h} \alpha_u^{k-h} \xi_{u,i}^h
     \right) \\
 &= 
    \sum_{u=1}^l \sum_{h=1}^k \binom{k}{h}\alpha_u^{k-h} p_{u,h}(\xi_u),
 \end{split}
\end{equation}
where $\binom{k}{h}=\frac{k(k-1)\cdots (k-h+1)}{h!}$
denotes the binomial coefficient.
Hence $\psi_k(\xi)$ is a linear combination of
$p_{u,h}(\xi_u)$ for $1 \le u \le l$ and $1 \le h \le k$.

\begin{proposition}\label{pr.5.4}
 The equations 
  $\psi_k(\xi)=0$
 for $1 \le k \le d-l$ are equivalent to the equations
 \begin{equation}\label{eq.5.linear} 
  p_{u,k}(\xi_u)
   = \sum_{v=1}^l\sum_{h=r_v+1}^{d-l}a_{u,k,v,h}p_{v,h}(\xi_v)
 \end{equation}
 for $1 \le u \le l$ and $1 \le k \le r_u$,
 where the coefficients
 $a_{u,k,v,h}$
 are some constants which depend only on
 $r_1,\ldots,r_l$ and $\alpha_1,\ldots,\alpha_l$.
\end{proposition}

\begin{proof}
 It suffices to show the invertibility of
 the square matrix composed of the coefficients of 
 $p_{u,h}(\xi_u)$ for $1 \le u \le l$ and $1 \le h \le r_u$ in the
 right hand side of the expressions~(\ref{eq.5.psip}).
 Proposition~\ref{pr.5.4} is therefore reduced to
 the problem on linear algebra,
 whose proof is given in Lemma~\ref{lm.5.5} at the end of this section.
\end{proof}

By the aid of Propositions~\ref{pr.5.0} and~\ref{pr.5.4}, we have reduced
Theorem~\ref{thm.A} to the following:

\begin{proposition}\label{pr.5.9}
 Suppose that an element $\mb{I}\in\mathfrak{I}(\la)$ is maximal.
 Then
 for any complex numbers $a_{u,k,v,h}$, the origin
 $0$ is a discrete solution of the equations~(\ref{eq.5.linear})
 for $1\le u \le l$ and $1 \le k \le r_u$
 with its intersection multiplicity $r_1!\cdots r_l!$.
\end{proposition}

In the following, we prove
Proposition~\ref{pr.5.9}.

\begin{lemma}\label{lm.5.10}
 Let $m_1,\ldots,m_r$ be complex numbers
 such that $\sum_{i\in I}m_i \ne 0$ holds
 for any non-empty $I \subseteq \{1,\ldots,r\}$.
 We put
  $p_k(\xi) :=\sum_{i=1}^r m_i \xi_i^k$
 for $\xi=(\xi_1,\ldots,\xi_r) \in \mb{C}^r$.
 Then $0$ is the only solution of the equations
 $p_k(\xi)=0$ for $1 \le k \le r$ with its intersection multiplicity
  $\mult_0(p_1,\ldots,p_r) = r!$.
\end{lemma}
\begin{proof}
 By the same argument as in the proof of Lemma~\ref{lm.4.6},
 the existence of a solution other than $0$
 implies the equality $\sum_{i\in I}m_i =0$
 for some non-empty $I \subseteq \{1,\ldots,r\}$;
 thus the contradiction assures the uniqueness of the solution.
 
 By Lemmas~\ref{lm.4.1} and~\ref{lm.4.6},
 the set of the common zeros of $p_1,\ldots,p_{r-1}$ in
 $\mb{P}^{r-1}$ is discrete and has $(r-1)!$ points,
 whose intersection multiplicities are all $1$.
 Hence the set of the common zeros of $p_1,\ldots,p_{r-1}$ in
 $\mb{C}^r$ consists of $(r-1)!$ lines $\ell_1,\ldots,\ell_{(r-1)!}$,
 all of which pass the origin. Moreover their intersection multiplicities 
 $\mult_{\ell_i}(p_1,\ldots,p_{r-1})$ are all $1$.
 Since each line $\ell_i$ intersects the hypersurface $\{p_r(\xi)=0\}$ only
 at the origin, the intersection multiplicity $\mult_0(\ell_i,p_r)$ 
 is $r$ for each $i$.
 We thus have the equality
  $\mult_0(p_1,\ldots,p_r) = r \cdot (r-1)! = r!$.
\end{proof}

The most important part in the proof of Proposition~\ref{pr.5.9} is 
to reduce Proposition~\ref{pr.5.9} to Lemma~\ref{lm.5.10}
by replacing all the coefficients $a_{u,k,v,h}$ by $0$.

We denote by
 $A = (a_{u,k,v,h})$
an element of $\mb{C}^{(l-1)(d-l)^2}$,
where the indices $u,k,v,h$ range in $1 \le u \le l$, $1 \le k \le r_u$,
$1 \le v \le l$ and $r_v + 1 \le h \le d-l$.
We put
\[
 D_R := \left\{ A=(a_{u,k,v,h}) \in \mb{C}^{(l-1)(d-l)^2} \Bigm| 
      \left|a_{u,k,v,h}\right| < R \ \text{ for any } u,k,v,h \right\}
\]
and define the map
 $F : \mb{C}^{d-l} \times D_R \to \mb{C}^{d-l} \times D_R$
by
\[
 (\xi,A) \mapsto 
                  \left(\left(
                        p_{u,k}(\xi_u) - \sum_{v,h}a_{u,k,v,h}p_{v,h}(\xi_v)
                       \right)_{u,k} ,A \right),
\]
where the indices $u,k$ range in $1 \le u \le l$ and $1 \le k \le r_u$.

\begin{proposition}\label{pr.5.11}
 Suppose that an element $\mb{I}\in \mathfrak{I}(\la)$ is maximal.
 Then for any positive real number $R$ and any open neighborhood $U_0$
 of $0$ in $\mb{C}^{d-l}$,
 there exist open neighborhoods $U,W$ of $0$ in $\mb{C}^{d-l}$
 with $U \subseteq U_0$ such that the map
 \begin{equation}\label{eq.5.11}
  \left(U \times D_R \right) \cap F^{-1}\left( W \times D_R \right)
  \stackrel{F}{\to} W \times D_R
 \end{equation}
 is proper, and therefore a finite branched covering.
\end{proposition}

 In the following, we prove Proposition~\ref{pr.5.9} first 
 under the assumption of Proposition~\ref{pr.5.11}, and
 secondly we prove Proposition~\ref{pr.5.11}.
 
\begin{proof}[Proof of Proposition~\ref{pr.5.9}]
 First for any given coefficients $a_{u,k,v,h}$,
 we take a positive real number 
 $R$ sufficiently large such that the ball $D_R$ contains $A=(a_{u,k,v,h})$.
 Then the discreteness of the solution $0$ is verified by the 
 finiteness of the map~(\ref{eq.5.11}).
 Secondly we take an open neighborhood $U_0$ of $0$ in $\mb{C}^{d-l}$
 sufficiently small
 such that
 the only solution of the equations~(\ref{eq.5.linear}) in $U_0$ is $0$.
 Then
 the intersection multiplicity of the equations~(\ref{eq.5.linear}) at $0$
 is equal to the degree of the branched covering map~(\ref{eq.5.11}),
 which is also equal to
 the intersection multiplicity of the equations~(\ref{eq.5.linear}) at $0$
 with all the coefficients $a_{u,k,v,h}$ equal to $0$.
 Therefore it
 is $r_1! \cdots r_l!$ by Lemma~\ref{lm.5.10},
 which completes the proof of Proposition~\ref{pr.5.9}.
\end{proof}

\begin{proof}[Proof of Proposition~\ref{pr.5.11}]
 We put
 $|\xi_u| := \max_{1 \le i \le r_u} |\xi_{u,i}|$,
 $Z_u := \left\{\xi_u\in\mb{C}^{r_u} \bigm| |\xi_u|=1 \right\}$ and
 $\delta_u := \inf_{\xi_u \in Z_u}
   \max_{1 \le k \le r_u} \left| p_{u,k}(\xi_u) \right|$
 for each $u$.
 Then by the maximality of $\mb{I}\in \mathfrak{I}(\la)$
 and Lemma~\ref{lm.5.10},
 we have $\delta_u >0$ for each $u$,
 which implies the inequality
  $\max_{1\le k\le r_u}\left|p_{u,k}(\xi_u)\right| \ge \delta_u |\xi_u|^{r_u}$
 for any $\xi_u \in \mb{C}^{r_u}$ with $|\xi_u| \le 1$.
 Hence putting
  $\delta := \min_{1\le u\le l}\delta_u$ and
  $||\xi ||:= \max_{1\le u\le l} |\xi_u |^{r_u}$,
 we have the inequality
 \begin{equation}\label{eq.5.11.1} 
  \max_{u,k}\left|p_{u,k}(\xi_u)\right| \ge \delta \cdot ||\xi ||
 \end{equation}
 for $||\xi || \le 1$.

 On the other hand, for any $A=(a_{u,k,v,h}) \in D_R$ and
 $\xi \in \mb{C}^{d-l}$ with $||\xi || \le 1$,
 we have
 \begin{equation}\label{eq.5.11.2} 
  \begin{split} 
   \max_{u,k}
   \left| \sum_{v,h} a_{u,k,v,h} p_{v,h}(\xi_v) \right|
    &\le 
      \sum_{v,h}R\left(\sum_{i=1}^{r_v}|m_{v,i}|\right) |\xi_v|^h \\
   &\le L \cdot ||\xi ||^{1+\mu},
  \end{split}
 \end{equation}
 where we put
 $L := R\sum_{v=1}^l (d-l-r_v)\left(\sum_{i=1}^{r_v}|m_{v,i}|\right)$
 and
  $\mu :=  \frac{1}{\max_u r_u}$.
 
 Therefore if we take $\xi \in \mb{C}^{d-l}$
 with
  $||\xi || \le \left( \frac{\delta}{2L} \right)^{1/\mu}$,
 then by the inequalities~(\ref{eq.5.11.1}) and~(\ref{eq.5.11.2}),
 we have
 \begin{align*} 
   \max_{u,k}
  &\left| 
     p_{u,k}(\xi_u)-\sum_{v,h}a_{u,k,v,h}p_{v,h}(\xi_v)\right|\\
  &\ge \max_{u,k} \left| p_{u,k}(\xi_u) \right| 
   - \max_{u,k} \left| 
     \sum_{v,h}a_{u,k,v,h}p_{v,h}(\xi_v)\right| \\
  &\ge \delta \cdot ||\xi|| - L\cdot ||\xi ||^{1+\mu}
  \ge \delta \cdot ||\xi|| - L\cdot 
    \frac{\delta}{2L} \cdot ||\xi ||
  = 
     \frac{\delta}{2}\cdot ||\xi||.
 \end{align*}
 We define a positive number $\epsilon$ sufficiently small
 such that the inequality
   $0 < \epsilon < \left( \frac{\delta}{2L} \right)^{1/\mu}$
 holds and that the set
    $U := \left\{ \xi  \in \mb{C}^{d-l} \bigm| ||\xi || < \epsilon \right\}$
 is included in $U_0$. Moreover we put
 \[
  W := \left\{ \eta = (\eta_{u,k}) \in \mb{C}^{d-l} \ \left| \ 
        |\eta | = \max_{u,k}|\eta_{u,k}| < \tfrac{1}{2}\delta\epsilon
             \right.\right\}.
 \]
 Then we can easily verify that the map~(\ref{eq.5.11}) is proper.
 Therefore by Lemma~\ref{lm.5.11.2} below, 
 the map~(\ref{eq.5.11}) is a finite branched covering.
\end{proof}

\begin{lemma}\label{lm.5.11.2}
 Let $U, V$ be connected open subsets of $\mathbb{C}^n$, and $f:U \to V$ a proper holomorphic map. Then $f:U \to V$ is a finite branched covering.
\end{lemma}

\begin{proof}[Proof of Lemma~\ref{lm.5.11.2}]
 Note that there does not exist a compact analytic subset of $\mathbb{C}^n$ 
 whose dimension is greater than or equal to 1.
 Since $K:= \left\{z \in U \mid \det(Df)(z)=0 \right\}$ is an analytic subset of $U$ 
 with $K \ne U$, $f(K)$ is also an analytic subset of $V$ by proper mapping theorem.
 Hence the map $U \setminus f^{-1}\circ f(K) \to V \setminus f(K)$ is 
 proper and locally homeomorphic, and therefore is a covering space of finite degree, 
 which implies that $f$ is a finite branched covering.
\end{proof}

\vspace{12pt}

The rest of this section
is devoted to Lemma~\ref{lm.5.5} and its proof.

\begin{definition}
 For non-negative integers $n,b,k,h$ with $n>k$ and $b>h$,
 we denote by $A_{n,k}^{b,h}(\alpha)$ the $(n-k,b-h)$ matrix
 whose $(i,j)$-th entry is $\binom{i+k-1}{j+h-1} \alpha^{(i+k)-(j+h)}$
 for each $i$ and $j$.
 Moreover we put $A_{n,k}^b(\alpha):=A_{n,k}^{b,0}(\alpha)$ and
 $A_n^b(\alpha):=A_{n,0}^{b,0}(\alpha)$.
\end{definition}

By definition,
the matrix $A_{n,k}^{b,h}(\alpha)$ is obtained from the $(n,b)$ matrix
\[
 A_n^b(\alpha)
  = \begin{pmatrix}
      1        & 0         & 0         & 0       & \cdots & 0      \\
      \alpha   & 1         & 0         & 0       & \cdots & 0      \\
      \alpha^2 & 2\alpha   & 1         & 0       & \cdots & 0      \\
      \alpha^3 & 3\alpha^2 & 3\alpha   & 1       & \cdots & 0      \\
      \vdots   & \vdots    & \vdots    & \vdots  & \ddots & \vdots \\
      \alpha^{n-1} & (n-1)\alpha^{n-2} & \binom{n-1}{2}\alpha^{n-3} &
        \binom{n-1}{3}\alpha^{n-4} & \cdots & \ddots
     \end{pmatrix}
\]
by cutting off the upper $k$ rows and the left $h$ columns.

\begin{lemma}\label{lm.5.5}
 We put $r := r_1 + \cdots + r_l = d -l$, and
 denote by $M$ the $(r,r)$ square matrix defined by
 \[
  M = \Bigl( A_{r+1,1}^{r_1+1,1}(\alpha_1), \ldots,
             A_{r+1,1}^{r_l+1,1}(\alpha_l) \Bigr).
 \]
 Then we have
 \[
  \det M = \frac{r!}{r_1!\cdots r_l!}\cdot
  \prod_{1\le v<u \le l} (\alpha_u -\alpha_v)^{r_v r_u}.
 \]
\end{lemma}

The matrix $M$ defined above is the same as the square matrix
composed of the coefficients of $p_{u,h}(\xi_u)$ for $1 \le u \le l$ and
$1 \le h \le r_u$ in the right hand side of the expressions~(\ref{eq.5.psip});
hence Proposition~\ref{pr.5.4} is reduced to Lemma~\ref{lm.5.5}.

To prove Lemma~\ref{lm.5.5}, we give a definition and a lemma.

\begin{definition}
 For a positive integer $b$,
 we denote by $X_b$ the $(b,b)$ diagonal matrix
 whose $(i,i)$-th entry is $i$ for $1 \le i \le b$,
 and by $N_b$ the $(b,b)$ nilpotent matrix
 whose $(i,i+1)$-th entry is $1$ for $1 \le i \le b-1$
 and whose other entries are $0$,
 i.e.,
 \[
  X_b = \begin{pmatrix}
	 1      & 0      & \cdots & 0 \\
	 0      & 2      & \cdots & 0 \\
	 \vdots & \vdots & \ddots & \vdots \\
	 0      & 0      & \cdots & b
	\end{pmatrix}
  \quad \text{and} \quad
  N_b = \begin{pmatrix}
	 0      & 1      & 0      & \cdots & 0 \\
	 0      & 0      & 1      & \cdots & 0 \\
	 \vdots & \vdots & \vdots & \ddots & \vdots \\
	 0      & 0      & 0      & \cdots & 1 \\
	 0      & 0      & 0      & \cdots & 0
	\end{pmatrix}.
 \]
\end{definition}

\begin{lemma}\label{lm.5.6}
 For positive integers $n$ and $b$, we have the equalities
 \[
  A_{n+1,1}^{b+1,1}(\alpha) = X_n \cdot A_n^b(\alpha) \cdot {X_b}^{-1}
   \quad \text{and} \quad
  A_n^n(\beta) \cdot A_n^b(\alpha) = A_n^b(\beta + \alpha).
 \]
 Moreover for positive integers $n,b,k$ with $n>k$
 and a non-zero complex number $\alpha$,
 we have the equality
 \[ 
  A_{n,k}^b(\alpha) \cdot
  \sum_{h=0}^{b-1} \binom{-k}{h} \left(\alpha^{-1}N_b\right)^h
  = \alpha^k A_{n-k}^b(\alpha),
 \]
 where $\left(\alpha^{-1}N_b\right)^0$ denotes the identity matrix of
 size $(b,b)$.
\end{lemma}
\begin{proof}
 The first equality is verified
 by $\binom{i}{j}=\binom{i-1}{j-1}\cdot\frac{i}{j}$,
 the second one by
 $\binom{i}{h}\binom{h}{j}=\binom{i}{j}\binom{i-j}{h-j}$
 and $\sum_{h=0}^k \binom{k}{h}\alpha^h\beta^{k-h} = (\alpha+\beta)^k$,
 and the last one by the equality
 $\sum_{h=0}^j \binom{x}{h}\binom{y}{j-h} = \binom{x+y}{j}$.
\end{proof}

\begin{proof}[Proof of Lemma~\ref{lm.5.5}]
 By Lemma~\ref{lm.5.6}, we have
  $A_{r+1,1}^{r_u+1,1}(\alpha_u)
   = X_r \cdot A_r^{r_u}(\alpha_u) \cdot \left(X_{r_u}\right)^{-1}$
 for each $1 \le u \le l$.
 Hence putting 
  $M' = \Bigl( A_r^{r_1}(\alpha_1), \ldots,
              A_r^{r_l}(\alpha_l) \Bigr)$,
 we have the equalities
 \[ 
  \det M
   = \det X_r \cdot \det M' \cdot \prod_{u=1}^l \det \left(X_{r_u}\right)^{-1}
   = \frac{r!}{r_1!\cdots r_l!} \cdot \det M'.
 \]
 Therefore to prove Lemma~\ref{lm.5.5}, we only need to show the equality
 \begin{equation}\label{eq.5.Mprime} 
  \det M' = \prod_{1\le v<u \le l} (\alpha_u -\alpha_v)^{r_v r_u}.
 \end{equation}
 
 If there exist distinct indices $u,v$ with $\alpha_u=\alpha_v$,
 then both hand sides of the equality~(\ref{eq.5.Mprime})
 are clearly zero;
 hence we only need to consider the equality~(\ref{eq.5.Mprime}) in the case
 that $\alpha_1,\ldots,\alpha_l$ are mutually distinct.
 Moreover if $l=1$,
 the equality~(\ref{eq.5.Mprime}) trivially holds since $\det M'=1$.
 In the following,
 we show the equality~(\ref{eq.5.Mprime}) by induction of $l$.

 We put $r' = r_2+\cdots +r_l$
 and $\alpha'_u = \alpha_u - \alpha_1$ for $2 \le u \le l$.
 Then by Lemma~\ref{lm.5.6}, we have
 \[
  A_r^r(-\alpha_1) \cdot M'
  = \Bigl( A_r^{r_1}(0),A_r^{r_2}(\alpha'_2),\ldots,A_r^{r_l}(\alpha'_l) \Bigr)
  = \begin{pmatrix}
     I_{r_1} & * \\
     O       & \widetilde{M}
    \end{pmatrix},
 \]
 where we put $\widetilde{M}=\Bigl( A_{r,r_1}^{r_2}(\alpha'_2), \ldots,
    A_{r,r_1}^{r_l}(\alpha'_l) \Bigr)$, and
 $I_{r_1}$ denotes the identity matrix of size $(r_1, r_1)$.
 Moreover by Lemma~\ref{lm.5.6}, we have
 \[ 
  A_{r,r_1}^{r_u}(\alpha'_u) \cdot
   \sum_{h=0}^{r_u-1}\binom{-r_1}{h}\left((\alpha'_u)^{-1}N_{r_u}\right)^h
  = (\alpha'_u)^{r_1} \cdot A_{r'}^{r_u}(\alpha'_u)
 \]
 for each $2 \le u \le l$.
 Hence putting
 $M''=\Bigl( A_{r'}^{r_2}(\alpha'_2), \ldots, A_{r'}^{r_l}(\alpha'_l) \Bigr)$,
 we have the equalities
 \begin{equation*}
  \det M' = \det \widetilde{M} = \det M'' \cdot
    \prod_{u=2}^l \left(\alpha'_u\right)^{r_1 r_u},
 \end{equation*}
 which completes the proof by induction of $l$.
\end{proof}

\section{Proof of Theorem~\ref{thm.B}}\label{sec.6}

In this section we give the proofs of Theorem~\ref{thm.B} and
Proposition~\ref{pr.D} introduced in Section~\ref{sec.4},
which are also the most crucial part in the proof of the main theorems.
We first give a key estimate in Proposition~\ref{pr.6.1},
which is a substitute for the inequalities~(\ref{eq.5.11.1}) and~(\ref{eq.5.11.2})
in the case that $\mb{I}\in\mathfrak{I}(\la)$ is not necessarily maximal.

\begin{proposition}\label{pr.6.1}
 Let $r$ be a positive integer, and $m_1,\ldots, m_r$ non-zero complex
 numbers with $\sum_{i=1}^r m_i \ne 0$.
 We put $m=(m_1,\ldots, m_r)$,
 \[
  p_k(\xi) := 
                 \sum_{i=1}^r m_i \xi_i^k, \quad 
  B(m) := \left\{ \xi \in \mb{C}^r \bigm|
          p_k(\xi)=0 \text{ for } 1 \le k \le r \right\},
 \]
 and
 $|\xi | := \max_{1 \le i \le r}|\xi_i |$
 for $\xi = (\xi_1,\ldots,\xi_r) \in \mb{C}^r$.
 Then
 \begin{enumerate}
  \item  for each positive integer $h$, there exists a positive real
	 number $L_h$ such that the inequality
	 \begin{equation}\label{eq.6.1.2}
	  \left| p_h(\xi)\right| \le L_h \cdot
	   \max_{1 \le k \le r}\left| p_k(\xi)\right|
	 \end{equation}\label{en.6.1.2}
	 holds for any $\xi \in \mb{C}^r$ with $|\xi |=1$.
  \item  There exist an open neighborhood $O$ of 
	 $B(m) \cap \{ \xi \in \mb{C}^r \bigm| |\xi | =1 \}$ in
	 $\mb{C}^r$
	 and a positive real number $L'$ such that the inequality
	 \begin{equation}\label{eq.6.1.2.2}
	  \left| p_r(\xi)\right| 
	   \le L'\cdot\max_{1\le k\le r-1}\left| p_k(\xi)\right|
	 \end{equation}
	 holds for any $\xi \in O$. \label{en.6.1.3}
 \end{enumerate}
\end{proposition}
\begin{proof}
 We put $m_0 := -\sum_{i=1}^r m_i$,
 \begin{align*}
  \mathfrak{I}(m) &:= \left\{ \left\{I_1,\ldots,I_l\right\}\ \left| \
   \begin{matrix}
    I_1 \amalg \cdots \amalg I_l = \{0,\ldots,r \}, \quad l \ge 1 \\
    I_u \ne \emptyset \text{ and } \sum_{i \in I_u} m_i = 0
       \text{ for } 1 \le u \le l \\
   \end{matrix}
   \right.\right\}, \\
  E(\mb{I}) &:= \left\{ (\xi_1,\ldots,\xi_r) \in \mb{C}^r \bigm|
    \text{If $i,j \in I \in \mb{I}$, then $\xi_i=\xi_j$}
           \right\} \\
  \intertext{for each $\mb{I} \in \mathfrak{I}(m)$, and}
  \mb{I}(\xi) &:= \left\{ I \subseteq \{0,1,\ldots,r\} \ \left| \
   \begin{matrix}
    I \ne \emptyset. \qquad
    \text{If $i,j\in I$, then $\xi_i = \xi_j$.} \\
    \text{If $i \in I$ and $j \in \{0,1,\ldots,r\} \setminus I$,
      then $\xi_i \ne \xi_j$}.
   \end{matrix}
  \right.\right\}
 \end{align*}
 for each $\xi \in B(m)$,
 where we are assuming $\xi_0 = 0$.
 Then we have the equality
	 \begin{equation}\label{eq.6.1.1} 
	  B(m) = \bigcup_{\mb{I} \in \mathfrak{I}(m)}E(\mb{I}),
	 \end{equation}
 and we also have $\mb{I}(\xi) \in \mathfrak{I}(m)$ and
 $\xi \in E(\mb{I}(\xi))$ for each $\xi \in B(m)$
 by the same argument as the proof of Lemma~\ref{lm.4.6}.
 Note that in this setting,
 the set $\mathfrak{I}(m)$ always contains
 the element $\mb{I}_0:=\left\{\{0,\ldots,r\}\right\}$,
 and that the equalities $E(\mb{I}_0)= \{0\}$
 and $\mb{I}(0)=\mb{I}_0$ hold.
 
 We make use of the following auxiliary lemmas:

 \begin{lemma}\label{lm.6.1}
  There exists an open neighborhood $O$ of 
  $B(m) \cap \{ \xi \in \mb{C}^r \bigm| |\xi | =1 \}$ in $\mb{C}^r$
  such that for each positive integer $h$,
  there exists a positive real number $L'_h$ such that the inequality
  \begin{equation}\label{eq.6.1.3}
   \left| p_h(\xi)\right| \le
    L'_h\cdot\max_{1\le k\le r-1}\left| p_k(\xi)\right|
  \end{equation}
  holds for any $\xi \in O$.
 \end{lemma}

 \begin{lemma}\label{lm.6.2}
  Let $\alpha$ be a point in $B(m)\setminus \{0\}$. Then
  there exists an open neighborhood $O_{\alpha}$ of $\alpha$ in $\mb{C}^r$
  such that for each positive integer $h$,
  there exists a positive real number $L_{\alpha,h}$ such that 
  the inequality
  \begin{equation}\label{eq.6.2}
   |p_h(\xi)| \le L_{\alpha,h} \cdot
    \max_{1\le k\le r+1-\#\left(\mb{I}(\alpha)\right)} |p_k(\xi)|
  \end{equation}
  holds for any $\xi \in O_{\alpha}$.
 \end{lemma}

 Note that the implications
 \begin{center}
  ``Proposition~\ref{pr.6.1} $\Longrightarrow$ Lemma~\ref{lm.6.1}
  $\Longrightarrow$
  The assertion~(\ref{en.6.1.3}) in Proposition~\ref{pr.6.1}''
 \end{center}
 are clear.
 In the following, we prove Lemmas~\ref{lm.6.1},~\ref{lm.6.2} and
 the assertion~(\ref{en.6.1.2})
 in Proposition~\ref{pr.6.1} simultaneously by induction.
 To make the induction work well, we define
 the ``depth'' of a point $\alpha \in B(m)$ by
 \[
  \tau_m(\alpha) := \max \left\{\nu \ \left| \
   \begin{matrix}
    \mb{I}(\alpha)=:\mb{I}_1 \precneqq \mb{I}_2
          \precneqq \cdots \precneqq \mb{I}_{\nu} \\
    \mb{I}_{\omega} \in
      \mathfrak{I}(m) \ \text{ for } 1 \le \omega \le \nu
   \end{matrix}
    \right. \right\},
 \]
 where the symbol $\mb{I}\precneqq\mb{I}'$ 
 for $\mb{I},\mb{I}'\in\mathfrak{I}(m)$ denotes that
 $\mb{I}'$ is a refinement of $\mb{I}$
 with
 $\mb{I}\ne\mb{I}'$.
 Note that the inequality $\tau_m(0) > \tau_m(\alpha)$ holds
 for any $\alpha \in B(m) \setminus \{0\}$
 and that the equality $\tau_m(0)=1$ holds if and only if $B(m)=\{0\}$.

 We consider the following assertions
 for each non-negative integer $\nu$:
 \begin{enumerate}
  \item[$(1)_{\nu}$] if $\tau_m(0) \le \nu +1$,
	     then the assertion~(\ref{en.6.1.2}) in
	     Proposition~\ref{pr.6.1} holds.
  \item[$(2)_{\nu}$] If $\tau_m(0) \le \nu +1$,
	     then Lemma~\ref{lm.6.1} holds.
  \item[$(3)_{\nu}$] If $\tau_m(\alpha) \le \nu$,
	     then Lemma~\ref{lm.6.2} holds.
 \end{enumerate}
 Note that the assertion~$(2)_{0}$ trivially holds
 since $\tau_m(0) \le 1$ implies $B(m)=\{0\}$.
 In the following, we show the implications
 \[
  (1)_{\nu-1} \Rightarrow (3)_{\nu} \Rightarrow
  (2)_{\nu} \Rightarrow (1)_{\nu}
 \]
 for each $\nu$,
 which will complete the proofs of Lemmas~\ref{lm.6.1},~\ref{lm.6.2} and
 Proposition~\ref{pr.6.1}.
 We put
 \[
  Z := \{ \xi \in \mb{C}^r \bigm| |\xi | =1 \}.
 \]
 
 \begin{proof}[Proof of the implication~$(3)_{\nu} \Rightarrow (2)_{\nu}$]
  We suppose $(3)_{\nu}$ and prove $(2)_{\nu}$.
  When $\tau_m(0) \le \nu +1$, the inequality $\tau_m(\alpha) \le \nu$ holds
  for any $\alpha \in Z \cap B(m)$.
  Hence by the assumption $(3)_{\nu}$,
  we can choose, for each $\alpha \in Z \cap B(m)$,
  an open neighborhood $O_{\alpha}$ of $\alpha$
  and a positive real number $L_{\alpha,h}$
  for each $h \in \mb{N}$ such that the inequality~(\ref{eq.6.2})
  holds for any $\xi \in O_{\alpha}$.
  Since $Z \cap B(m)$ is compact, there exist
  finite number of open neighborhoods $O_{\alpha_1},\ldots,O_{\alpha_{\mu}}$
  which cover $Z \cap B(m)$.
  On the other hand,
  since $\#\left(\mb{I}(\alpha)\right) \ge 2$
  for any $\alpha \in Z \cap B(m)$,
  we always have $r+1-\#\left(\mb{I}(\alpha)\right) \le r-1$.
  Therefore, putting
   $O :=\bigcup_{1 \le \omega \le \mu}O_{\alpha_{\omega}}$
  and
   $L'_h := \max_{1 \le \omega \le \mu} L_{\alpha_{\omega},h}$
  for each $h$,
  we have,
  by the inequality~(\ref{eq.6.2}),
  the inequality~(\ref{eq.6.1.3}) for any $\xi \in O$.
 \end{proof}

 \begin{proof}[Proof of the implication~$(2)_{\nu} \Rightarrow (1)_{\nu}$]
 We suppose $(2)_{\nu}$ and verify $(1)_{\nu}$. 
 The set $Z\setminus O$ is compact
 and does not have common zeros of $p_1,\ldots,p_r$.
 Hence the infimum\\
 $\inf_{\xi \in Z\setminus O} \max_{1 \le k \le r}|p_k(\xi)|$
 is positive, which assures the existence of a positive real number $L_h$
 for each $h \in \mb{N}$
 satisfying the inequality~(\ref{eq.6.1.2}) for any $\xi \in Z\setminus O$.
 Replacing the maximum of $L_h$ and $L'_h$ by $L_h$,
 we have the inequality~(\ref{eq.6.1.2}) for any $\xi \in Z$.
 \end{proof}

 In the rest of the proof,
 we suppose $(1)_{\nu-1}$ and prove $(3)_{\nu}$.
 We fix $\alpha \in B(m) \setminus \{0\}$ with $\tau_m(\alpha) \le \nu$,
 put
  $\mb{I}(\alpha) =: \left\{I_1,\ldots,I_l \right\}$,
 and denote by $\alpha_u^0$ the $i$-th coordinate of $\alpha$ for $i \in I_u$.
 Note that $\alpha_1^0,\ldots,\alpha_l^0$ are mutually distinct.
 We put $\#(I_u) = r_u+1$,
 $(\xi_i)_{i\in I_u} = (\xi_{u,0},\xi_{u,1},\ldots,\xi_{u,r_u})$,
 $(m_i)_{i\in I_u} = (m_{u,0},m_{u,1},\ldots,m_{u,r_u})$,
 $m(I_u)=(m_{u,1},\ldots,m_{u,r_u})$,
 $x_{u,i}=\xi_{u,i}-\xi_{u,0}$, $\alpha_u=\xi_{u,0}$, 
 $x_u = (x_{u,1},\ldots,x_{u,r_u})$, $x=(x_1,\ldots,x_l)$,
 $|x_u|= \max_{1 \le i \le r_u}|x_{u,i}|$
 and $|x|= \max_{1\le u \le l}|x_u|$.
 We may assume $\alpha_l=\alpha_l^0=\xi_{l,0}=\xi_0=0$.
 We may also consider
 the coordinates $(\alpha_1,\ldots,\alpha_{l-1},x)$
 as a local coordinate system around $\alpha$ in $\mb{C}^r$.
 Note that the point $(\alpha_1,\ldots,\alpha_{l-1},x)$ coincides
 with $\alpha$ if and only if $x=0$ and $\alpha_u=\alpha_u^0$ for 
 $1 \le u \le l-1$,
 and that the point $(\alpha_1,\ldots,\alpha_{l-1},x)$ belongs to
 $E\left(\mb{I}(\alpha)\right)$ if and only if $x=0$.
 Furthermore we put
 \[ 
  \theta_{u,k}(x_u) = \sum_{i=1}^{r_u} m_{u,i} x_{u,i}^k
 \]
 for $1 \le u \le l$ and $k \in \mb{N}$.

 Then we have the equality
 \begin{equation}\label{eq.6.ptheta} 
  p_k(\xi) =
   \sum_{u=1}^l \sum_{h=1}^k \binom{k}{h}\alpha_u^{k-h} \theta_{u,h}(x_u)
 \end{equation}
 by the same computation as in the equalities~(\ref{eq.5.psip}).
 Moreover by Lemma~\ref{lm.5.5},
 the equalities~(\ref{eq.6.ptheta}) for $1 \le k \le r+1-l$
 are equivalent in some neighborhood of $\alpha$ to the equalities
 \begin{equation}\label{eq.6.linear2} 
  \theta_{u,k}(x_u) = \sum_{h=1}^{r+1-l} b_{u,k,h} p_h(\xi)
   + \sum_{v=1}^l \sum_{h=r_v +1}^{r+1-l} a_{u,k,v,h} \theta_{v,h}(x_v)
 \end{equation}
 for $1 \le u \le l$ and $1 \le k \le r_u$,
 where the coefficients $b_{u,k,h}$ and $a_{u,k,v,h}$
 depend only on $r_1,\ldots,r_l$ and $\alpha_1,\ldots,\alpha_{l-1}$.
 Moreover its dependence is continuous
 on the domain where $\alpha_1,\ldots,\alpha_{l-1}$ and $0$ are
 mutually distinct.
 Therefore taking a small open neighborhood $\Delta$
 of $(\alpha_1^0,\ldots,\alpha_{l-1}^0)$ in $\mb{C}^{l-1}$
 and a sufficiently large real number $R$,
 we may assume that the inequalities
 \[
  \left| \alpha_u \right| \le R, \quad \left| b_{u,k,h} \right| \le R
   \quad \text{and}\quad \left| a_{u,k,v,h} \right| \le R
 \]
 hold for all $u,k,v,h$ and
 for any $(\alpha_1,\ldots,\alpha_{l-1}) \in \Delta$.

 On the other hand, since $\tau_m(\alpha) \le \nu$,
 we always have $\tau_{m(I_u)}(0) \le \nu$ for any $u$.
 Hence by the assumption $(1)_{\nu-1}$, there exists,
 for each $u$ and for each positive integer $h$, a positive real
 number $L_{u,h}$ such that the inequality
 \[
  \left| \theta_{u,h}(x_u)\right| \le
   L_{u,h} \cdot \max_{1 \le k \le r_u}\left| \theta_{u,k}(x_u)\right|
 \]
 holds for any $x_u \in \mb{C}^{r_u}$ with $|x_u| = 1$.
 Hence by the homogeneity of $\theta_{u,k}(x_u)$,
 the inequality
 \[
  \left| \theta_{u,h}(x_u)\right| \le
   L_{u,h} \cdot
   \max_{1 \le k \le r_u}\left| \theta_{u,k}(x_u)\right| \cdot |x_u|
 \]
 holds for $h \ge r_u +1$ and
 for any $x_u \in \mb{C}^{r_u}$ with $|x_u| \le 1$.
 Therefore by the equality~(\ref{eq.6.linear2}),
 we have the following
 for $(\alpha_1,\ldots,\alpha_{l-1}) \in \Delta$ and $|x|\le 1$:
 \begin{align*}
  \max_{u,k} \left| 
    \sum_{h=1}^{r+1-l} b_{u,k,h} p_h(\xi) \right|
  &\ge \max_{u,k} \left| \theta_{u,k}(x_u) \right| -
        \max_{u,k} \left| 
                          \sum_{v=1}^l \sum_{h=r_v +1}^{r+1-l}
         a_{u,k,v,h} \theta_{v,h}(x_v) \right| \\
  &\ge
   \left( 
     1 - R\sum_{v=1}^l\sum_{h=r_v +1}^{r+1-l}L_{v,h}\cdot |x|\right)
       \max_{u,k} \left| \theta_{u,k}(x_u) \right|.
 \end{align*}
 Hence putting
 \[
  \begin{split}
   &J := \max \left\{1,\  
    2R\sum_{v=1}^l \sum_{h=r_v +1}^{r+1-l}L_{v,h} \right\},
    \quad L := 2R(r+1-l) \\
   \text{and} \quad
   &O_{\alpha} := \left\{(\alpha_1,\ldots,\alpha_{l-1},x) \in \mb{C}^{r}
   \bigm| (\alpha_1,\ldots,\alpha_{l-1}) \in \Delta,
    |x| < 1/J  \right\},
  \end{split}
 \]
 we have,
 for any $\xi = (\alpha_1,\ldots,\alpha_{l-1},x) \in O_{\alpha}$,
 the inequality
 \begin{equation}\label{eq.6.3.1}
  \max_{u,k} \left| \theta_{u,k}(x_u) \right| \le
   2 \max_{u,k} \left| 
		 \sum_{h=1}^{r+1-l}b_{u,k,h}p_h(\xi)\right| \le
   L \cdot \max_{1 \le k \le r+1-l} |p_k(\xi)|.
 \end{equation}

 On the other hand, by the equality~(\ref{eq.6.ptheta}),
 we have, for each positive integer $h$,
 the inequalities
 \begin{equation}\label{eq.6.3.2}
  \left| p_h(\xi) \right|
   \le 
      \sum_{u=1}^l \sum_{k=1}^{h}\binom{h}{k} R^{h-k} L_{u,k} \cdot
      \max_{1 \le k \le r_u} \left| \theta_{u,k}(x_u) \right|
   \le L_h \cdot \max_{u,k} \left| \theta_{u,k}(x_u) \right|
 \end{equation}
 for any $(\alpha_1,\ldots,\alpha_{l-1},x) \in O_{\alpha}$,
 where we put
 $L_h := \sum_{u=1}^l \sum_{k=1}^{h}\binom{h}{k} R^{h-k} L_{u,k}$.
 Therefore by the inequalities~(\ref{eq.6.3.1}) and~(\ref{eq.6.3.2}),
 we have
 \[
  \left| p_h(\xi) \right|
  \le L_h L \cdot \max_{1 \le k \le r+1-l} \left| p_k(\xi) \right|
 \]
 for any $\xi = (\alpha_1,\ldots,\alpha_{l-1},x) \in O_{\alpha}$
 and for each $h$.
 Thus the assertion~$(3)_{\nu}$ is proved, which completes the proof of
 Lemmas~\ref{lm.6.1},~\ref{lm.6.2} and Proposition~\ref{pr.6.1}.
\end{proof}

In the rest of this section,
the notation follows that in Section~\ref{sec.5}.
Therefore $\la$ is an element of $V_d$,
and $\mb{I}=\{I_1,\ldots,I_l\}$ an element of $\mathfrak{I}(\la)$, which
are fixed throughout the rest of this section.
Moreover the notation $r_u$, $\z_{u,i}$, $\la_{u,i}$, $m_{u,i}$, $\alpha_u$,
$\alpha$, $\xi_{u,i}$, $\xi_u$, $\xi$, $\psi_k(\xi)$, $p_{u,k}(\xi_u)$,
$A=(a_{u,k,v,h})$, $D_R$ and the map $F$
is the same as in Section~\ref{sec.5}.
Note that Propositions~\ref{pr.5.0} and~\ref{pr.5.4} are valid 
for non-maximal $\mb{I} \in \mathfrak{I}(\la)$

We give a proposition next which is the most important part in 
the proof of Theorem~\ref{thm.B}, whose proof is essentially based on
Proposition~\ref{pr.6.1}.

\begin{proposition}\label{pr.6.4}
 For any positive real numbers $R$ and
 $1 > \epsilon >0$,
 and for any open neighborhood $U_0$ of $0$ in $\mb{C}^{d-l}$,
 there exist open neighborhoods $U,W$ of $0$ in $\mb{C}^{d-l}$
 with $U \subseteq U_0$ such that the map
 \begin{equation}\label{eq.6.4.1}
  \left(U \times D_R \right) \cap F^{-1}\left( W_{\epsilon} \times D_R \right)
  \stackrel{F}{\to} W_{\epsilon} \times D_R
 \end{equation}
 is proper, and therefore a finite branched covering,
 where 
 \[
  W_{\epsilon} := W \cap \Xi_{\epsilon} \quad \text{and} \quad
  \Xi_{\epsilon} := \left\{
   \eta= \left(\eta_{u,k}\right) \in \mb{C}^{d-l} \ \left| \
   \min_{1 \le u \le l} \left| \eta_{u,r_u} \right| >
    \epsilon \cdot \max_{u,k} \left| \eta_{u,k} \right|
   \right. \right\}.
 \]
\end{proposition}
\begin{proof}
 Remember that
 the map $F : \mb{C}^{d-l}\times D_R \to \mb{C}^{d-l}\times D_R$
 is defined by $F(\xi, A) = (\eta, A)$, where $\xi = (\xi_{u,i})$,
 $\eta = (\eta_{u,k})$, $A=(a_{u,k,v,h})$ and 
 \[ 
  \eta_{u,k} = p_{u,k}(\xi_u) - \sum_{v=1}^l \sum_{h = r_v+1}^{d-l}
   a_{u,k,v,h} p_{v,h}(\xi_v)
 \]
 for $1 \le u \le l$ and $1 \le k \le r_u$.
 We put
 \begin{align*}
  |\xi_u| &:= \max_{1 \le i \le r_u} |\xi_{u,i}|, \quad
   |\xi | := \max_{1 \le u \le l} |\xi_u|, \quad
   |\eta | := \max_{u,k} |\eta_{u,k}|, \\
  \widetilde{B}_u(\la_{I_u}) &:= \left\{ \xi_u \in \mb{C}^{r_u} \bigm|
   p_{u,k}(\xi_u)=0 \text{ for } 1 \le k \le r_u \right\} \text{ and} \\
  Z_u &:= \left\{ \xi_u \in \mb{C}^{r_u} \bigm| |\xi_u| = 1 \right\}.
 \end{align*}

 By the assertion~(\ref{en.6.1.2}) in Proposition~\ref{pr.6.1},
 there exists a positive real number $L_{u,h}$
 for each $u$ and $h$
 such that the inequality
 \[
  \bigl| p_{u,h}(\xi_u) \bigr| \le L_{u,h} \cdot
   \max_{1 \le k \le r_u} \bigl| p_{u,k}(\xi_u) \bigr|
 \]
 holds for any $\xi_u \in Z_u$.
 Hence by the homogeneity of $p_{u,k}(\xi_u)$, we have
 \begin{equation}\label{eq.6.4.2}
  \bigl| p_{u,h}(\xi_u) \bigr| \le L_{u,h} \cdot
   \max_{1 \le k \le r_u} \bigl| p_{u,k}(\xi_u) \bigr| \cdot |\xi_u|
 \end{equation}
 for any $\xi_u \in \mb{C}^{r_u}$ with $|\xi_u| \le 1$
 and for each $h \ge r_u +1$.

 On the other hand,
 by the assertion~(\ref{en.6.1.3}) in Proposition~\ref{pr.6.1},
 there exist an open neighborhood $O_u$ of
 $\widetilde{B}_u(\la_{I_u})\cap Z_u$ in $\mb{C}^{r_u}$
 and a positive real number $L'_u$ for each $u$ such that the inequality
 \[
  \bigl| p_{u,r_u}(\xi_u) \bigr| \le L'_u \cdot
   \max_{1 \le k \le r_u - 1} \bigl| p_{u,k}(\xi_u) \bigr|
 \]
 holds for any $\xi_u \in O_u$.
 We put
 \[
  \Omega_u
   :=\left\{(t\xi_{u,1},\ldots,t\xi_{u,r_u})\in \mb{C}^{r_u} \bigm|
      t \in \mb{R},\ \ t>0,\ \
       (\xi_{u,1},\ldots,\xi_{u,r_u}) \in O_u \cap Z_u
     \right\}  
 \]
 for each $u$ and
 \[
  \Omega := \left\{\xi = (\xi_1,\ldots, \xi_l) \in \mb{C}^{d-l} \bigm|
    \xi_u \in \Omega_u \text{ holds for some $1 \le u \le l$} \right\}.
 \]
 Then $\Omega_u$ is an open neighborhood
 of $\widetilde{B}_u(\la_{I_u}) \setminus \{0\}$ in $\mb{C}^{r_u} \setminus \{0\}$,
 and $\Omega$ is an open set in $\mb{C}^{d-l}$.
 Moreover for $\xi_u \in \mb{C}^{r_u} \setminus \{0\}$,
 the point $\xi_u /|\xi_u|$ belongs
 to the set $O_u \cap Z_u = \Omega_u \cap Z_u$
 if and only if $\xi_u \in \Omega_u$.
 Hence by the homogeneity of $p_{u,k}(\xi_u)$,
 we have the inequality
 \begin{equation}\label{eq.6.4.3}
  \bigl| p_{u,r_u}(\xi_u) \bigr| \le L'_u \cdot
   \max_{1 \le k \le r_u - 1} \bigl| p_{u,k}(\xi_u) \bigr| \cdot |\xi_u|
 \end{equation}
 for any $\xi_u \in \Omega_u$ with $|\xi_u| \le 1$.

 For the simplicity of notation, we put
 \[ 
  L:= \max_{1 \le u \le l}\Bigl( \max_{r_u+1 \le h \le d-l} L_{u,h} \Bigr)
  \quad \text{and} \quad L':= \max_{1 \le u \le l} L'_u.
 \]
 For any positive real numbers $R$ and
 $1> \epsilon >0$,
 and for any open neighborhood $U_0$ of $0$ in $\mb{C}^{d-l}$,
 we take a positive real number $\delta$ such that the inequality
 \[ 
  0 < \delta < \min \left\{1,\,\, \frac{\epsilon}{3(l-1)(d-l)RL},\,\, 
                           \frac{\epsilon}{3L'} \right\}
 \]
 holds and that the set
 \[
  U:= \left\{ \xi \in \mb{C}^{d-l} \Bigm| |\xi| < \delta \right\}
 \]
 is included in $U_0$.

 Then for any $A=(a_{u,k,v,h}) \in D_R$ and $\xi \in U$,
 we have
 \begin{align*}
  \max_{u,k}
  \left| \sum_{v=1}^l \sum_{h = r_v+1}^{d-l} a_{u,k,v,h} p_{v,h}(\xi_v) \right|
   &\le \sum_{v=1}^l \sum_{h = r_v+1}^{d-l} R \cdot L_{v,h} \cdot |\xi_v| \cdot
        \max_{1 \le k \le r_v} \bigl| p_{v,k}(\xi_v) \bigr| \\
   &\le \frac{\epsilon}{3} \cdot \max_{u,k}\bigl| p_{u,k}(\xi_u) \bigr|
 \end{align*}
 by the inequality~(\ref{eq.6.4.2}),
 which implies
 \begin{equation}\label{eq.6.etap}
  \begin{split}
   |\eta | =
    \max_{u,k}\left| \eta_{u,k} \right|
    &\ge \max_{u,k} \bigl| p_{u,k}(\xi_u) \bigr| - \max_{u,k} \left|
     \sum_{v=1}^l \sum_{h = r_v+1}^{d-l} a_{u,k,v,h} p_{v,h}(\xi_v) \right| \\
    &\ge \frac{2}{3} \max_{u,k} \bigl| p_{u,k}(\xi_u) \bigr|.
  \end{split}
 \end{equation}

 On the other hand,
 for $A=(a_{u,k,v,h}) \in D_R$ and $\xi \in U \cap \Omega$,
 we have $\xi_u \in \Omega_u$ for some $u$, which implies
 \begin{align*}
  |\eta_{u,r_u}|
  &\le \bigl| p_{u,r_u}(\xi_u) \bigr| + \left| \sum_{v=1}^l
      \sum_{h = r_v+1}^{d-l} a_{u,r_u,v,h} p_{v,h}(\xi_v) \right| \\
  &\le L'_u \cdot
     \max_{1 \le k \le r_u - 1} \bigl| p_{u,k}(\xi_u) \bigr| \cdot |\xi_u|
     + \frac{\epsilon}{3} \cdot \max_{u,k} \bigl| p_{u,k}(\xi_u) \bigr| \\
  &\le \frac{2\epsilon}{3} \cdot \max_{u,k} \bigl| p_{u,k}(\xi_u) \bigr|
   \le \epsilon \cdot |\eta |
 \end{align*}
 by the inequality~(\ref{eq.6.4.3}).
 Therefore we have
 \begin{lemma}\label{lm.6.5}
  For $(\xi, A) \in (U \cap \Omega) \times D_R$,
  we have $F(\xi, A) \notin \Xi_{\epsilon}\times D_R$.
 \end{lemma}

 We put
 \[
  \mu_u := \min_{\xi_u \in Z_u \setminus \Omega_u} \max_{1 \le k \le r_u}
   \bigl| p_{u,k}(\xi_u) \bigr|
   \quad \text{and} \quad
  \mu := \min_{1 \le u \le l} \mu_u.
 \]
 Then $\mu$ is positive
 by the compactness of $Z_u \setminus \Omega_u$ for each $u$.
 Moreover
 by the homogeneity of $p_{u,k}(\xi_u)$, we have the inequality
 \begin{equation}\label{eq.6.4.4}
  \max_{1 \le k \le r_u}\bigl| p_{u,k}(\xi_u)\bigr| \ge \mu_u |\xi_u|^{r_u}
 \end{equation}
 for any $\xi_u \in \mb{C}^{r_u} \setminus \Omega_u$ with $|\xi_u| \le 1$.
 We put $r := \max_{u}r_u$.
 \begin{lemma}\label{lm.6.6}
  For $(\xi,A) \in (U \setminus \Omega) \times D_R$,
  we have
   $|\eta | \ge \frac{2}{3} \mu |\xi|^r$.
 \end{lemma}
 \begin{proof}
  For $\xi \in U \setminus \Omega$,
  we have $\xi_u \notin \Omega_u$ for any $u$.
  Hence for $(\xi,A) \in (U \setminus \Omega) \times D_R$,
  by the inequalities~(\ref{eq.6.etap}) and~(\ref{eq.6.4.4}),
  we have
  \[
   |\eta | \ge \tfrac{2}{3} \max_{u,k}\bigl| p_{u,k}(\xi_u) \bigr|
    \ge \tfrac{2}{3} \max_{1 \le u \le l} \mu_u |\xi_u|^{r_u}
    \ge \tfrac{2}{3} \mu |\xi|^r.
  \]
 \end{proof}
 We put
 \[ 
  W := \left\{\left. \eta= (\eta_{u,k}) \in \mb{C}^{d-l} \ \right| \
       |\eta | < \tfrac{2}{3} \mu \cdot \delta^r
        \right\}.
 \]
 Then Lemma~\ref{lm.6.5} implies the inclusion relation
 \[
  \left(U \times D_R \right) \cap F^{-1}\left( W_{\epsilon} \times D_R \right)
   \subseteq (U \setminus \Omega) \times D_R.
 \]
 Therefore for any $(\xi, A) \in
 \left(U \times D_R \right) \cap F^{-1}\left( W_{\epsilon} \times D_R \right)$,
 we have the inequality $|\eta | \ge \frac{2}{3} \mu |\xi|^r$
 by Lemma~\ref{lm.6.6},
 which assures that the map~(\ref{eq.6.4.1}) is proper.
 Hence by Lemma~\ref{lm.5.11.2} the map~(\ref{eq.6.4.1}) is a finite branched covering.
\end{proof}

\begin{proposition}\label{pr.6.5}
 The degree of the branched covering map~(\ref{eq.6.4.1})
 defined in Proposition~\ref{pr.6.4} is equal to the right hand side of 
 the equality~(\ref{eq.4.B}) in Theorem~\ref{thm.B}.
\end{proposition}
\begin{proof}
 To consider the map $F$ on
 \[
  \widetilde{W}'
    := \left\{(\eta,0) \in W_{\epsilon} \times D_R \bigm| \eta_{u,k}=0
       \ \ \text{for $1 \le u \le l$ and $1 \le k \le r_u -1$} \right\},
 \]
 we define the map $F_u : \mb{C}^{r_u} \to \mb{C}^{r_u}$ by 
  $F_u(\xi_u) = (p_{u,1}(\xi_u), \ldots, p_{u,r_u}(\xi_u))$,
 and put
 \[
  X_u := \left\{ \xi_u \in \mb{C}^{r_u} \ \left| \
           \begin{split}
	    p_{u,k}(\xi_u) &= 0 \ \text{ for } 1 \le k \le r_u -1 \\
	    p_{u,r_u}(\xi_u) &\ne 0
	   \end{split} 
          \right. \right\}
 \]
 for each $u$.
 We consider the following two lemmas:
 \begin{lemma}\label{lm.6.8}
  The Jacobian of the map $F_u$ is not zero at any point of $X_u$.
 \end{lemma}
 \begin{lemma}\label{lm.6.9}
  The degree of the map
  $p_{u,r_u}|_{X_u} : X_u \to \mb{C}^{*}$
  is 
  $r_u \cdot \#\left( S_{r_u + 1}(\la_{I_u}) \right)$,
  where we define $\#\left( S_{r_u + 1}(\la_{I_u}) \right)=1$ if $r_u \le 2$.
 \end{lemma}
 Lemma~\ref{lm.6.8} assures
 that the branched covering map~(\ref{eq.6.4.1}) is unbranched
 on some neighborhood of $\widetilde{W}'$ in $W_{\epsilon}\times D_R$,
 that $X_u$ is a smooth Riemann surface,
 and that the map $p_{u,r_u}|_{X_u} : X_u \to \mb{C}^{*}$
 is unbranched.
 Therefore the degree of the map~(\ref{eq.6.4.1}) is equal to that of the map
 $\left(U \times D_R \right) \cap F^{-1}(\widetilde{W}')
  \stackrel{F}{\to} \widetilde{W}'$,
 which is also equal to
 $\prod_{1 \le u \le l} \deg\left(p_{u,r_u}|_{X_u}\right)$;
 hence Lemmas~\ref{lm.6.8} and~\ref{lm.6.9} imply the proposition.

 We show Lemma~\ref{lm.6.8} first. Since
  $p_{u,k}(\xi_u) = \sum_{i=1}^{r_u} m_{u,i}\xi_{u,i}^k$,
 we have
 \[ 
  \det (dF_u)(\xi_u) = r_u!\cdot \prod_{i=1}^{r_u}m_{u,i}\cdot
    \prod_{1 \le i < j \le r_u}(\xi_{u,j} - \xi_{u,i})
 \]
 by a similar computation to the proof of Lemma~\ref{lm.4.1}.
 Hence the Jacobian is not equal to zero if and only if
 $\xi_{u,1},\ldots,\xi_{u,r_u}$ are mutually distinct.
 On the other hand,
 by a similar argument to the proof of Lemma~\ref{lm.4.6},
 we find that
 for a common zero $\xi_u = (\xi_{u,1},\ldots,\xi_{u,r_u})$
 of $p_{u,1},\ldots,p_{u,r_u-1}$,
 the inequality $p_{u,r_u}(\xi_u) \ne 0$ holds
 if and only if
 $0,\xi_{u,1},\ldots,\xi_{u,r_u}$ are mutually distinct.
 Hence for any $\xi_u \in X_u$, the Jacobian $\det (dF_u)(\xi_u)$ is not
 zero,
 which completes the proof of Lemma~\ref{lm.6.8}.

 We show Lemma~\ref{lm.6.9} next.
 Since $p_{u,k}(\xi_u)$ is homogeneous for any $u$ and $k$,
 the Riemann surface $X_u$ is invariant under the action of $\mb{C}^*$;
 hence the set
 \[
  \left\{ (\xi_{u,1}:\cdots: \xi_{u,r_u}) \in \mb{P}^{r_u-1} \bigm|
          (\xi_{u,1},\ldots,\xi_{u,r_u}) \in X_u \right\}
 \]
 is well-defined and
 is equal to
 $S_{r_u+1}\left(\la_{I_u}\right)$ by definition.
 Therefore $X_u$ consists of\\
 $\#\left(S_{r_u+1}\left(\la_{I_u}\right)\right)$
 components, each of which is biholomorphic to $\mb{C}^*$.
 Moreover on each component of $X_u$, the degree of the map $p_{u,r_u}$
 is $\deg p_{u,r_u} = r_u$,
 which completes the proofs of Lemma~\ref{lm.6.9} and the proposition.
\end{proof}

On the basis of Propositions~\ref{pr.6.4} and~\ref{pr.6.5},
we prove the following:

\begin{proposition}\label{pr.6.6}
 Let $\psi_k(\xi)$ be the expression
 defined in the equality~(\ref{eq.5.defpsi}).
 Then the number
 \[
  \mult_{0}(\psi_1,\ldots,\psi_{d-l})
 \]
 is equal to the right hand side of the equality~(\ref{eq.4.B})
 in Theorem~\ref{thm.B}.
\end{proposition}
\begin{proof}
 We define the map $\Psi : \mb{C}^{d-l} \to \mb{C}^{d-l}$ by
  $\Psi(\xi):=(\psi_k(\xi))_{1\le k \le d-l}$,
 and put
 \[
  Y:= \left\{\xi\in\mb{C}^{d-l} \bigm|
     \psi_1(\xi)=\cdots = \psi_{d-l-1}(\xi)=0,\ \psi_{d-l}(\xi) \ne 0\right\}.
 \]
 We denote by $M_{(r_1,\ldots,r_l)}$ the square matrix $M$ defined
 in Lemma~\ref{lm.5.5}.
 \begin{lemma}\label{lm.6.11}
  For any open neighborhood $\widetilde{U}'$ of $0$ in $\mb{C}^{d-l}$,
  there exist open neighborhoods $U', W'$ of $0$
  with $U' \subset \widetilde{U}'$ and $W' \subset \mb{C}$ such that
  $Y \cap U'$ is a smooth Riemann surface,
  that the map
  \begin{equation}\label{eq.6.6.1}
   Y \cap U' \cap \psi_{d-l}^{-1}(W'\setminus \{0\}) \stackrel{\psi_{d-l}}{\to}
    W'\setminus \{0\}
  \end{equation}
  is an unbranched covering,
  and that the number $\mult_{0}(\psi_1,\ldots,\psi_{d-l})$
  is equal to the degree of the map~(\ref{eq.6.6.1}).
 \end{lemma}
 \begin{proof}
  First we shall check that 
  $\det (d\Psi)(\xi) \ne 0$ holds for any $\xi \in Y\cap U'$,
  if we take $U'$ sufficiently small.
  By a similar argument to the proof of Lemma~\ref{lm.4.1},
  the equality $\det (d\Psi)(\xi) = 0$ holds for $\xi \in U'$
  if and only if $\alpha_u+\xi_{u,i} =\alpha_v+\xi_{v,j}$ holds
  for some $u,i,v$ and $j$ with $(u,i)\ne (v,j)$,
  which is equivalent to the condition that
  $\xi_{u,i}=\xi_{u,j}$ holds for some $u,i$ and $j$ with $i \ne j$
  if we take $U'$ sufficiently small.
  Suppose for instance
  that $\xi_{1,1}= \xi_{1,2}$ holds for some $\xi \in Y\cap U'$.
  Then putting $\Psi'(\xi):=(\psi_k(\xi))_{1\le k \le d-l-1}$,
  considering the map $M_{(r_1-1,r_2,\ldots,r_l)}^{-1}\circ\Psi'$,
  and keeping in mind the inequalities~(\ref{eq.6.etap}),
  we have $p_{u,k}(\xi)=0$ for any $u$ and $k$, which contradicts
  $\psi_{d-l}(\xi)\ne 0$.
  Therefore we have $\det (d\Psi)(\xi) \ne 0$ for any $\xi \in Y\cap U'$,
  which assures that $Y\cap U'$ is a smooth Riemann surface, and that
  the map~(\ref{eq.6.6.1}) is an unbranched covering if we take $W'$
  sufficiently small.
  Moreover
  since $\det (d\Psi)(\xi) \ne 0$ for any $\xi \in Y\cap U'$,
  we have $\mult_{Y'}(\psi_1,\ldots,\psi_{d-l-1}) =1$
  for any connected component $Y'$ of $Y \cap U'$;
  hence we have
  $\mult_{0}(\psi_1,\ldots,\psi_{d-l})
    = \mult_0(\overline{Y} \cap U' , \psi_{d-l})$
  by definition,
  where $\overline{Y} \cap U'$
  is the closure of $Y \cap U'$ in $U'$.
  Since $\mult_0(\overline{Y} \cap U' , \psi_{d-l})$ is clearly equal to
  the degree of the covering map~(\ref{eq.6.6.1}),
  all the assertions in Lemma~\ref{lm.6.11} are verified.
 \end{proof}

 We proceed the proof of the proposition.
 It is clear that there exists $A=(a_{u,k,v,h}) \in \mb{C}^{(l-1)(d-l)^2}$
 such that
 the equality $F(\xi,A)=(M_{(r_1,\ldots,r_l)}^{-1}\circ\Psi(\xi),A)$
 holds for any $\xi\in \mb{C}^{d-l}$.
 Let $e$ be the $(d-l,1)$ column vector whose $(d-l)$-th entry is $1$
 and whose other entries are $0$.
 Moreover we put $M_{(r_1,\ldots,r_l)}^{-1}e=:\eta=
 (\eta_{u,k})_{1\le u \le l,1\le k\le r_u}$.
 Then the equality $Y\times \{A\} = F^{-1}(\mb{C}\eta\setminus\{0\},A)$ holds,
 and the map $F|_{Y\times \{A\}}$ is equal
 to the map $M_{(r_1,\ldots,r_l)}^{-1}\circ\Psi|_{Y}$.
 Hence,
 if we can show $\eta_{u,r_u} \ne 0$ for $1 \le u \le l$, then we have
 $\left(\mb{C}\eta \setminus \{0\}\right) \cap W \subseteq W_{\epsilon}$
 for some $\epsilon$, which assures that
 the degree of the covering map~(\ref{eq.6.6.1}) is equal to that
 of the branched covering map~(\ref{eq.6.4.1});
 thus the proposition will be verified
 by Proposition~\ref{pr.6.5} and Lemma~\ref{lm.6.11}.

 We show $\eta_{u,r_u} \ne 0$ for $1 \le u \le l$.
 Suppose $\eta_{l,r_l}=0$ for instance,
 and put $\eta'= {^t}(\eta_{1,1},\ldots,\eta_{l,r_l-1})\in \mb{C}^{d-l-1}$
 so that the equality $\eta = {^t}({^t}\eta',0)$ holds.
 Then by the equality $e=M_{(r_1,\ldots,r_l)}\eta$,
 we have $0 = M_{(r_1,\ldots,r_{l-1},r_l-1)}\eta'$.
 Since $M_{(r_1,\ldots,r_{l-1},r_l-1)}$ is invertible,
 we have $\eta' =0$,
 which implies $\eta = 0$ and
 the contradiction $e = M_{(r_1,\ldots,r_l)}0 = 0$.
 Therefore $\eta_{u,r_u} \ne 0$ holds for any $1 \le u \le l$,
 which completes the proof of the proposition.
\end{proof}

We complete the proof of Theorem~\ref{thm.B}.

\begin{proof}[Proof of Theorem~\ref{thm.B}]
 Remember the definition of $\mb{I}(\alpha) \in \mathfrak{I}(\la)$
 for $\alpha \in B_d(\la)$ in the proof of Lemma~\ref{lm.4.6}.
 By Lemma~\ref{lm.6.2}, we can easily verify that
 for any $\alpha \in B_d(\la)$
 there exists an open neighborhood $O_{\alpha}$ of $\alpha$ in $\Pd$
 such that the equality
 \[
  \left\{\z \in O_{\alpha} \bigm|
   \varphi_k(\z) = 0 \
   \text{ for } 1 \le k \le d-\#\left(\mb{I}(\alpha)\right) \right\}
   = B_d(\la) \cap O_{\alpha}
 \]
 holds,
 which implies
 the first two assertions
 in Theorem~\ref{thm.B}.
 On the other hand, the last assertion
 in Theorem~\ref{thm.B} is 
 the direct consequence of Propositions~\ref{pr.5.0} and~\ref{pr.6.6}.
\end{proof}

At the end of this section, we prove Proposition~\ref{pr.D}.

\begin{proof}[Proof of Proposition~\ref{pr.D}]
 For the brevity of notation, we put
 \begin{align*}
  \mathfrak{I}'(\la) &:=
   \mathfrak{I}(\la) \cup \Bigl\{\bigl\{\{1,\ldots,d\}\bigr\}\Bigr\}
 \quad \text{for $\la \in V_d$,} \\
  e_{\mb{I}}(\la) &:= 
   \mult_{E_d(\mb{I})}(\varphi_1, \ldots, \varphi_{d-\#\left(\mb{I}\right)})
  \quad \text{for each $\mb{I}\in\mathfrak{I}(\la)$, and} \\
  e_{\{\{1,\ldots,d\}\}}(\la) &:= (d-1)\cdot \#\left(S_d(\la)\right).
 \end{align*}
 Note that $\{\{1,\ldots,d\}\}$ is the only minimum element
 of $\mathfrak{I}'(\la)$ with respect to the partial order $\prec$.

 Under the notation above,
 the equality~(\ref{eq.4.C}) in Proposition~\ref{pr.C} is equivalent 
 to the equality
 \begin{equation}\label{eq.6.C}
  (d-1)! = \sum_{\mb{I} \in \mathfrak{I}'(\la)} \left( e_{\mb{I}}(\la)\cdot
   \prod_{k= d - \#\left(\mb{I}\right) +1}^{d-1} k \right),
 \end{equation}
 whereas the equality~(\ref{eq.4.B}) in Theorem~\ref{thm.B} is rewritten
 in the form
 \begin{equation}\label{eq.6.B}
  e_{\mb{I}}(\la) = \prod_{u=1}^l e_{\{I_u\}}\left(\la_{I_u}\right)
   = \prod_{I \in \mb{I}}e_{\{I\}}\left(\la_I\right),
 \end{equation}
 where $\mb{I} = \left\{I_1,\ldots,I_l\right\} \in \mathfrak{I}(\la)$,
 and $\{I\}$ denotes the minimum element
 of the set $\mathfrak{I}'\left(\la_I\right)$
 for each $I \in \mathcal{I}(\la)$.
 On the other hand, Proposition~\ref{pr.D} is rewritten in the form
 \begin{equation}\label{eq.6.D}
  \prod_{u=1}^l \bigl(\#\left(I_u \right)-1\bigr)!
   =  \sum_{\mb{I}'\in \mathfrak{I}(\la),\,\, \mb{I}' \succ \mb{I}} \left(
       e_{\mb{I}'}(\la) \cdot
       \prod_{u=1}^l \left(
        \prod_{k=\#\left(I_u\right)-\chi_u\left(\mb{I}'\right)+1}^{\#\left(
           I_u\right)-1} k
       \right)\right)
 \end{equation}
 for $\mb{I} = \left\{I_1,\ldots,I_l\right\} \in \mathfrak{I}(\la)$,
 where $\chi_u(\mb{I}')$ is the one defined in
 Main Theorem~\ref{mthm.3}.
 Note that $\mb{I}\succ\mb{I}$ holds for any $\mb{I}\in\mathfrak{I}'(\la)$.
 To complete the proof of Proposition~\ref{pr.D},
 we only need to derive the equality~(\ref{eq.6.D})
 from the equalities~(\ref{eq.6.C}) and~(\ref{eq.6.B}).

 Note that
 for $\mb{I} = \left\{I_1,\ldots,I_l\right\} \in \mathfrak{I}'(\la)$,
 we have
 \[
  \left\{ \mb{I}' \in \mathfrak{I}'(\la) \bigm| \mb{I}' \succ \mb{I} \right\}
   = \left\{ \mb{I}_1 \cup \cdots \cup \mb{I}_l \bigm| 
      \mb{I}_u \in \mathfrak{I}'\left(\la_{I_u}\right)
       \ \text{ for } 1 \le u \le l \right\}
 \]
 by definition.
 Hence we have the following equalities
 for $\mb{I} = \{I_1,\ldots,I_l\} \in \mathfrak{I}(\la)$
 from the equalities~(\ref{eq.6.C}) and~(\ref{eq.6.B}):
 \begin{align*}
  \prod_{u=1}^l \bigl(\#\left(I_u \right)-1\bigr)!
   &= \prod_{u=1}^l \left(
       \sum_{\mb{I}_u \in \mathfrak{I}'\left(\la_{I_u}\right)} \left(
        e_{\mb{I}_u}\left(\la_{I_u}\right) \cdot 
        \prod_{k=\#\left(I_u\right)-\#\left(\mb{I}_u\right)+1}^{\#\left(
           I_u\right)-1} k
      \right)\right) \\
   &= \sum_{\mb{I}_1 \in \mathfrak{I}'\left(\la_{I_1}\right)} \cdots
      \sum_{\mb{I}_l \in \mathfrak{I}'\left(\la_{I_l}\right)}
       \prod_{u=1}^l \left(
         \prod_{I'_u \in \mb{I}_u} e_{\{I'_u\}}\left(\la_{I'_u}\right) \cdot
        \prod_{k=\#\left(I_u\right)-\#\left(\mb{I}_u\right)+1}^{\#\left(
           I_u\right)-1} k
       \right) \\
   &= \sum_{\mb{I}_1 \in \mathfrak{I}'\left(\la_{I_1}\right)} \cdots
      \sum_{\mb{I}_l \in \mathfrak{I}'\left(\la_{I_l}\right)} \left(
       e_{\mb{I}_1 \cup \cdots \cup \mb{I}_l}(\la) \cdot
       \prod_{u=1}^l \left(
        \prod_{k=\#\left(I_u\right)-\#\left(\mb{I}_u\right)+1}^{\#\left(
           I_u\right)-1} k
       \right) \right) \\
   &= \sum_{\mb{I}'\in \mathfrak{I}(\la),\,\, \mb{I}' \succ \mb{I}} \left(
       e_{\mb{I}'}(\la) \cdot
       \prod_{u=1}^l \left(
        \prod_{k=\#\left(I_u\right)-\chi_u\left(\mb{I}'\right)+1}^{\#\left(
           I_u\right)-1} k
       \right)\right).
 \end{align*}
  The equality~(\ref{eq.6.D}) is thus obtained, which completes
 the proof of Proposition~\ref{pr.D}.
\end{proof}

\section{Relation between the sets $S_d(\la)$ and
$\Phi_d^{-1}\left(\bar{\la}\right)$}\label{sec.7}

In this section we state the explicit relation between the 
cardinalities  $\#\left(S_d(\la)\right)$ and\\ 
$\#\left(\Phi_d^{-1}\left(\bar{\la}\right)\right)$.
Let $\la$ be an element of $V_d$, which is fixed throughout this section.
Remember the definitions of $K_1,\ldots,K_q$, $\kappa_1,\ldots,\kappa_q$,
$g_1,\ldots,g_q$ defined in Definition~\ref{df.1.6},
and $\mathfrak{S}\left(\mathcal{K}(\la)\right)$
defined in Definition~\ref{df.2.2}.
We put
\[
 \Sigma_d(\la) := \left\{(\z_1:\cdots:\z_d) \in \mb{P}^{d-1} \
        \left|\ \begin{matrix}
		 \sum_{i=1}^d \z_i = 0 \\
		 \sum_{i=1}^d m_i\z_i^k = 0 \quad 
		  \textrm{for} \quad 1 \le k \le d-2 \\
		 \z_1,\ldots,\z_d \textrm{ are mutually distinct}
		\end{matrix}
        \right. \right\}.
\]

\begin{proposition}\label{pr.7.1}
 The bijection
  $\tilde{\iota} : \Sigma_d(\la) \to S_d(\la)$
 is
 defined by
 \[
  (\z_1:\cdots:\z_d) \mapsto (\z_1 - \z_d:\cdots:\z_{d-1}-\z_d).
 \]
 The group $\mathfrak{S}(\mathcal{K}(\la))$ acts on $\Sigma_d(\la)$
 by the permutation of the homogeneous coordinates.
 Moreover the actions of $\mathfrak{S}(\mathcal{K}(\la))$
 on $S_d(\la)$ and $\Sigma_d(\la)$ commute with the map $\tilde{\iota}$;
 hence we have the bijection
  $\Sigma_d(\la)/\mathfrak{S}(\mathcal{K}(\la)) \stackrel{\cong}{\to}
   \Phi_d^{-1}\left(\bar{\la}\right)$.
\end{proposition}
\begin{proof}
 The bijectivity of the map $\iota(\la)$ in
 Proposition~\ref{pr.2.8}
 implies the proposition.
\end{proof}

\begin{lemma}\label{lm.7.2}
 Let $\z=(\z_1:\cdots:\z_d)$ be an element of $\Sigma_d(\la)$
 and suppose that there exists
 a non-identity permutation $\sigma \in \mathfrak{S}(\mathcal{K}(\la))$
 with $\sigma\cdot\z=\z$.
 Then there exists a unique suffix $i$ with $\z_i = 0$.
 Moreover if $i\in K_w$, then the fixing subgroup
  $\left\{\sigma\in\mathfrak{S}\left(\mathcal{K}(\la)\right)
    \bigm| \sigma\cdot\z=\z\right\}$
 of $\z$ is a cyclic group whose order divides $g_w$.
\end{lemma}
\begin{proof}
 For any $\sigma\in\mathfrak{S}\left(\mathcal{K}(\la)\right)$
 with $\sigma\cdot\z=\z$,
 there exists a non-zero complex number $a$
 satisfying
 $\z_{\sigma^{-1}(i)}=a\z_i$ for $1\le i\le d$,
 which induces the injective group homomorphism
 \[
 \mathfrak{S}(\z) := \left\{\sigma\in\mathfrak{S}\left(\mathcal{K}(\la)\right)
  \bigm| \sigma\cdot\z=\z\right\} \ni \sigma
  \stackrel{\mathfrak{a}}{\mapsto}
  a \in \left\{a \in \mb{C}^* \bigm| |a|=1 \right\}.
 \]

 In the following,
 we fix non-identity $\sigma\in\mathfrak{S}(\z)$, and
 denote by $t$ the order of $\sigma$.
 Then
 $a=\mathfrak{a}(\sigma)$ is a primitive $t$-th radical root of $1$.
 Moreover
 the cardinality $\#\left(\left\{\sigma^s(i)\bigm|s\in\mb{Z}\right\}\right)$
 is equal to $1$ or $t$ according as $\z_i$ is equal to $0$ or not.

 Suppose that $\z_i \ne 0$ holds for any $i$.
 Then $t$ is a common divisor of $\kappa_1,\ldots,\kappa_q$.
 We may assume
 \[
  m=(\underbrace{m_1,\ldots,m_1}_{t},\ldots,
      \underbrace{m_{d/t},\ldots,m_{d/t}}_{t})
 \]
 and
 \[
  \z=(\z_1:a\z_1:\cdots:a^{t-1}\z_1:\cdots:
      \z_{d/t}:a\z_{d/t}:\cdots:a^{t-1}\z_{d/t}).
 \]
 Under the above notation,
 the equations $\varphi_k(\z)=0$ for $1 \le k \le d-2$
 are equivalent to the equations
  $\sum_{i=1}^{d/t}m_i\z_i^{tk} = 0$
 for $1 \le k \le \frac{d}{t}-1$,
 which implies $m_i=0$ for any $i$
 by the mutual distinctness of $0,\z_1^t,\ldots,\z_{d/t}^t$.
 We thus obtain contradiction,
 which assures the existence of $i$ with $\z_i=0$.

 Next we suppose $\z_i=0$ and $i \in K_w$.
 Then for any $\sigma\in\mathfrak{S}(\z)$,
 the order $t$ of $\sigma$ is a common divisor
 of $\kappa_1,\ldots,\kappa_{w-1}$, $\kappa_w-1$,
 $\kappa_{w+1},\ldots,\kappa_q$, i.e., a divisor of $g_w$.
 Therefore $\mathfrak{S}(\z)$ is isomorphic
 to a subgroup of $\left\{a \in \mb{C}^* \bigm| a^{g_w}=1\right\}$
 by the map $\mathfrak{a}$,
 which completes the proof.
\end{proof}

Remember the definitions of $d[t]$ and $\la[t]$ in Definition~\ref{df.1.6}.
In the following, the symbol $a|b$ denotes
that $a$ divides $b$ for positive integers $a$ and $b$.

\begin{vartheorem}\label{thm.E}
 If we put $s_d(\la):= \#\left(S_d(\la)\right) = \#\left(\Sigma_d(\la)\right)$
 for $\la \in V_d$,
 then the third and fourth steps in Main Theorem~\ref{mthm.3} hold.
\end{vartheorem}
\begin{proof}
 For each
 $t \in \bigcup_{1\le w \le q}\left\{t \bigm| t|g_w \right\}$, we put
 \[
  \Theta_t(\la) := \left\{
   C \in \Sigma_d(\la) / \mathfrak{S}\left(\mathcal{K}(\la)\right) \ \left| \
   \#(C) = \frac{\#\left(\mathfrak{S}\left(\mathcal{K}(\la)\right)\right)}{t}
  \right. \right\}
 \]
 and
  $c_t(\la) := \#\left(\Theta_t(\la)\right)$.
 Then by Proposition~\ref{pr.7.1} and Lemma~\ref{lm.7.2}, we have
 \[ 
  \Phi_d^{-1}\left(\bar{\la}\right) \stackrel{\cong}{\gets}
  \Sigma_d(\la) / \mathfrak{S}\left(\mathcal{K}(\la)\right) =
  \left( \coprod_{w=1}^q \left( \coprod_{t|g_w,\ t\ge 2}
   \Theta_t(\la) \right) \right)
  \coprod \Theta_1(\la),
 \]
 which implies the equalities~(\ref{eq.7.E2}) and~(\ref{eq.7.E3}).
 Hence to complete the proof,
 we only need to show the equalities~(\ref{eq.7.E1}) for each $t$ with $t \geq 2$.
 In the rest of the proof, we fix $1 \le w \le q$.

 For each $t$ with $t|g_w$ and $t\ge 2$, we define
 the group $\mathfrak{S}\left(\mathcal{K}'\left(\la[t]\right)\right)$
 to be isomorphic to 
 $\mathfrak{S}_{\frac{\kappa_1}{t}}\times\cdots\times
   \mathfrak{S}_{\frac{\kappa_w-1}{t}}\times\cdots\times
   \mathfrak{S}_{\frac{\kappa_q}{t}}$.
 Then $\mathfrak{S}\left(\mathcal{K}'\left(\la[t]\right)\right)$
 naturally acts on $S_{d[t]}(\la[t])$, 
 and we have 
 $\mathfrak{S}\left(\mathcal{K}'\left(\la[t]\right)\right) \subseteq
  \mathfrak{S}\left(\mathcal{K}\left(\la[t]\right)\right)$.
 Note that in some cases
 the equality
 $\mathfrak{S}\left(\mathcal{K}'\left(\la[t]\right)\right) =
  \mathfrak{S}\left(\mathcal{K}\left(\la[t]\right)\right)$
 does not hold, e.g., $\la[2]$ in Example~\ref{ex.0.3} in Section~\ref{sec.1.1.2}.
 For each divisor $b$ of $\frac{g_w}{t}$,
 we put
 \[
  \Theta'_b(\la[t]) := \left\{ C'\in S_{d[t]}(\la[t])/
   \mathfrak{S}\left(\mathcal{K}'\left(\la[t]\right)\right)\ \left|\
   \#(C') = \frac{
            \#\left(\mathfrak{S}\left(\mathcal{K}'\left(\la[t]\right)\right)
           \right)}{b} \right. \right\}.
 \]
 Then we have
 \begin{equation}\label{eq.7.E4} 
  S_{d[t]}(\la[t])/
   \mathfrak{S}\left(\mathcal{K}'\left(\la[t]\right)\right) =
   \coprod_{b|(g_w/t)}\Theta'_b(\la[t])
 \end{equation}
 by a similar argument to the proof of Lemma~\ref{lm.7.2}.

 Let $t,\ b$ be positive integers with $t|b$, $b|g_w$ and $t\ge 2$,
 and $a$ a primitive $b$-th radical root of $1$.
 Then a point
 \[
  \left(\z_1 : a\z_1 : \cdots : a^{b-1}\z_1 : \cdots
   : \z_{d[b]-1} : a\z_{d[b]-1} : \cdots : a^{b-1}\z_{d[b]-1} : 0\right)
  \in \mb{P}^{d-1}   
 \]
 represents an element of $\Theta_b(\la)$
 if and only if
 \[
  \left(\z_1^t : a^t\z_1^t : \cdots : a^{t\left((b/t)-1\right)}\z_1^t : \cdots
   : \z_{d[b]-1}^t : a^t\z_{d[b]-1}^t : \cdots
   : a^{t\left((b/t)-1\right)}\z_{d[b]-1}^t\right) \in \mb{P}^{d[t]-2}
 \]
 represents an element of $\Theta'_{b/t}(\la[t])$,
 which gives the bijection
 between $\Theta_b(\la)$ and $\Theta'_{b/t}(\la[t])$.
 The bijection and the equality~(\ref{eq.7.E4}) imply
 the equalities~(\ref{eq.7.E1}), which completes the proof of
 the theorem.
\end{proof}

\section{Completion of the proof}\label{sec.8}

In Propositions~\ref{pr.mt5},~\ref{pr.mt7},~\ref{pr.mt1} and~\ref{pr.mt4},
we had already proved
the assertions~(\ref{en.mt5}),~(\ref{en.mt7}),~(\ref{en.mt1}) and~(\ref{en.mt4})
in Main Theorem~\ref{mthm.1}.
In this section
we complete the rest of the proofs of the main theorems.

\begin{proposition}\label{pr.mt2}
 Main Theorem~\ref{mthm.3} and
 the assertion~(\ref{en.mt2}) in Main Theorem~\ref{mthm.1} hold.
\end{proposition}
\begin{proof}
 These two are the direct consequences of
 Theorem~\ref{thm.B}, Propositions~\ref{pr.C},~\ref{pr.D} and 
 Theorem~\ref{thm.E}.
\end{proof}

\begin{proposition}
 Main Theorem~\ref{mthm.2} and
 the assertion~(\ref{en.mt3}) in Main Theorem~\ref{mthm.1} hold.
\end{proposition}
\begin{proof}
 In the following, we always identify $V_d$ with
 $\left\{(m_1,\ldots,m_d) \in (\mb{C}^*)^d \ \left| \
   \sum_{i=1}^d m_i = 0
   \right.\right\}$
 by the correspondence $m_i = \frac{1}{1- \la_i}$,
 and define the following spaces:
 \begin{align*}
  \mathrm{MP}'_d &:= \Phi_d^{-1}(\widetilde{V}_d), \\
  \mathcal{X}_d &:= 
   \left\{(\z_1,\ldots,\z_d,\rho) \in \mb{C}^d \times \mb{C}^* \bigm|
     \z_1,\ldots,\z_d \text{ are mutually distinct}\right\}, \\
  \widetilde{\mathcal{X}}_d &:= \mathcal{X}_d / \mathrm{Aut}(\mathbb{C}), \\
  (\mathcal{PX})_d &:= \left\{(\z_1,\ldots,\z_d) \in \mb{C}^d \bigm|
   \z_1,\ldots,\z_d \text{ are mutually distinct}\right\}, \\
  (\widetilde{\mathcal{PX}})_d &:= (\mathcal{PX})_d / \mathrm{Aut}(\mathbb{C}), \\
  (\mathcal{PV})_d &:=\left\{(m_1 : \cdots : m_d) \in \mb{P}^{d-1} \ \left| \
    \sum_{i=1}^d m_i = 0,
   \ m_i \ne 0 \text{ for } 1 \le i \le d   \right.\right\}, \\
  \mathcal{Y}_d &:= \left\{((\z,\rho),m) 
   \in \widetilde{\mathcal{X}}_d \times V_d \ 
   \left| \ \sum_{i=1}^d m_i\z_i^k = \begin{cases}
				      0 & (1 \le k \le d-2) \\
				      -\frac{1}{\rho} & (k=d-1)
				     \end{cases} \right.\right\}, \\
  (\mathcal{PY})_d &:= \left\{(\z,m)
   \in (\widetilde{\mathcal{PX}})_d \times (\mathcal{PV})_d \ \left| \
    \sum_{i=1}^d m_i \z_i^k = 0
   \text{ for } 1 \le k \le d-2 \right.\right\},
 \end{align*}
 where the actions of $\mathrm{Aut}(\mathbb{C})$ on $\mathcal{X}_d$ and $(\mathcal{PX})_d$ 
 are defined by
 \[
  \gamma \cdot (\z_1,\ldots,\z_d,\rho)
   = \left(\gamma(\z_1),\ldots,\gamma(\z_d), a^{-d+1}\rho\right)
  \quad \text{and} \quad
  \gamma \cdot (\z_1,\ldots,\z_d)
   = \left(\gamma(\z_1),\ldots,\gamma(\z_d)\right)  
 \]
 for $\gamma(z)=az+b \in \mathrm{Aut}(\mathbb{C})$,
 $(\z_1,\ldots,\z_d,\rho) \in \mathcal{X}_d$
 and $(\z_1,\ldots,\z_d) \in (\mathcal{PX})_d$.
 Then we have the commutative diagram
 \begin{center}
  \begin{picture}(122,109)(0,-7)
   \put(0,90){$(\widetilde{\mathcal{PX}})_d$}
   \put(60,90){$\widetilde{\mathcal{X}}_d$}
   \put(0,45){$(\mathcal{PY})_d$}
   \put(60,45){$\mathcal{Y}_d$}
   \put(105,45){$\mathrm{MP}'_d$}
   \put(0,0){$(\mathcal{PV})_d$}
   \put(60,0){$V_d$}
   \put(105,0){$\widetilde{V}_d$,}
   \put(55,3){\vector(-1,0){20}}
    \put(44,8){\scriptsize $P$}
    \put(40,-7){\scriptsize $/\mathbb{C}^*$}
   \put(55,48){\vector(-1,0){20}}
    \put(40,38){\scriptsize $/\mathbb{C}^*$}
   \put(55,93){\vector(-1,0){20}}
    \put(40,83){\scriptsize $/\mathbb{C}^*$}
   \put(78,3){\vector(1,0){20}}
    \put(85,8){\scriptsize $\textit{pr}$}
    \put(80,-7){\scriptsize /$\mathfrak{S}_d$}
   \put(78,48){\vector(1,0){20}}
    \put(80,38){\scriptsize /$\mathfrak{S}_d$}
   \put(15,36){\vector(0,-1){20}}
    \put(18,24){\scriptsize $\widetilde{\Phi}'_d$}
   \put(15,60){\vector(0,1){20}}
    \put(18,66){\scriptsize $\cong$}
   \put(65,36){\vector(0,-1){20}}
    \put(68,24){\scriptsize $\Phi'_d$}
   \put(65,60){\vector(0,1){20}}
    \put(68,66){\scriptsize $\cong$}
   \put(110,36){\vector(0,-1){20}}
    \put(113,24){\scriptsize $\Phi_d$}
   \put(79,85){\vector(1,-1){24}}
    \put(89,80){\scriptsize /$\mathfrak{S}_d$}
  \end{picture}
 \end{center}
 where each map is defined to be the natural projection except for the
 maps $\Phi_d$ and
 \[
  \widetilde{\mathcal{X}}_d \ni (\z_1,\ldots,\z_d,\rho) \mapsto
   z+\rho (z-\z_1)\cdots(z-\z_d) \in \mathrm{MP}'_d.
 \]
 Here, the first projection maps $\mathcal{Y}_d \to \widetilde{\mathcal{X}}_d$
 and $(\mathcal{PY})_d \to (\widetilde{\mathcal{PX}})_d$ are isomorphisms.
 The $d$-th symmetric group $\mathfrak{S}_d$ acts
 on $\widetilde{\mathcal{X}}_d$, $\mathcal{Y}_d$ and $V_d$
 by the permutation of coordinates.
 These actions of $\mathfrak{S}_d$ commute with the projection maps
 $\mathcal{Y}_d \stackrel{\cong}{\to} \widetilde{\mathcal{X}}_d$
 and $\Phi'_d:\mathcal{Y}_d \to V_d$.
 Moreover we have the natural isomorphisms
 $\mathcal{Y}_d/\mathfrak{S}_d \cong
  \widetilde{\mathcal{X}}_d/\mathfrak{S}_d \cong \mathrm{MP}'_d$
 and $V_d/\mathfrak{S}_d \cong \widetilde{V}_d$.
 On the other hand, the multiplicative group $\mb{C}^*$ acts
 on $\widetilde{\mathcal{X}}_d$, $\mathcal{Y}_d$ and $V_d$
 by $a \cdot (\z,\rho) = (\z, a^{-1}\rho)$ and
 $a\cdot (m_1,\ldots,m_d) = (a m_1,\ldots,a m_d)$
 for $a \in \mb{C}^*$, $(\z,\rho) \in \widetilde{\mathcal{X}}_d$
 and $(m_1,\ldots,m_d) \in V_d$.
 These actions of $\mb{C}^*$ are free,
 commute with the actions of $\mathfrak{S}_d$,
 and also commute with the projection maps
 $\mathcal{Y}_d \stackrel{\cong}{\to} \widetilde{\mathcal{X}}_d$
 and $\Phi'_d:\mathcal{Y}_d \to V_d$.
 We have the natural isomorphisms
 $\widetilde{\mathcal{X}}_d/\mb{C}^* \cong (\widetilde{\mathcal{PX}})_d
 \cong (\mathcal{PY})_d \cong \mathcal{Y}_d/\mb{C}^*$
 and $V_d/\mb{C}^* \cong (\mathcal{PV})_d$.

 Therefore to analyze the fiber structure of the map
 $\Phi_d|_{\mathrm{MP}'_d}$,
 we only need to consider the second projection map
 $\widetilde{\Phi}'_d : (\mathcal{PY})_d \to (\mathcal{PV})_d$
 and the actions of $\mathfrak{S}_d$ on $\mathcal{Y}_d$ and $V_d$,
 most of which had however already been examined
 since we can make the following identifications as usual:
 \begin{align*}
  (\widetilde{\mathcal{PX}})_d
    &= \left\{\left.(\z_1:\cdots:\z_{d-1}) \in \Pd \ \right| \
     \z_1,\ldots,\z_{d-1},0 \text{ are mutually distinct} \right\}, \\
  (\mathcal{PV})_d &= \left\{(m_1:\cdots:m_{d-1}) \in \Pd \ \left| \
    \sum_{i=1}^{d-1} m_i \ne 0,
   \  m_i \ne 0 \text{ for } 1 \le i \le d-1 \right. \right\}, \\
  (\mathcal{PY})_d &= \left\{(\z,m)
   \in (\widetilde{\mathcal{PX}})_d \times (\mathcal{PV})_d \ \left| \
    \sum_{i=1}^{d-1} m_i \z_i^k = 0
   \text{ for } 1 \le k \le d-2 \right.\right\}.
 \end{align*}
 Especially, we have $(\widetilde{\Phi}'_d)^{-1}(P(\la)) = S_d(\la)$
 for any $\la \in V_d$.

 For each
 $(\mathcal{I},\mathcal{K}) \in 
 \left\{(\mathcal{I}(\la),\mathcal{K}(\la)) \bigm| \la \in V_d\right\}$,
 we put
 \begin{align*}
  \overline{V(\mathcal{I},\mathcal{K})}
   &:= \left\{\la \in V_d \bigm| \mathcal{I}(\la) \supseteq \mathcal{I}, \ 
   \mathcal{K}(\la) \supseteq \mathcal{K} \right\}, \\
  V\left(\mathcal{I},\mathcal{K}\right)
   &:= \left\{\la \in V_d \bigm| \mathcal{I}(\la)=\mathcal{I}, \ 
   \mathcal{K}(\la)=\mathcal{K}\right\}, \\
  V\left(\mathcal{I},*\right)
   &:= \left\{\la \in V_d \bigm| \mathcal{I}(\la)=\mathcal{I}\right\}, \\
  V\left(*,\mathcal{K}\right)
   &:= \left\{\la \in V_d \bigm| \mathcal{K}(\la)=\mathcal{K}\right\}
 \end{align*}
 and
 $\mathcal{PV}\left(\mathcal{I},*\right)
  := P\left(V\left(\mathcal{I},*\right)\right)$.
 Remember that $\widetilde{V}\left(\mathcal{I},\mathcal{K}\right)
 = \textit{pr}\left(V\left(\mathcal{I},\mathcal{K}\right)\right)$,\\
 $\widetilde{V}\left(\mathcal{I},*\right)
 = \textit{pr}\left(V\left(\mathcal{I},*\right)\right)$ and
 $\widetilde{V}\left(*,\mathcal{K}\right)
 = \textit{pr}\left(V\left(*,\mathcal{K}\right)\right)$ hold
 by the definition in Main Theorem~\ref{mthm.2}.
 Note that $V\left(\mathcal{I},\mathcal{K}\right)$ is a Zariski open subset
 of $\overline{V(\mathcal{I},\mathcal{K})}$.

 First,
 we show the assertion~(\ref{en.mt3}) in Main Theorem~\ref{mthm.1}.
 Let $\la_0, \la'$ be elements of $V_d$
 with $\mathcal{I}(\la_0)\subseteq\mathcal{I}(\la')$
 and $\mathcal{K}(\la_0)\subseteq\mathcal{K}(\la')$.
 Then we have
 $\la' \in \overline{V(\mathcal{I}(\la_0),\mathcal{K}(\la_0))}$
 and $\mathfrak{S}\left(\mathcal{K}(\la_0)\right)\subseteq
 \mathfrak{S}\left(\mathcal{K}(\la')\right)$.
 By lemma~\ref{lm.4.1} and Implicit function theorem,
 the second projection map $\widetilde{\Phi}'_d$
 is locally homeomorphic,
 which implies that the map $\Phi'_d$ is also a local homeomorphism.
 We put $(\Phi'_d)^{-1}(\la') = \{\z(1),\ldots,\z(s_d(\la'))\}$.
 Then there exist an open neighborhood $U$ of $\la'$
 in $\overline{V(\mathcal{I}(\la_0),\mathcal{K}(\la_0))}$ and
 holomorphic sections
 $\tau_j: U \to \mathcal{Y}_d$ for $1\le j\le s_d(\la')$ such that
 $\Phi'_d\circ\tau_j={\it id}_U$ and $\tau_j(\la')=\z(j)$.
 Moreover the action of $\mathfrak{S}\left(\mathcal{K}(\la_0)\right)$ on
 $(\Phi'_d)^{-1}(\la')$ is naturally extended to
 the action of $\mathfrak{S}\left(\mathcal{K}(\la_0)\right)$ on
 $\left\{\tau_j(\la) \bigm| 1\le j\le s_d(\la') \right\}$ for any $\la \in U$.
 Hence we have $\#\left(\Phi_d^{-1}(\bar{\la}_0)\right)\ge
 \#\left(\Phi_d^{-1}(\bar{\la'})\right)$,
 which completes the proof of the assertion~(\ref{en.mt3})
 in Main Theorem~\ref{mthm.1}.

 Let us prove next
 the assertion~(\ref{en.mthm2.2}) in Main Theorem~\ref{mthm.2}.
 Since the map $\Phi'_d$ is locally homeomorphic and
 since the map
 $\textit{pr}|_{V\left(*,\mathcal{K}\right)} :
 V\left(*,\mathcal{K}\right) \to \widetilde{V}\left(*,\mathcal{K}\right)$
 is an unbranched covering
 for each $\mathcal{K} \in \left\{\mathcal{K}(\la)\bigm| \la \in V_d\right\}$,
 the map
 $
 \Phi_d|_{\Phi_d^{-1}(\widetilde{V}\left(*,\mathcal{K}\right))} :
 \Phi_d^{-1}\bigl(\widetilde{V}\left(*,\mathcal{K}\right)\bigr) \to
 \widetilde{V}\left(*,\mathcal{K}\right)
 $
 is a local homeomorphism, which verifies
 the assertion~(\ref{en.mthm2.2.2}) in Main Theorem~\ref{mthm.2}.
 For each $\mathcal{I} \in \left\{\mathcal{I}(\la)\bigm| \la \in V_d\right\}$,
 the cardinality of $(\widetilde{\Phi}'_d)^{-1}(m)$
 does not depend
 on the choice of $m \in \mathcal{PV}\left(\mathcal{I},*\right)$,
 which assures that the map
 $(\widetilde{\Phi}'_d)^{-1}\left(\mathcal{PV}\left(\mathcal{I},*\right)\right)
   \stackrel{\widetilde{\Phi}'_d}{\to} \mathcal{PV}\left(\mathcal{I},*\right)$
 is an unbranched covering.
 Hence the map
 $(\Phi'_d)^{-1}(V\left(\mathcal{I},*\right)) \stackrel{\Phi'_d}{\to}
  V\left(\mathcal{I},*\right)$
 is also an unbranched covering.
 Therefore since the map
 $V\left(\mathcal{I},*\right) \stackrel{\textit{pr}}{\to}
  \widetilde{V}\left(\mathcal{I},*\right)$
 is proper,
 the map
 $\Phi_d^{-1}\bigl(\widetilde{V}\left(\mathcal{I},*\right)\bigr)
  \stackrel{\Phi_d}{\to} \widetilde{V}\left(\mathcal{I},*\right)$
 is also proper, which verifies
 the assertion~(\ref{en.mthm2.2.1}) in Main Theorem~\ref{mthm.2}.
 The assertions~(\ref{en.mthm2.2.1}) and~(\ref{en.mthm2.2.2})
 imply the assertion~(\ref{en.mthm2.2.3});
 thus we have completed
 the proof of the assertion~(\ref{en.mthm2.2}) in Main Theorem~\ref{mthm.2}.

 Finally,
 we prove the assertion~(\ref{en.mthm2.1}) in Main Theorem~\ref{mthm.2}.
 In the following, we consider $V_d$ as an open dense subset
 of the vector space
 $\mb{C}^{d-1} = \left\{ (m_1,\ldots,m_d) \in \mb{C}^d \ \left| \ 
    \sum_{i=1}^d m_i = 0
   \right.\right\}$
 with the standard inner product.
 We take $\la \in V_d$, and put $\mathcal{I}(\la)=:\mathcal{I}$ and
 $\mathcal{K}(\la)=:\mathcal{K}$, which are fixed in the rest of 
 the proof.
 We denote by $H(\la)$ the orthogonal complement of the linear subspace 
 spanned by $V(\mathcal{I},\mathcal{K})$ in $\mb{C}^{d-1}$.
 Then the space $H(\la)$ is invariant
 under the action of $\mathfrak{S}(\mathcal{K}(\la))$.
 Hence we can take
 an arbitrarily small open neighborhood $H_{\epsilon}(\la)$ of $0$ in $H(\la)$
 which is invariant under the action of $\mathfrak{S}(\mathcal{K}(\la))$.
 Moreover we denote by $U(\la)$ a sufficiently small open neighborhood of $\la$
 in $V(\mathcal{I},\mathcal{K})$.
 Then the map
 $H_{\epsilon}(\la) \times U(\la) \ni (h,m) \mapsto h+m \in V_d$
 defines a local coordinate system around $\la$ in $V_d$.
 Hereafter, we identify $(h,m)\in H_{\epsilon}(\la)\times U(\la)$
 with $h+m \in V_d$.

 Since $H_{\epsilon}(\la)$ and $U(\la)$ are sufficiently small,
 we have $\mathcal{I}(h,m) \subseteq \mathcal{I}(\la)$ and
 $\mathcal{K}(h,m) \subseteq \mathcal{K}(\la)$
 for any $(h,m) \in H_{\epsilon}(\la) \times U(\la)$.
 Moreover 
 $\mathcal{I}(h,m)$ and $\mathcal{K}(h,m)$ do not depend on the choice
 of $m\in U(\la)$.
 Hence, for each $h \in H_{\epsilon}(\la)$ and
 for each connected component $Y$
 of $(\Phi'_d)^{-1}\left(\{h\} \times U(\la) \right)$,
 the map $\Phi'_d|_Y : Y \to \{h\} \times U(\la)$ is a homeomorphism.
 Therefore we have the natural isomorphism
 $(\Phi'_d)^{-1}\left(H_{\epsilon}(\la) \times U(\la)\right) \to
  (\Phi'_d)^{-1}\left(H_{\epsilon}(\la) \times \{\la\}\right) \times U(\la)$
 which commutes with the projection maps
 onto $H_{\epsilon}(\la) \times U(\la)$.

 For each $m \in U(\la)$, the space $H_{\epsilon}(\la) \times \{m\}$ is 
 invariant under the action of $\mathfrak{S}(\mathcal{K}(\la))$
 with a fixed point $(0,m)$.
 Moreover we have the natural isomorphism
 $\left(H_{\epsilon}(\la)/\mathfrak{S}(\mathcal{K}(\la))\right) \times U(\la)
  \cong
  \left(H_{\epsilon}(\la) \times U(\la)\right)/\mathfrak{S}(\mathcal{K}(\la))
  \cong \textit{pr}\left(H_{\epsilon}(\la) \times U(\la)\right)$.
 Hence $(\Phi'_d)^{-1}\left(H_{\epsilon}(\la) \times U(\la)\right)$
 is also invariant under the action of $\mathfrak{S}(\mathcal{K}(\la))$,
 and its action commutes with the isomorphism
 $(\Phi'_d)^{-1}\left(H_{\epsilon}(\la) \times U(\la)\right) \to
  (\Phi'_d)^{-1}\left(H_{\epsilon}(\la) \times \{\la\}\right) \times U(\la)$.
 Therefore we have the isomorphism
 \[
  \Phi_d^{-1}\left(\textit{pr}\left(H_{\epsilon}(\la) \times
   U(\la)\right)\right) \cong
  \Phi_d^{-1}\left(\textit{pr}\left(H_{\epsilon}(\la) \times
   \{\la\}\right)\right) \times U(\la)
 \]
 which commutes with the projection maps onto
 $\textit{pr}\left(H_{\epsilon}(\la) \times U(\la)\right)$.
 Hence
 for each $\la \in V\left(\mathcal{I},\mathcal{K}\right)$,
 \[
  \left\{\la' \in V\left(\mathcal{I},\mathcal{K}\right) \ \left| \
   \begin{matrix}
    \text{the pair $\la, \la'$ satisfies the condition} \\
    \text{in the assertion~(\ref{en.mthm2.1}) in Main Theorem~\ref{mthm.2}}
   \end{matrix}
    \right.\right\}
 \]
 is an open subset of $V\left(\mathcal{I},\mathcal{K}\right)$ containing $\la$.
 Since $V\left(\mathcal{I},\mathcal{K}\right)$ is connected,
 the assertion~(\ref{en.mthm2.1}) in Main Theorem~\ref{mthm.2} holds.
\end{proof}

\begin{proposition}\label{pr.mt6}
 The assertion~(\ref{en.mt6}) in Main Theorem~\ref{mthm.1} holds.
\end{proposition}
\begin{proof}
 The set $\Phi_d^{-1}(\bar{\la})$ is empty if and only if
 the set $S_d(\la)$ is empty by Proposition~\ref{pr.2.3}.
 On the other hand, the cardinality $\#\left(S_d(\la)\right)$
 is completely determined and is computed by 
 $\mathfrak{I}(\la)$.
 Hence to show the assertion~(\ref{en.mt6}) in Main Theorem~\ref{mthm.1},
 we only need to check all the possible cases of 
 $\mathfrak{I}(\la)$.
 However this may be hard for $d=6$ or $7$, and we shall relieve it a little.

 By a similar argument to the proof of the assertion~(\ref{en.mt3}) in Main Theorem~\ref{mthm.1},
 we can verify that for $\la, \la' \in V_d$, the inequality $\#\left(S_d(\la)\right) \ge \#\left(S_d(\la')\right)$ holds
 if $\mathfrak{I}(\la) \subseteq \mathfrak{I}(\la')$.
 Hence putting
 \[
     \widetilde{\mathfrak{I}}_d
     := \left\{ \mathfrak{I}(\la)  \ \left| \ 
	\begin{matrix}
	  \la \in V_d \text{ does not satisfy the assumption}\\
	  \text{in the assertion~(\ref{en.mt5}) in Main Theorem~\ref{mthm.1}}
	\end{matrix}
     \right.\right\},
 \]
 we only need to show the inequality $\#\left(S_d(\la)\right) > 0$ for any
 {\it maximal} $\mathfrak{I}(\la)$ in $\widetilde{\mathfrak{I}}_d$.
 On the other hand, $\mathfrak{S}_d$ naturally acts on $\widetilde{\mathfrak{I}}_d$
 by $\sigma\cdot\mathfrak{I}(\la) := \mathfrak{I}(\sigma\cdot\la)$,
 and the cardinality $\#\left(S_d(\la)\right)$ is determined only by the equivalence class of $\mathfrak{I}(\la)$.
 Moreover the inclusion relation in $\widetilde{\mathfrak{I}}_d$ naturally induces 
 the partial order in $\widetilde{\mathfrak{I}}_d / \mathfrak{S}_d$.
 Hence it suffices to show $\#\left(S_d(\la)\right) > 0$ for any
 maximal $\overline{\mathfrak{I}(\la)}$ in $\widetilde{\mathfrak{I}}_d / \mathfrak{S}_d$.
 In the following, we shall consider the cases of $d=4, 5, 6$ and $7$ individually.

 In the case $d=4$, the family $\widetilde{\mathfrak{I}}_4 / \mathfrak{S}_4$ consists of 
 the equivalence class of the empty set and that of $\bigl\{ \{ \{1,2\}, \{3,4\} \} \bigr\}$.
 Hence the unique maximal element of $\widetilde{\mathfrak{I}}_4 / \mathfrak{S}_4$ is
 represented by $\bigl\{ \{ \{1,2\}, \{3,4\} \} \bigr\}$, 
 which is obtained from $\la=(\la_1,\dots\la_4) \in V_4$ with
 $
   \left(1-\la_1\right)^{-1} : \cdots : \left(1-\la_4\right)^{-1} = 1:-1:a:-a,
 $
 where $a \ne 0, \pm 1$. For such $\la \in V_4$, we have $\#\left(S_4(\la)\right) = 1$.

 Let us consider the case $d=5$ next.
 If $\mathfrak{I}(\la) \in \widetilde{\mathfrak{I}}_5$ have only one element, 
 then $\mathfrak{I}(\la)$ lies in the equivalence class of $\bigl\{ \{ \{1,2\}, \{3,4,5\} \} \bigr\}$.
 If $\mathfrak{I}(\la) \in \widetilde{\mathfrak{I}}_5$ have exactly two elements, 
 then it must be in the equivalence class of
 $\bigl\{ \{ \{1,2\}, \{3,4,5\} \},\ \{ \{1,3\}, \{2,4,5\} \} \bigr\}$
 since any $\left(1-\la_i\right)^{-1}$ is not equal to $0$.
 By a similar argument, if $\mathfrak{I}(\la)$ have at least three elements, then it must be
 in the equivalence class of
 $\bigl\{ \{ \{1,2\}, \{3,4,5\} \},\ \{ \{1,3\}, \{2,4,5\} \},\ \{ \{1,4\}, \{2,3,5\} \} \bigr\}$.
 However this is obtained only from $\la=(\la_1,\dots\la_5) \in V_5$ with
 $
   \left(1-\la_1\right)^{-1} : \cdots : \left(1-\la_5\right)^{-1} = -1:1:1:1:-2;
 $
 hence it is not in $\widetilde{\mathfrak{I}}_5$ by definition. 
 Therefore the maximal element of $\widetilde{\mathfrak{I}}_5 / \mathfrak{S}_5$ is also unique and is represented by
 $\bigl\{ \{ \{1,2\}, \{3,4,5\} \},\ \{ \{1,3\}, \{2,4,5\} \} \bigr\}$, which is obtained from $\la=(\la_1,\dots\la_5) \in V_5$ with
 \[
   \left(1-\la_1\right)^{-1} : \cdots : \left(1-\la_5\right)^{-1} = -1:1:1:a:(-a-1),
 \]
 where $a \ne 0, \pm1, -2$.
 For such $\la \in V_5$, we have $\#\left(S_5(\la)\right) = 2$.

 In the case $d=6$ or $7$, we only give the list of $\la=(\la_1,\dots,\la_d)\in V_d$ 
 which generate all the maximal $\overline{\mathfrak{I}(\la)}$ in $\widetilde{\mathfrak{I}}_d / \mathfrak{S}_d$.

 In the case $d=6$, there are six maximal elements in $\widetilde{\mathfrak{I}}_6 / \mathfrak{S}_6$,
 and they are obtained from $\la=(\la_1,\dots\la_6) \in V_6$ such that 
 $\left(1-\la_1\right)^{-1} : \cdots : \left(1-\la_6\right)^{-1}$ is equal to either of the followings:
 \begin{align*}
  &1:1:-1:-1:a:-a, \text{ where } a \ne 0, \pm1, \pm2,\\
  &1:-1:a:-a:(a+1):-(a+1), \text{ where } a \ne 0, -1/2, \pm1, -2,\\
  &1:1:1:-1:a:-(a+2), \text{ where } a \ne 0, \pm1, -2, -3,\\
  &1:1:a:a:-(a+1):-(a+1), \text{ where } a \ne 0, -1/2, \pm1, -2,\\
  &1:1:1:2:-2:-3 \qquad \text{or} \qquad 1:1:3:-1:-2:-2.
 \end{align*}

 In the case $d=7$, there are 27 maximal elements in $\widetilde{\mathfrak{I}}_7 / \mathfrak{S}_7$,
 and they are obtained from $\la=(\la_1,\dots\la_7) \in V_7$ such that 
 $\left(1-\la_1\right)^{-1} : \cdots : \left(1-\la_7\right)^{-1}$ is equal to either of the followings:
 \begin{align*}
  &1:1:1:-1:-1:a:-(a+1), \text{ where } a \ne 0, \pm1, \pm2, -3,\\
  &1:1:1:1:-1:a:-(a+3), \text{ where } a \ne 0, \pm1, -2, -3, -4,\\
 \end{align*}
\vspace{-35pt}
 \begin{align*}
  &1:1:2:2:-1:-1:-4, & &1:1:1:3:-1:-2:-3, & &1:1:2:2:-1:-2:-3,\\
  &1:1:2:3:-1:-2:-4, & &1:1:2:3:-2:-2:-3, & &1:2:2:3:-1:-3:-4,\\
  &1:1:1:4:-1:-3:-3, & &1:2:2:2:-1:-1:-5, & &1:1:3:3:-1:-2:-5,\\
  &1:1:1:2:2:-2:-5,  & &1:1:2:5:-2:-3:-4, & &1:1:1:2:3:-3:-5,\\
  &1:1:2:4:-2:-3:-3, & &1:2:2:4:-1:-3:-5, & &1:1:1:1:3:-3:-4,\\
  &1:1:1:1:2:-2:-4,  & &1:2:2:2:-1:-3:-3, & &1:1:1:1:2:-3:-3,\\
  &1:1:1:2:2:-3:-4,  & &1:1:1:3:-2:-2:-2, & &1:1:1:4:-2:-2:-3,\\
  &1:1:1:5:-2:-3:-3, & &1:1:3:3:-2:-2:-4, & &1:1:3:4:-2:-2:-5 \quad \text{or}\\
  &1:1:2:2:3:-4:-5.
 \end{align*}

 We can verify the inequality $\#\left(S_d(\la)\right) > 0$ for every $\la \in V_d$ listed above, 
 which completes the proof of the proposition.
\end{proof}

To summarize the above mentioned,
we have completed the proof of the main theorems.

\end{document}